\documentclass[11pt,oneside]{article}
\RequirePackage{color}
\RequirePackage[usenames,dvipsnames]{xcolor}
\usepackage{amssymb,amsmath,amsthm}
\usepackage{url}
\usepackage[english]{babel}
\usepackage{pstricks,pst-node}
\usepackage{hyperref}
\usepackage{wasysym}% for \diameter
\renewcommand{\le}{\leqslant}
\renewcommand{\ge}{\geqslant}
\newcommand{\bad}{\mathbf{Bad}}

\newcommand{\rad}{\mathrm{rad}}
\renewcommand{\cent}{\mathrm{cent}}
\newcommand{\diam}{\mathrm{diam}}

\newcommand{\AAA}{\mathcal{A}}
\newcommand{\BBB}{\mathcal{B}}
\newcommand{\PPP}{\mathcal{P}}
\newcommand{\DDD}{\mathcal{D}}
\newcommand{\RR}{\mathbb{R}}

\newcommand{\ZZ}{\mathbb{Z}}
\newcommand{\QQ}{\mathbb{Q}}

\newcommand{\NN}{\mathbb{N}}
\newcommand{\NNN}{\mathcal N}

\newcommand{\CCC}{\mathcal{C}}
\newcommand{\KKK}{\mathcal{K}}
\newcommand{\calk}{\mathcal{K}}

\newcommand{\R}{\mathbb R}
\newcommand{\WWW}{\mathcal{W}}
\newcommand{\SSS}{\mathcal{S}}

\newcommand{\vr}{\mathbf{r}}

\newcommand{\vd}{\mathbf{d}}

\newtheorem{theorem}{Theorem}[section]

\newtheorem{lemma}[theorem]{Lemma}

\newtheorem{proposition}[theorem]{Proposition}%[section]

\newtheorem{corol}[theorem]{Corollary}

\theoremstyle{remark}
\newtheorem{rem}[theorem]{Remark}
\theoremstyle{definition}
\newtheorem{example}[theorem]{Example}
\newtheorem{defn}[theorem]{Definition}

\newcommand{\HD}{{\dim_H}}
\newcommand{\butnot}{\setminus}
\newcommand{\w}{\widetilde}

\global\long\def\sub{\subseteq}
\global\long\def\subp{\supseteq}
\global\long\def\sep{\;:\;}

\global\long\def\bbr{\mathbb{R}}

\global\long\def\bbn{\mathbb{N}}

\global\long\def\calh{\mathcal{H}}

\global\long\def\calc{\mathcal{C}}

\global\long\def\calb{\mathcal{B}}

\global\long\def\cala{\mathcal{A}}

\global\long\def\br{\mathbf{r}}

\global\long\def\eps{\varepsilon}

\renewcommand{\epsilon}{\eps}
\renewcommand{\subset}{\sub}

\newcommand{\OC}[2]{(#1,#2]}% half-closed half-open interval
\newcommand{\smallemptyset}{{\varnothing}}
\renewcommand{\emptyset}{{\diameter}}

\usepackage{tikz}
\usetikzlibrary{backgrounds}
\tikzstyle{densely dotted}=[dash pattern=on \pgflinewidth off 1pt]
\usepackage{tikz-cd}
\usetikzlibrary{decorations.markings}
\tikzset{negated/.style={
		decoration={markings,
			mark= at position 0.5 with {
				\node[transform shape] (tempnode) {$\backslash$};}},postaction={decorate}}}
\usepackage{chngcntr}
\counterwithout{figure}{section}

\newif\ifdraft
\drafttrue

\newif\ifcolorcomments
\colorcommentstrue

\newcommand{\allowcomments}[4]{
\newcommand{#1}[1]{\ifdraft{\ifcolorcomments{ \textcolor{#4}{##1 --#3}}\else{\textsl{ ##1 \ --#3}}\fi}\else{}\fi}
}

\allowcomments{\comdzmitry}{DB}{Dzmitry}{blue}
\allowcomments{\comstephen}{SH}{Stephen}{orange}
\allowcomments{\comerez}{EN}{Erez}{red}
\allowcomments{\comdavid}{DS}{David}{Green}

\begin{document}

%\title{Variants of Schmidt's game and Cantor-winning sets}
\title{Schmidt games and Cantor winning sets}

\author{
 Dzmitry Badziahin, Stephen Harrap, Erez Nesharim, David Simmons}
%\footnote{Research supported by EPSRC  Grant EP/L005204/1} }
\date{}

\maketitle

%\tableofcontents

\begin{abstract}
Schmidt games and the Cantor winning property give alternative notions of largeness, similar to the more standard notions of measure and category. Being intuitive, flexible, and applicable to recent research made them an active object of study. We survey the definitions of the most common variants and connections between them. A new game called the Cantor game is invented and helps with presenting a unifying framework. We prove surprising new results such as the coincidence of absolute winning and $1$ Cantor winning in metric spaces, and the fact that $1/2$ winning implies absolute winning for subsets of $\bbr$, and we suggest a prototypical example of a Cantor winning set to show the ubiquity of such sets in metric number theory and ergodic theory. %\comdavid{Fixed many grammatical errors here}
\end{abstract}

\begin{center}
\end{center}

\section{Introduction and Results}\label{intro}

When attacking various
%difficult
problems in
%the field of
Diophantine approximation, the application of certain games has proven extremely fruitful.
%in recent years. In particular,
Many authors have appealed to Schmidt's celebrated
%\textit{$(\alpha, \beta)$
game \cite{Schmidt1} (e.g. \cite{BBFKW, Dani2, EinsiedlerTseng, Farm, Fishman2, {KleinbockWeiss2}, Moshchevitin, Tseng2}), %McMullen's
the \textit{absolute game} \cite{McMullen_absolute_winning} (e.g. \cite{BFKRW, CCM, FSU4, HKSW, MayedaMerrill}),
%and \textit{$k$-dimensional absolute games} \cite{BFKRW} (the most utilised of which is the \textit{hyperplane absolute game})
%have also seen widespread application
and the \textit{potential game} \cite[Appendix C]{FSU4} (e.g. \cite{AGGT, AGK, GuanYu, NesharimSimmons, Wu, Wu2}).
%and in some sense fall ``in between'' the former two concepts.
%By definition, each of the above games has a particular associated collection of \textit{winning} sets - in Schmidt's original game and in the potential game each winning set depends on some fixed parameter $\alpha>0$ or $c>0$ respectively, whereas in the ``absolute'' games of \cite{McMullen_absolute_winning} and \cite{BFKRW} they are independent of any parameter.
It is the amenable properties of the winning sets associated with games of this type that make the games such attractive tools. These properties most commonly  include (here we use winning in a loose sense to mean winning for either the $\alpha$ game, the $c$ potential game, or the absolute game):
\begin{itemize}
    \item[(W1)] A winning set is dense within the ambient space. In $\RR^N$ winning sets are thick; i.e., their intersection with any open set has Hausdorff dimension $N$.
%If the game is being played on subsets of $\RR^N$, or more generally on subsets of a regular complete metric space, then a winning set additionally has ``full'' Hausdorff dimension.
    \item[(W2)] The intersection of a countable collection of winning sets is %again
 winning.
%In the case of winning and potential winning sets, the associated parameter is preserved.
    \item[(W3)] The image of a winning set under a certain type of mapping (depending on the game in question) is %also a
winning.% set.
 %(for example,  bi-Lipschitz homomorphisms, quasymmetric homomorphisms or $C^1$ diffeomorphisms).
\end{itemize}

%\comdavid{Is ``winning set'' here supposed to mean a set winning for Schmidt's game, or for any of the games? This should be made clearer. Also, the intersection of a countable collection of Schmidt-winning sets is not necessarily Schmidt-winning.}

At the same time, many related problems in Diophantine approximation have been solved via the method of constructing certain tree-like
%(or ``Cantor-type'')
(or ``Cantor-type'') structures inside a Diophantine set of interest
\cite{An2, An1, BPV, BadziahinVelani, Beresnevich_BA, Davenport,
PollingtonVelani,Yang}.
%For example,
One of the key ingredients in the proof of a famous Schmidt
conjecture in~\cite{BPV} was the construction of a certain generalised Cantor
set in $\RR$. The main ideas were formalised in the subsequent
paper~\cite{BadziahinVelani} and then in~\cite{BadziahinHarrap} the
theory of \textit{generalised Cantor sets} and \textit{Cantor
winning sets} was finalised in the general setting of complete
metric spaces and a host of further applications to problems in
Diophantine approximation was given.
%Within the proofs of~\cite{BPV} \textit{generalised Cantor sets} and
%\textit{Cantor winning sets} in $\RR$ were introduced and utilized
%to settle a famous conjecture of Schmidt. The main ideas were
%formalised in the subsequent paper~\cite{BadziahinVelani} and
%%were very recently
%extended to the more general setting of complete metric spaces in \cite{BadziahinHarrap},
%%Furthermore, a powerful class of related sets was introduced in \cite{BadziahinHarrap},
%where a host of further applications to problems in Diophantine approximation was given. By design,
%%this class of so called \textit{Cantor winning sets}
The family of $\eps$ Cantor winning sets satisfies
%the desirable
properties (W1)-(W3). %\comdavid{A set cannot satisfy properties (W1)-(W3), they are attributes of a class of sets.}
%usually attributed to winning sets.
%However, the realisation of these properties now relied on the construction of ``ubiquitous'' generalised Cantor sets rather than the playing of
%some topological
%Schmidt games.

It is therefore very natural to investigate how these two
%very
different approaches are related, and that is precisely the intention of this paper. %\comstephen{If anyone has any more recent references worth adding to any of the above lists, please do :)}
%\comstephen{Any idea why our page numbers are so low/margins topheavy? Is it just my compiler?} \comdavid{I changed {\tt\textbackslash topmargin=-1cm} to {\tt\textbackslash topmargin=-2cm}, which looks a bit better. If you want to fiddle with it you can also consider adjusting {\tt\textbackslash textheight} which changes the bottom margin without changing the top margin.}

\subsection{Known connections}\label{oldresults}
%\comstephen{It seems odd to me that we had a section entitled "New and old connections" comprising of the old results, and then a section entitled "New results" comprising the new. So better to change this to "Known connections"}

In \S\ref{definitions} we provide a complete set of definitions for the unfamiliar reader, along with further commentary relating to the results in this and the succeeding subsection. %We broadly follow the same notations as used in \cite{BadziahinHarrap}.

The main results of this paper can be summarised in Figure \ref{fig:flow}. In Figure \ref{fig:flow} the solid arrows represent known implications, the dashed arrows represent our new results, and a crossed arrow stands for where an implication does not hold in general.
\begin{figure}

\begin{center}
\label{fig:flow}
%\comdavid{I commented out this figure because it doesn't compile on my computer}
\begin{tikzcd}[background rectangle/.style={fill=yellow!30}, show background rectangle]
        & \text{$\frac12$ winning}
        \arrow[dd, dashed, "\, \mathbb R"] \arrow[r]
        & \text{$\alpha$ winning} \arrow[dd, leftrightarrow, dashed, negated]\\
        \text{Absolute winning} \arrow[ur, bend left] \arrow[dr, bend left]  & & \\
         & \text{$1$ Cantor winning}\arrow[r] \arrow[ul, dashed, bend left] & \text{$\varepsilon$ Cantor winning} \arrow[dl, "t", leftrightarrow, dashed, to path=|- (\tikztotarget)]\\
    \text{$0$ Potential winning} \arrow[r] \arrow[uu, leftrightarrow] & \text{$c$ Potential winning} &
    \end{tikzcd}

\end{center}
\caption{Main results}
\label{fig:flow}
\end{figure}
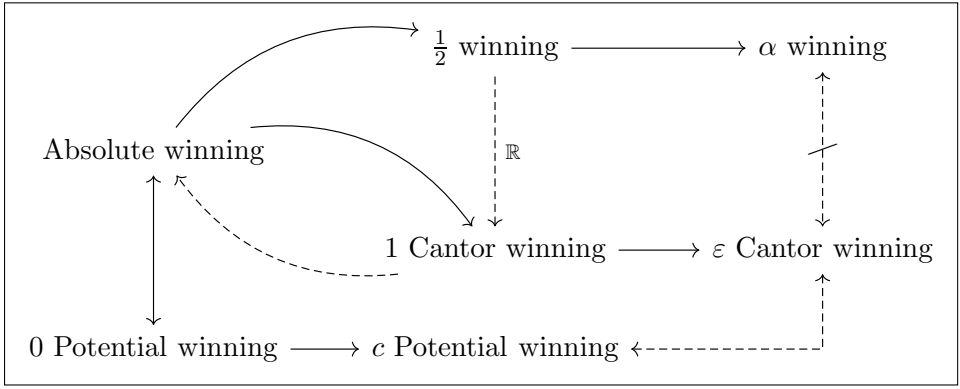
%$$\comerez{ Insert Diagram(s) here} $$
%\comstephen{ I think journals usually expect figures at the top or bottom of pages rather than in-line. This is currently the worst of the diagrams I've done - I will work on it more}

 % for some reason \ref{fig:flow} gives 1.1 !
%Erez: looks like labeling after the caption fixed this problem (according to google, labeling before the caption references the subsection)

The following connections are already known:

%In particular, the following connections are already known:
%between Schmidt
%topological
%games and Cantor sets are well known. % (What is below is in terms of the Euclidean metric - need to nuance between that and our results later.)
%\begin{proposition}[McMullen \cite{McMullen_absolute_winning}]  \label{mcmullen}
\begin{proposition}[\cite{McMullen_absolute_winning}]  \label{mcmullen}
%   If a set $E \sub \RR^N$ is absolute winning then $E$ is $\alpha$ winning for every $\alpha \in (0, 1/2)$.
    If $E \sub \RR^N$ is absolute winning then $E$ is $\alpha$ winning for every $\alpha \in (0, 1/2)$.
\end{proposition}

One can actually slightly improve on this via the following folklore argument (see \cite[Proposition 2.1]{GuanYu} for a related observation):
\begin{proposition}\label{lemmacontinuity}
If $E$ is winning and $\alpha_0: = \sup\left\{\alpha:\textup{$E$ is
$\alpha$ winning}\right\}$, then $E$ is $\alpha_0$ winning.
\end{proposition}
For the sake of completeness we will prove this proposition in Section~\ref{definitions}. %\label{definitions}.
%The following provides a powerful tool for proving that a set is absolute winning:
Potential winning is related to absolute winning by:
%\begin{proposition}[FSU \cite{FSU4}, Theorem C.8]\label{thm:potential_implies_absolute}
\begin{proposition}[{\cite[Theorem C.8]{FSU4}}]\label{thm:potential_implies_absolute}
%   If a set $E \sub \RR^N$ is $c$ potential winning for every $c>0$ then $E$ is absolute winning.
    Let $X$ be a complete doubling\footnote{See Definition~\ref{def:doubling} for the definition of doubling metric spaces.} metric space. If $E \sub X$ is $c$ potential winning for every $c>0$ (i.e., $0$ potential winning), then $E$ is absolute winning. %\comdzmitry{I do not see the equivalence between absolute winnign and 0-potential winning. Possibly, Erezes changes were lost somewhere.}
    %\comstephen{Isn't the statement in [26] an if and only if?}
    %In particular, if $E$ is absolute winning then $E$ is  $1$-Cantor-winning
\end{proposition}

Meanwhile,
%$k$-dimensional
absolute winning
%sets are
is
%known to be
related to Cantor winning
%sets
in the following way:
%\begin{proposition}[BadziahinHarrap \cite{BadziahinHarrap}, Theorem 12]\label{BHabs}
%   If a set $E \sub \RR^N$ is %$k$-dimensional
%    absolute winning
%    %for some integer $k \in [0, N-1]$
%    then $E$ is $1$ Cantor winning.
%    %In particular, if $E$ is absolute winning then $E$ is  $1$-Cantor-winning
%\end{proposition}

\begin{proposition}[{\cite[Theorem 12]{BadziahinHarrap}}]\label{BHabs}
    Absolute winning subsets of $\RR^N$ are $1$ Cantor winning.
\end{proposition}

%(Need to decide whether to provide above results in full generality - metric spaces - possibly finding the weakest conditions for Prop \ref{McImprove} to hold (Hilbert probably). Also possible extend Prop \ref{McImprove} to cover Theorem \ref{BFKRW})

%(Discussion of historical usage of the above concept might be nice, and this or the end of the previous section would probably be the place - more details of Schmidt Conjecture/An/Erez-Simmons etc - but it depends how much we want to sidetrack as this is already quite a long intro)

\subsection{New results}\label{newresults}

%Our first result  serves as a justification for the term ``Cantor winning''. This is done by introducing a new game called the ``Cantor game''. To prove that its associated winning sets are exactly the Cantor winning sets we use ideas that are related to the potential game (cf. \S\ref{sec:cantor_game}).
In some sense, the above statements suggest that the construction of Cantor winning sets mimics the playing of a topological game.  In \S\ref{sec:cantor_game} we introduce a new game called the ``Cantor game''. We  prove that its associated winning sets are exactly the Cantor winning sets (cf. Theorem \ref{thm:potentialCantor}); to do this we utilise ideas related to the potential game. This serves as a justification for the term ``Cantor winning''.

Next, we prove a
%striking
connection between Cantor winning and potential winning, and, in particular, provide a converse to Proposition~\ref{BHabs}:%Theorem~\ref{BHabs}:

\begin{theorem}\label{potentialequivR}
A set $E \sub \RR^N$ is $\eps$ Cantor winning if and only if it is
$N(1-\eps)$ potential winning. In particular, $E$ is Cantor winning if and only if $E$ is potential winning, and if $E\sub \RR^N$ is $1$ Cantor winning then it is absolute winning.
\end{theorem}

\begin{rem}\label{rem:barak}
It is proved in \cite[Theorem 5.2]{BadziahinHarrap} that hyperplane absolute winning sets (cf. \S\ref{sec:absgame} for the definition) in $\bbr^N$ are $\frac{1}{N}$ Cantor winning, however, the converse is not true. It is possible to develop a theory of generalised Cantor sets in which the removed sets come from neighborhoods of hyperplanes and then use the hyperplane potential game (cf. \cite[Appendix C]{FSU4} or \cite{NesharimSimmons} for the definition) to state a complete analogue of Theorem \ref{potentialequivR}. This will not be pursued in this note.
\end{rem}

The connection between winning sets for Schmidt's game and Cantor winning sets is more delicate, as demonstrated by the existence of the following counterexamples (which are given explicitly in \S\ref{counters}):
\begin{theorem}\label{theoremPWnotW}
There is a Cantor winning set in $\RR$ that is not winning.
\end{theorem}

\begin{theorem}\label{theoremWnotPW}
There is a winning set in $\RR$ that is not Cantor winning.
\end{theorem}

%\comerez{CAN WE PARAMETRISE AS IN FOR EVERY $\eps$ AND $\alpha$ THERE EXISTS ... }.

However,
%we
%can
%provide
a partial result connecting Cantor winning sets and winning sets when some restrictions are placed on the parameters
is possible:
% (we state and prove a related result in \S\ref{extras}):

%\begin{theorem}\label{potentialschmidt}
    %\label{theoremschmidtgamerelation}
    %Let $X$ be a complete $\delta$-regular metric space and let $S\sub X$ be an %$\eps$-Cantor-winning set. Then for all $c > \delta(1-\eps)$, there exists $\gamma > 0$ such %that for all $\alpha,\beta > 0$ such that $\alpha < \gamma(\alpha\beta)^{c/\delta}$, $S$ is %$(\alpha,\beta)$-winning.
%\end{theorem}

\begin{theorem}\label{potentialschmidtRN}
Let $E\sub \RR^N$ be an $\eps$ Cantor winning set. Then for every $c >
N(1-\eps)$, there exists $\gamma > 0$ such that for all $\alpha,\beta > 0$
with $\alpha < \gamma(\alpha\beta)^{c/N}$, the set $E$ is $(\alpha,\beta)$
winning.
\end{theorem}

%\comerez{ CAN $\gamma$ BE EXPLICIT IN THIS SPECIAL CASE?}.

Even though Cantor winning does not imply winning,
%\comerez{ CAN WE SAY MORE?},
the intersection of the corresponding sets has full Hausdorff dimension:% \comdavid{reworded since I wasn't sure what the previous version was trying to say}
%For the proof, see \S\ref{potentialschmidtproof}.
%\begin{corol}
%   \label{theoremPWWintersection}
%   In an Ahlfors regular space, the intersection of a winning set and a Cantor-winning set has full Hausdorff dimension.
%\end{corol}
\begin{theorem}\label{theoremPWWintersectionR}
    The intersection  in $\RR^N$ of a winning set for Schmidt's game and a Cantor winning set has Hausdorff dimension $N$.
\end{theorem}

That being said, we prove a
surprising
result regarding subsets of the real line:

\begin{theorem}\label{thmwinningmainN}
    If a set $E \sub \RR$ is $1/2$ winning then it is $1$ Cantor winning.
\end{theorem}

In view of Theorem~\ref{thmwinningmainN}, Propositions~\ref{mcmullen} \& \ref{lemmacontinuity}, and Theorem~\ref{potentialequivR}, we deduce the following
%remarkable
equivalence:

\begin{corol}\label{corolwinningmain}
    The following statements about a set $E \sub \RR$ are equivalent:
    \begin{itemize}
        \item[1.] $E$ is $1/2$ winning.

        \item[2.] $E$ is $1$ Cantor winning.

        \item[3.] $E$ is absolute winning.

        %   \item[4.] $E$ is absolute Cantor-winning.

        %Plus Strong/weak/very strong w.r.t. Schmidt's
    \end{itemize}
\end{corol}

In~\cite{Beresnevich_BA} the notion of \textit{Cantor rich} subsets of $\RR$ was
%recently
introduced. Cantor rich sets were shown to satisfy some of the desirable properties associated with %classical
winning sets, and have been applied since elsewhere \cite{Yang}. As an application of the results in this paper we can prove that Cantor rich
%ness 
and Cantor winning coincide.% for subsets of $\RR$.

\begin{theorem}\label{thm:cantor_rich}
A set $E\sub \RR$ is Cantor rich in an interval $B_0$ if and only if $E$ is Cantor winning in $B_0$.
\end{theorem}

%On a related note, a further known dimension property pertains to the intersection of such sets with ``nice'' fractals. It is well known that $k$-dimensional absolute winning sets this intersection is large \cite{BFKRW}:
Another way to measure how large a set is, is to consider its intersections with classes of tamed sets.
In fact, proving 
%such 
nonempty intersection with certain classes of sets
%served as part of the motivation that led to 
motivated the invention of the absolute game. For example, it is well known that absolute winning sets are large in the following dual sense \cite{BFKRW}:
\begin{itemize}
\item[(W$4$)] The intersection of an absolute winning set with any closed nonempty uniformly perfect set is nonempty.
%Furthermore, if $K$ is also Ahlfors regular then the intersection has full Hausdorff dimension in $K$.
%\comdavid{Do we really need to introduce new terminology here? Usually diffuse sets are called ``uniformly perfect''}
\end{itemize}

We demonstrate that a similar property holds for Cantor winning sets. In fact, Cantor winning Borel sets are characterised by an analogue of property (W4):

\begin{theorem}\label{thm:intersection_with_fractal_rn}
Let $E\sub \RR^N$. If $E$ is $\eps$ Cantor winning then $E\cap K\neq \emptyset$ for any $K \in \RR^N$ which is $N(1-\eps)$ Ahlfors regular. If $E$ is Borel then the converse is also true.
\end{theorem}

Moreover, the intersection of Cantor winning sets with Ahlfors regular sets is often large in terms of Hausdorff dimension (cf. Theorem \ref{propositionfullHD}). The characterisation in Theorem \ref{thm:intersection_with_fractal_rn} makes the notion of Cantor winning sets
%very
natural from a geometric point of view.

\begin{rem}\label{rem:barak2}
Continuing Remark \ref{rem:barak}, it is also possible to characterise hyperplane absolute winning sets as sets which have a nontrivial intersection with any closed \emph{hyperplane diffuse} fractal (cf. \cite[Definition 4.2]{BBFKW}). This characterisation is not discussed in this note.
\end{rem}

As a typical example of a Cantor winning set we suggest the following:

\begin{theorem}\label{thm:dolgopyat}
Let $X=\{0,1\}^\bbn$ with the metric $d(x,y) = 2^{-v(x,y)}$, where
$v(x,y)$ denotes the length of the longest common initial segment of $x$ and $y$. Let $T:X\to X$ be the left shift map. If
$K\sub X$ satisfies $\dim_H K < 1$ then the set
    \begin{equation}\label{eq:exceptionalSet}
        E=\left\{x\in X\sep \overline{\{T^i x\sep i\geq 0\}}\cap K=\emptyset \right\}
    \end{equation}
    is $1 - \dim_H K$ Cantor winning, or, equivalently, $\dim_H K$ potential winning.
\end{theorem}

%\begin{theorem}\label{thm:dolgopyat}
%Let $X=\bbr/\bbz$ be the one-dimensional torus with the standrad Euclidean metric and let $T:X\to X$ be the times two map. Assume $K\sub X$ with $\dim_H K<1$. Then $E=%\left\{x\in X\sep \overline{\{T^n x\sep n\in\bbn\}}\cap K=\emptyset \right\}$ is $\dim_H K$ potential winning.
%\end{theorem}

Here, $\dim_H$ stands for the Hausdorff dimension. The set $E$ in (\ref{eq:exceptionalSet}) was first considered by Dolgopyat~\cite{Dolgopyat},
who proved that $\dim_H E=1$. Unlike the property of having full Hausdorff dimension, the potential winning property passes automatically to subspaces (see Lemma \ref{lem:restriction}).
%By the properties of the potential game,
%the conclusion of Theorem~\ref{thm:dolgopyat} automatically passes to
%subsets of $X$ which are invariant under the map $T$ (cf. \S\ref{sec:potential}), and, in particular, to subshifts of finite type.
%\comdavid{I don't understand the preceding sentence}\comstephen{Neither do I. Probably best to just to label the relevant (W) property rather than quote the section, since by now we've already defined (W1-W4).}
%\comerez{Here I'm just referring to the fact that if $Y\sub X$ is a subspace that inherits the splitting structure, then if $K\sub X$ is $\eps$ Cantor winning then $K\cap Y$ is $\eps$ Cantor winning. Together with the fact that $Y$ is a subshift, $K\cap Y$ is defined in a similar fashion to $K$ as the set of all points in $Y$ whose their orbit closure misses $K$} \comdavid{As we discussed on Friday I think you mean ``if $K\sub X$ is $c$ potential winning then $K\cap Y$ is $c$ potential winning on $Y$''.}
%More precisely, if $Y\sub X$ is a subspace of $X$ and $E\sub X$ is $c$ potential winning, then $E\cap Y$ is $c$ potential winning on $Y$.
Hence, in particular, using the notation of Theorem \ref{thm:dolgopyat}, if $Y\sub X$ is a subshift of finite type and $K\sub X$ satisfies $\dim_H \left(K\cap Y\right) < \dim_H Y$, then Theorems \ref{thm:dolgopyat} and \ref{propositionfullHD}
%immediately
imply that $\dim_H \left(E\cap Y\right) = \dim_H Y$, which is exactly the content of \cite[Theorem 1]{Dolgopyat}.

The set of all points whose orbit closure misses a given
%large
subset is considered in various contexts, such as piecewise expanding maps of the interval and Anosov diffeomorphisms on compact surfaces \cite{Dolgopyat}, certain partially hyperbolic maps \cite{CamposGelfert2,Wu}, rational maps on the Riemann sphere \cite{CamposGelfert}, and $\beta$ shifts \cite{FarmPersson,FPS}. We suggest that Theorem \ref{thm:dolgopyat} might be generalisable to these contexts; however, we choose to present the proof in the above prototypical example. %, for simplicity and to avoid having to introduce more notation that is otherwise irrelevant for the rest of the paper.
See also \cite{BMPS} for a related problem where the rate at which the orbit is allowed to approach the subset $K$ depends on the geometrical properties of $K$. %geometry?

%In \S\ref{definitions} we discuss some further dimension properties of the games/constructions at hand, alongside all the required definitions of the above (and further) concepts.
%In addition,
%In \S\ref{counters} we investigate counterexamples to possible ways one might hope to extend the above theorems. The proofs of the results of this section can be found in~\S\ref{proofs}.

%\begin{rem}
%In~\cite{Beresnevich_BA} the similar notion of \textit{Cantor
%   rich sets in $\RR$} was recently introduced. Cantor rich sets also satisfy some of the desirable properties associated with classical winning sets. It would be interesting to compare this notion to those discussed above.
%\end{rem}

%(We may also wish to compare this notions with those discussed above.
%-include any results/discussion on this.) (Also explore relationships between sup and Euclidean games if necessary)

%\comerez{CONSIDER ADDING A DIAGRAM FOR EACH GAME}

\section{Background, notation, and definitions}\label{definitions}

We formally define the concepts introduced in \S\ref{intro}. Each of the
%topological
games we describe below is played between two players, Alice
 %(A)
and Bob,
%~(B)
in a complete metric space $(X,\vd)$.
A ball in $B\sub X$ is specified by its centre and radius, denoted by $\cent(B)$ and $\rad(B)$.
%For balls $B,B'\sub X$ the notation
%\begin{equation}\label{eq:containment1}
%B\sub B'
%\end{equation}
%will mean
%\begin{equation}\label{eq:containment2}
%\cen(B)\in B' \text{ and } \rad(B) + \vd(\cen(B),\cen(B'))\leq\rad(B').
%\end{equation}
%Note that in any metric space \eqref{eq:containment2} implies \eqref{eq:containment1}, while the converse is not necessarily true.
%In what follows $\NO$ will denote the set of nonnegative integers $\NN \cup \{0\}$.

%(Need to decide whether to talk about $\RR^N$ or a more general metric space - depends on what we manage to prove).

\subsection{Schmidt's game}

Given two real parameters $\alpha,\beta \in (0,1)$, \textit{Schmidt's $(\alpha, \beta)$ game} begins with Bob choosing an arbitrary closed ball $B_0\sub X$. Alice and Bob then take turns to choose a sequence of nested closed balls (Alice chooses balls $A_i$ and Bob chooses balls $B_i$) satisfying both

%\comerez{NOTE THAT I CHANGED THE INDICES SO THAT ALICE'S FIRST CHOICE IS $A_1$}
%\comstephen{I think this is now consistent in the paper - excepting, see the comment on the potential game later}

$$
B_0\subp A_1\subp B_1 \subp A_2\subp\cdots\subp B_i \subp A_{i+1} \subp B_{i+1}\subp \cdots
$$
and
\begin{equation}\label{scmcond}
\rad(A_{i+1}) = \alpha \cdot \rad(B_i) \textrm{  and  } \rad(B_{i+1}) = \beta \cdot \rad(A_{i+1}) \textrm{  for all  } i\geq 0.
\end{equation}
Since the radii of the balls tend to zero, at the end of the game the intersection of Bob's (or Alice's) balls must consist of a single point. See Figure \ref{fig:schmidtsGame} for a pictorial representation of the game, in which the shaded area represents where Alice has specified Bob must play his next ball. A set $E\sub X$ is called \textit{$(\alpha,\beta)$ winning} if Alice has a strategy for placing her balls that ensures
\begin{equation}\label{eq:outcome}
\bigcap_{i =0}^\infty \, B_i \sub E.
\end{equation}
In this case we say that Alice wins. The \emph{$\alpha$ game} is
defined by allowing Bob to choose the parameter $\beta$ and then
continuing as in the $(\alpha,\beta)$ game. We say that $E\sub X$
is\textit{ $\alpha$ winning} if Alice has a winning strategy for the
$\alpha$ game. Note that $E\sub X$ is $\alpha$ winning if and only
if $E$ is $(\alpha,\beta)$ winning for all $0<\beta<1$. Finally,
\emph{Schmidt's game} is defined by allowing Alice to choose
$0<\alpha<1$ and then continuing as in the $\alpha$ game. We say
$E\sub X$ is \textit{winning} if Alice has a winning strategy for
Schmidt's game. Note that $E$ is winning if and only if it is $\alpha$ winning for
some $\alpha \in (0,1)$.

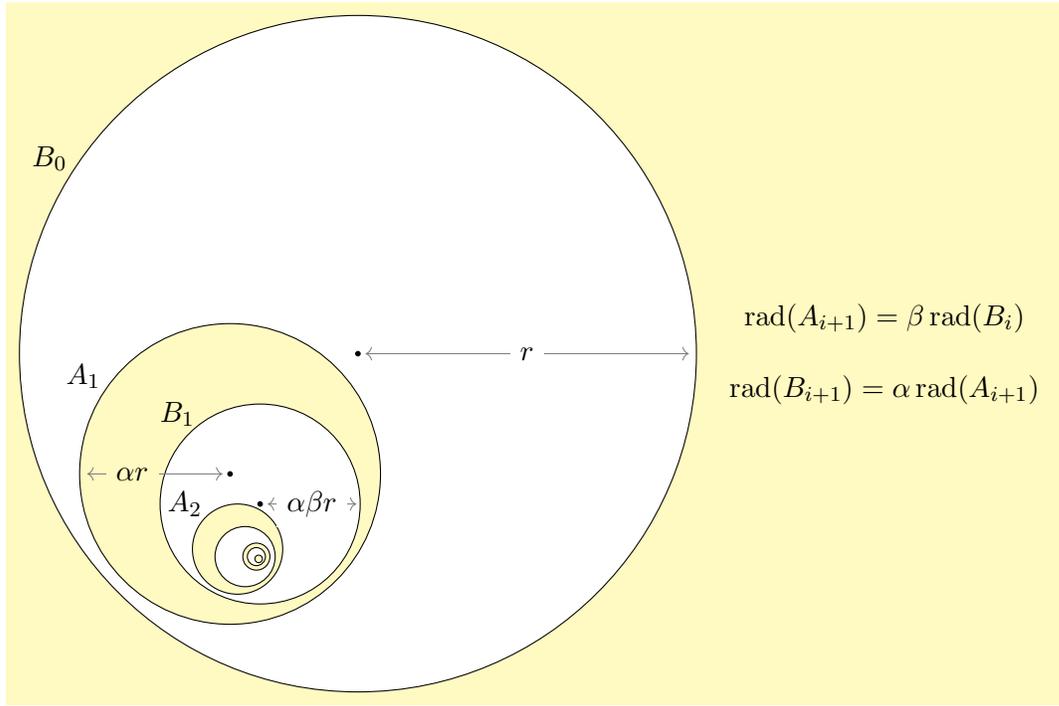
\begin{figure}
\begin{center}
\begin{tikzpicture}[background rectangle/.style={fill=yellow!30}, show background rectangle]

%Parameters: rad(B_0) = r = 4.5, \alpha = 4/9, \beta = 2/3

%balls
\draw[fill=white] (-0.2,0) circle [radius = 4.5]; %B_0
\draw[fill=yellow!30] (-1.9, -1.6) circle [radius =2]; %A_1
\draw[fill=white] (-1.5, -2.0) circle [radius =1.33]; %B_1
\draw[fill=yellow!30] (-1.8, -2.6) circle [radius =0.6]; %A_2
\draw[fill=white] (-1.7, -2.7) circle [radius =0.4];%B_2
\draw[fill=yellow!30] (-1.55, -2.7) circle [radius =0.18];%A_3
\draw[fill=white] (-1.55, -2.7) circle [radius =0.12];%B_3
\draw[fill=yellow!30] (-1.52, -2.73) circle [radius =0.05];%A_4

%centres
\draw[fill] (-0.2,0) circle [radius=0.03]; %B_0
\draw[fill] (-1.9, -1.6) circle [radius=0.03]; %A_1
\draw[fill] (-1.5, -2.0) circle [radius=0.03]; %B_1
%\draw[fill] (-1.8, -2.6) circle [radius=0.03]; %A_2
%\draw[fill] (-1.7, -2.7) circle [radius=0.03]; %B_2

% radius arrows
\draw[<->, gray, thin] (-0.1,0) --  (4.2, 0);  %B_0
\node[fill=white] at (2.05, 0) {$r$};  %B_0 label
\draw[<->, gray, thin] (-3.8, -1.6) -- (-2.0, -1.6);  %A_1
\node[fill=yellow!30] at (-3.2, -1.6)  {$\alpha r$};  %A_1 label
\draw[<->, gray, thin] (-1.40, -2.0)  -- (-0.22, -2.0);  %B_1
\node[fill=white] at (-0.81, -2.0) {$\alpha\beta r$}; %B_1 label

%ball labels
\node at (-4.3, 2.6) {$B_0$}; %B_0
\node at (-3.85, -0.3) {$A_1$}; %A_1
\node at (-2.6, -0.8) {$B_1$}; %B_1
\node at (-2.5, -2.0) {$A_2$}; %A_2

%formula (remove for neater diagram)
\node[align=center]  at (6.8, 0) {$\displaystyle {\rm rad}(A_{i+1}) = \beta \, {\rm rad}(B_i)$ \\ \\ ${\rm rad}(B_{i+1}) = \alpha \, {\rm rad}(A_{i+1})$};

\end{tikzpicture}

\end{center}
\caption{Schmidt's game}
\label{fig:schmidtsGame}
\end{figure}

It is easily observed that if a set $E$ is $\alpha$ winning for some
$\alpha \in (0,1)$ then $E$ is $\alpha'$ winning for every $\alpha'
\in (0, \alpha)$, and that $E$ is $\alpha$ winning for some $\alpha
\in (1/2,1)$ if and only if $E=X$ (cf. \cite[Lemma 5 and 8, respectively]{Schmidt1}); i.e., the property of being
$1/2$ winning is the strongest possible for a nontrivial subset.
% of $\RR^N$.

\begin{rem}\label{rem:closedBalls}
Strictly speaking, since a ball in an abstract metric space $X$ may not necessarily possess a unique centre and radius, Alice and Bob should pick successive pairs of centres and radii for their balls, satisfying certain distance inequalities that formally guarantee that the appropriate subset relations hold between the corresponding balls. However, we will assume that this nuance is accounted for in each player's strategy.
\end{rem}

We end the introduction of Schmidt games with the proof of the
folklore Proposition~\ref{lemmacontinuity}.

\begin{proof}[Proof of Proposition \ref{lemmacontinuity}]
Fix $0 < \beta_0 < 1$, let $\alpha_0$ be as in the statement, and set
\begin{equation}\label{def_ab}
\alpha = \alpha_0\frac{1 - (\alpha_0 \beta_0)^2 - \beta_0 (1-\alpha_0)}{1 - (\alpha_0 \beta_0)^2 - \alpha_0 \beta_0 (1-\alpha_0)} <
\alpha_0,\quad \alpha\beta = \left(\alpha_0\beta_0\right)^2.
\end{equation}
%\comdavid{This formula is equivalent to the last line in the proof, I don't know of a nicer way to write it. Of course, it is obvious that such a formula exists because as $\alpha\to \alpha_0$, the ratio between the two sides of the last line tends to $\frac{1-\beta_0 + \alpha_0\beta_0 - \alpha_0 \beta_0^2}{1-\alpha_0\beta_0^2} < 1$.}
We describe the winning strategy for Alice for the
$(\alpha_0,\beta_0)$ game (Game I) based on her strategy for the
$(\alpha,\beta)$ game (Game II). It can be given as follows:
\begin{itemize}
\item[1.] Whenever Bob plays a move $B_{2n} = B(b_{2n},\rho_{2n})$ in Game I, he correspondingly plays the move $B_n' = B(b_n',\rho_n')$ in Game II, where
\begin{align*}
(1-\alpha)\rho_n' &= (1-\alpha_0) \rho_{2n},&
b_n' &= b_{2n}.
\end{align*}
\item[2.] When Alice responds to this by playing a move $A_{n+1}' = B(a_{n+1}',\alpha\rho_n')$ in Game II (according to her winning strategy), she correspondingly plays the move $A_{2n+1} = B(a_{2n+1},\alpha_0 \rho_{2n})$ in Game I, where
\[
a_{2n+1} = a_{n+1}'.
\]
\item[3.] When Bob responds to this by playing a move $B_{2n+1} = B(b_{2n+1},\rho_{2n+1})$ in Game I, Alice responds by playing the move $A_{2n+2} = B(a_{2n+2},\alpha_0 \rho_{2n+1})$ in Game I, where
\[
a_{2n+2} = b_{2n+1}.
\]
This sets the stage for the next iteration.
\end{itemize}
To finish the proof, it suffices to check that all of these moves are legal. When $n=0$, move 1 is legal because any ball can be Bob's first move. Move 2 is legal because
\[
\vd(b_{2n},a_{2n+1}) = \vd(b_n',a_{n+1}') \leq (1-\alpha)\rho_n' = (1-\alpha_0) \rho_{2n},
\]
and move 3 is legal because $\vd(b_{2n+1},a_{2n+2}) = 0$. When $n\geq 1$, move 1 is legal because
\begin{align*}
\vd(a_n',b_n') &= \vd(a_{2n-1},b_{2n}) \leq \vd(a_{2n-1},b_{2n-1}) + \vd(a_{2n},b_{2n})\\
&\leq (1-\beta_0)\alpha_0 \rho_{2n-2} + (1-\beta_0)\alpha_0(\alpha_0\beta_0) \rho_{2n-2} = (1-\beta)\alpha\rho_{n-1}'.
\qedhere\end{align*}
\end{proof}

For details and other properties of winning sets see
%, for example,
Schmidt's book~\cite{Schmidt3}.

\subsection{The absolute game}\label{sec:absgame}

The \textit{absolute game} was introduced by McMullen in \cite{McMullen_absolute_winning}.
It is played similarly to Schmidt's game but instead of choosing a region where Bob
must play, Alice now chooses a region where Bob must not play. To be
precise, given $0<\beta<1$, the $\beta$ absolute game is played as
follows: Bob first picks some initial ball $B_0 \sub X$. Alice and
Bob then take turns to place successive closed balls in such a way
that

\begin{equation}\label{eq:absoluteGame}
B_0 \: \subp\: B_0 \setminus A_1 \:\subp \:B_1 \subp \cdots \subp\: B_i \: \subp \: B_i \setminus A_{i+1} \:\subp \:B_{i+1} \:\subp \: \cdots,
\end{equation}
subject to the conditions

\begin{equation}\label{eq:absoluteInequalities}
\rad(A_{i+1}) \leq \beta \cdot \rad(B_i) \quad \textrm{  and  } \quad \rad(B_{i+1}) \geq \beta \cdot \rad(B_i) \quad \textrm{  for all  } i \geq 0.
\end{equation}
See Figure \ref{fig:absoluteGame} for a pictorial view of the game in $\R^2$.

\begin{figure}
\begin{center}
\begin{tikzpicture}[background rectangle/.style={fill=yellow!30}, show background rectangle]

%Parameters:  rad(B_0) = r = 4.5, \beta = 3/10

%balls
\draw[fill=yellow!30] (-0.2,0) circle [radius = 4.5]; %B_0
\draw[fill=white] (-1.9, -1.6) circle [radius =1.35]; %A_1
\draw (0.3, 2.3) circle [radius =1.35]; %B_1
\draw[fill=white] (0.2, 3.1) circle [radius =0.41]; %A_2
\draw (0.1, 1.5) circle [radius =0.41]; %B_2
\draw[fill=white] (0.05, 1.6) circle [radius =0.12];%A_3
\draw (0.15, 1.3) circle [radius =0.12];%B_3
\draw[fill=white] (0.11, 1.25) circle [radius =0.04];%A_4
\draw (0.19, 1.35) circle [radius =0.04];%B_4

%centres
\draw[fill] (-0.2,0) circle [radius=0.03]; %B_0
\draw[fill] (-1.9, -1.6) circle [radius=0.03]; %A_1
\draw[fill] (0.3, 2.3) circle [radius=0.03]; %B_1
%\draw[fill] (0.2, 3.1) circle [radius=0.03]; %A_2
%\draw[fill] (0.1, 1.5) circle [radius=0.03]; %B_2

% radius arrows
\draw[<->, gray, thin] (-0.1,0) --  (4.2, 0);  %B_0
\node[fill=yellow!30] at (2.05, 0) {$r$};  %B_0 label
\draw[<->, gray, thin] (-3.15, -1.6) -- (-2.0, -1.6);  %A_1
\node[fill=white] at (-2.56, -1.6)  {$\beta r$};  %A_1 label
\draw[<->, gray, thin] (0.4, 2.3)  -- (1.55, 2.3);  %B_1
\node[fill=yellow!30] at (0.98, 2.3) {$\beta r$}; %B_1 label

%ball labels
\node at (-4.3, 2.6) {$B_0$}; %B_0
\node at (-3.35, -0.7) {$A_1$}; %A_1
\node at (-1.1, 3.3) {$B_1$}; %B_1
\node at (0.95, 3.1) {$A_2$}; %A_2
\node at (-0.5, 1.95) {$B_2$}; %B_2

%formula (remove for neater diagram)
\node[align=center] at (6.8, 0) {$\displaystyle {\rm rad}(A_{i+1}) \leq \beta \, {\rm rad}(B_i)$ \\ \\ ${\rm rad}(B_{i+1}) \geq \beta \, {\rm rad}(B_i)$};

\end{tikzpicture}
\end{center}
\caption{The absolute game}
\label{fig:absoluteGame}
\end{figure}

If $r_i:= \rad(B_i) \not \rightarrow 0$ then we say that Alice wins
by default. We say a set $E \sub X$ is \textit{$\beta$ absolute
winning} if Alice has a strategy which guarantees that either she
wins by default or~(\ref{eq:outcome}) is satisfied. In this case, we
say that Alice wins. The \emph{absolute game} is defined by allowing
Bob to choose $0<\beta<1$ on his first turn and then continuing as
in the $\beta$ absolute game. A set is called \emph{absolute
winning} if Alice has a winning strategy for the absolute game. Note
that a set is absolute winning if and only if it is $\beta$ absolute
winning for every $0<\beta<1$.

%For the sake of completeness, we briefly introduce 
A useful generalisation of the absolute game 
%which 
was introduced in \cite[Appendix C]{FSU4}. Let $\calh$ be any collection of nonempty closed subsets of $X$. The \emph{$\calh$ absolute game} is defined in a similar way to the absolute game, but Alice is now allowed to choose a neighborhood of an element $H\in\calh$, namely,
\[
A_{i+1} = B\left(H,\beta\cdot \rad\left(B_i\right)\right).
\]
McMullen's absolute game is the same as the $\calh$ absolute game when $\calh$ is the collection of all singletons. The first generalisation of the absolute game was the \emph{$k$ dimensional absolute game}, which is the $\calh$ absolute game where $X=\RR^N$ and $\calh$ is the collection of $k$ dimensional hyperplanes in $\RR^N$. When $k=N-1$ the associated winning sets are often referred to by the acronym HAW, short for ``hyperplane absolute winning'' \cite{BBFKW}. Some of the results in this note regarding the absolute game may be generalised to the context of $\calh$ absolute game but we will not pursue this line further in the current note. %\comerez{REFER THE READER TO DAVID'S BOOK?}

\begin{rem}
These definitions deserve some justification. We do not
require $0 < \beta < \frac{1}{3}$ as in the original
definitions for $X=\RR^N$ \cite{McMullen_absolute_winning,BBFKW}. To cover those cases where Bob chooses
$\frac{1}{3}\leq \beta < 1$, in which Bob might not have a legal
move that satisfies (\ref{eq:absoluteGame}) on one of his turns, we
declare that Bob is losing if this happens. This gives the wanted
effect when $X=\RR^N$, but puts us in a situation that a subset of
an abstract metric space might be winning because the space $X$ is
not large enough to allow Bob not to lose by default. In this
situation, even the empty set might be absolute winning (cf. Remark
\ref{rem:emptysetIsWinning}). This situation, however, cannot occur
when $X$ is uniformly perfect
(cf.~Section \ref{sec:doublingAndDiffuse}).

Also note that it is shown in \cite[Proposition 4.5]{FSU4} that
replacing the inequalities (\ref{eq:absoluteInequalities}) with
equalities does not change the class of absolute winning sets. The
reason for introducing these inequalities is to allow an easy proof
of the invariance of absolute winning sets under ``quasisymmetric
maps'' (cf.~\cite[Theorem 2.2]{McMullen_absolute_winning}).
\end{rem}

\subsection{The potential game}\label{sec:potential}

Given two parameters $c, \beta>0$, the \emph{$(c,\beta)$ potential game}
begins with Bob picking some initial ball~$B_0 \sub X$. On her %$i$th turn,
$(i+1)$st turn,
Alice chooses a countable collection of closed balls $A_{i+1, k}$, satisfying
\begin{equation}\label{eq:potentialLegal}
\sum_{k} \rad\left(A_{i+1,k}\right)^c \leq \left(\beta \cdot \rad\left(B_{i}\right)\right)^c.
\end{equation}
%\comstephen{TO BE CONSISTENT WITH THE NOTATION FOR THE OTHER GAMES I THINK THE ABOVE SHOULD BE $\mathcal{B}_{i-1}$, SO ALICE'S FIRST COLLECTION IS $\mathcal A_1$ not $\mathcal A_0$. SIMILARLY IN \eqref{eq:potentialDefault}. FUTURE PROOFS MAY NEED UPDATING SO I'VE LEFT IT FOR NOW.}
By convention, we assume that a ball always has a positive radius; i.e. Alice
can not choose %single elements for
$A_{i,k}$ to be singletons. Then, Bob on his $(i+1)$st
%th
turn chooses a ball $B_{i+1}$ satisfying $\rad\left(B_{i+1}\right) \geq \beta \cdot \rad\left(B_i\right)$.  See Figure \ref{fig:potentialGame} for a demonstration of the playing of such a game in $\R^2$ (in which, for simplicity, Bob's balls do not intersect Alice's preceding collections  - however, it should be understood that this is not a necessary condition for Bob).
%\comstephen{Added comment to make sure the figure is not misleading, and moved the figure to the bottom of the page to keep it in the relevant section.}

\begin{figure}[!b]
\begin{center}
\begin{tikzpicture}[background rectangle/.style={fill=yellow!30}, show background rectangle]

%Parameters:  rad(B_0) = r = 4.5, \beta = 3/10, c=1

%balls
\draw[fill=yellow!30] (-0.2,0) circle [radius = 4.5]; %B_0
%0.5 + 0.35 + 0.25 + 0.15 + 0.1 = \beta r = 1.35
\draw[fill=white] (-2.9, -1.6) circle [radius =0.5]; %A_{1, 1}
\draw[fill=white] (1.3, -1.9) circle [radius =0.35]; %A_{1, 2}
\draw[fill=white] (-2.1, 0.6) circle [radius =0.25]; %A_{1, 3}
\draw[fill=white] (-0.3, -2.9) circle [radius =0.15]; %A_{1, 4}
\draw[fill=white] (-1, -0.9) circle [radius =0.1]; %A_{1, 5}
\draw (0.3, 2.3) circle [radius =1.35]; %B_1
%0.15 + 0.15 + 0.06 + 0.05 = \beta^2 r = 0.41
\draw[fill=white] (0.2, 3.1) circle [radius =0.15]; %A_{2, 1}
\draw[fill=white] (0, 2.7) circle [radius =0.15]; %A_{2, 2}
\draw[fill=white] (-0.7, 2.6) circle [radius =0.06]; %A_{2, 3}
\draw[fill=white] (1, 1.6) circle [radius =0.05]; %A_{2, 4}
\draw (0.1, 1.5) circle [radius =0.41]; %B_2
%0.06 + 0.06 = \beta^3 r = 0.12
\draw[fill=white] (0.05, 1.6) circle [radius =0.06];%A_{3, 1}
\draw[fill=white] (0.35, 1.5) circle [radius =0.06];%A_{3, 2}
\draw (0.15, 1.3) circle [radius =0.12];%B_3
\draw[fill=white] (0.11, 1.25) circle [radius =0.04];%A_{4, 1}
\draw (0.18, 1.34) circle [radius =0.04];%B_4

%centres
\draw[fill] (-0.2,0) circle [radius=0.03]; %B_0
\draw[fill] (0.3, 2.3) circle [radius=0.03]; %B_1

% radius arrows
\draw[<->, gray, thin] (-0.1,0) --  (4.2, 0);  %B_0
\node[fill=yellow!30] at (2.05, 0) {$r$};  %B_0 label
\draw[<->, gray, thin] (0.4, 2.3)  -- (1.55, 2.3);  %B_1
\node[fill=yellow!30] at (0.98, 2.3) {$\beta r$}; %B_1 label

%ball labels
\node at (-4.3, 2.6) {$B_0$}; %B_0
\node at (-3.65, -1.1) {$A_{1,1}$}; %A_{1,1}
\node at (0.6, -1.5) {$A_{1,2}$}; %A_{1,2}
\node at (-2.7, 0.95) {$A_{1,3}$}; %A_{1,3}
\node at (-0.8, -2.6) {$A_{1,4}$}; %A_{1,4}
\node at (-1.4, -0.6) {$A_{1,5}$}; %A_{1,5}
\node at (-1.1, 3.3) {$B_1$}; %B_1
\node at (0.75, 3.1) {$A_{2, 1}$}; %A_{2, 1}
\node at (-0.5, 1.95) {$B_2$}; %B_2

%formula (remove for neater diagram)
\node[align=center] at (7.5, 0) {$\displaystyle \sum_k ({\rm rad}(A_{i+1, k}))^c \leq (\beta \, {\rm rad}(B_i))^c$ \\ \\ ${\rm rad}(B_{i+1}) \geq \beta \,{\rm rad}(B_i)$ };

\end{tikzpicture}
\end{center}
\caption{The potential game}
\label{fig:potentialGame}
\end{figure}
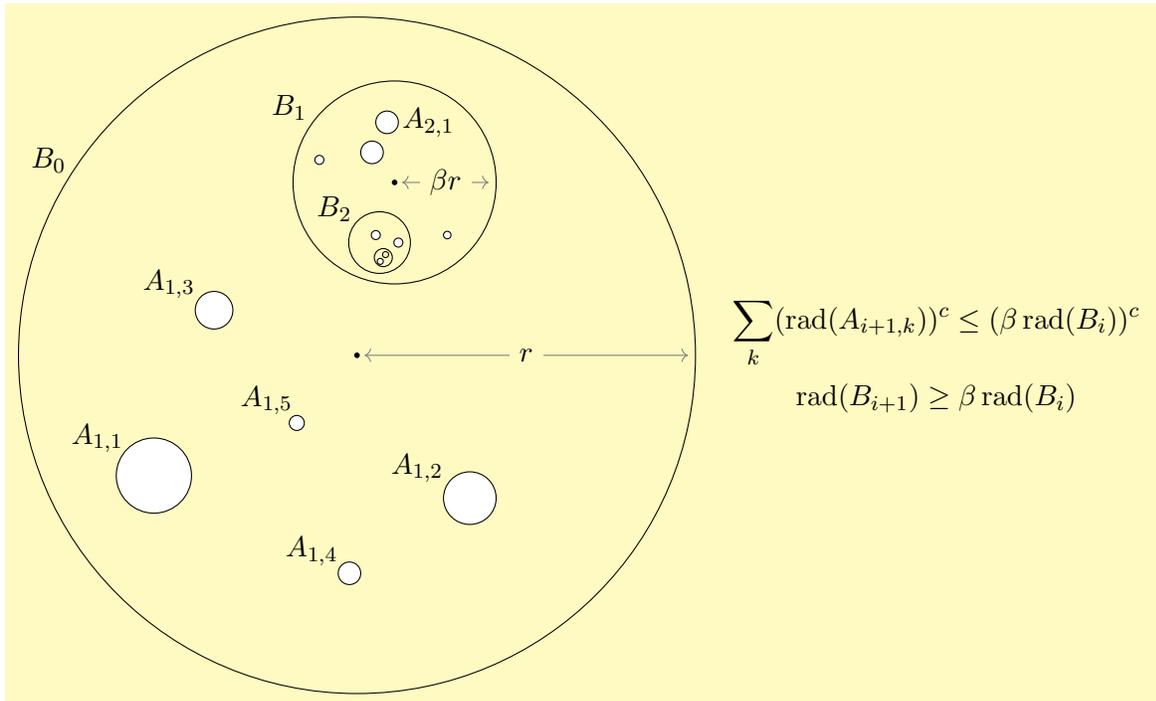

If
\begin{equation}\label{eq:potentialDefault}
\bigcap_{i} \, B_i  \sub  \bigcup_{i}\bigcup_{k} \, A_{i, k}
\end{equation}
or if $\rad\left(B_i\right) \not \rightarrow 0$ we say that Alice wins by default.
%Otherwise, the intersection $\bigcap_{i=0}^\infty \, B_i$ of Bob's balls is a singleton.

A set $E \sub X$ is
%then
said to be \textit{$(c, \beta)$ potential winning} if Alice has a strategy
guaranteeing that either she wins by default or (\ref{eq:outcome}) holds. The \emph{$c_0$ potential game} is defined by allowing Bob to choose in his first turn the parameters $c>c_0$ and $\beta>0$, and then continue with as in the $(c,\beta)$ potential game. A set $E$ is \textit{$c_0$ potential winning} if Alice has a winning strategy for $E$ in the $c_0$ potential game. Note that $E$ is $c_0$ potential winning if and only if $E$ is $(c, \beta)$ potential winning for every $c > c_0$ and $\beta > 0$. If $X$ is Ahlfors $\delta$ regular for some $\delta>0$, the \emph{potential game} is defined by allowing Alice to choose $c_0<\delta$ and then continuing as in the $c_0$ potential game,
and a set $E$ is called \emph{potential winning} if Alice has a winning strategy in the potential game. Note that $E$ is potential winning if and only if it is $c_0$ potential winning for some $c_0<\delta$.
%Also note that this definition disagrees with the one introduced in \cite{FSU4}; this is because the concept of ``potential winning'' as defined there is shown to be equivalent to absolute winning, and is therefore redundant. \comdavid{Here I introduced a definition of potential winning that disagrees with the definition found in \cite{FSU4}; I think we should either keep it like this or just remove the definition, since I don't like the definition of potential winning in \cite{FSU4}. Note that the proof of Corollary \ref{theoremPWWintersection} currently uses the definition of potential winning given above.}

\begin{rem}
The $c_0$ potential game with $c_0=0$ was introduced in \cite[Appendix C]{FSU4} as the potential game. In light of the connections between Cantor, absolute, and potential winning~\S\ref{sec:CantorPotentialconnections} and the geometrical interpretation \S\ref{sec:intersections}, it is natural to adjust the definition that appears there to the one that appears here.
\end{rem}

%\comdzmitry{I do not see that we check (W2) for potential games anywhere. We can not refer to that property of Cantor winning, since potential winning is more general than Cantor winning. Therefore I put the proof of (W2) here.}
For the sake of completeness we will verify Properties~(W1) and (W2) for potential winning
%games
sets. Property~(W1) is verified for a class of Ahlfors regular spaces $X$ in Theorem~\ref{propositionfullHD}.

\begin{proposition}\label{w2_potential}
Let $E_1,E_2\sub X$ be $c_1$ and $c_2$ potential winning, respectively. Then $E_1\cap E_2$ is $\max\{c_1,c_2\}$ potential winning.
\end{proposition}

\begin{proof}
Fix arbitrary $c>\max\{c_1,c_2\}$ and $\beta>0$. By assumption, $E_1$ and $E_2$ are both $(c, \beta^2)$ potential winning. Now Alice uses the following strategy to win the $(c,\beta)$ game on the set $E_1\cap E_2$. After Bob's initial move $B_0$ Alice chooses the collection $A_{1,k}$ from her $(c,\beta^2)$ winning strategy for the set $E_1$. After Bob's first move $B_1$ Alice chooses the collection $A^*_{1,k}$ from her $(c,\beta^2)$ winning strategy for the set $E_2$, assuming that Bob's initial move was $B_1$.

In general, after $B_{2m}$ Alice chooses the collection $A_{m,k}$ following her $(c,\beta^2)$ winning strategy for the set $E_1$; after $B_{2m+1}$ she chooses the collection $A^*_{m,k}$ following her $(c,\beta^2)$ winning strategy for the set $E_2$.

It is easy to check that Alice's moves are legal, because
$$
\sum_{k} \rad\left(A_{i+1,k}\right)^c \leq \left(\beta^2 \cdot \rad\left(B_{i}\right)\right)^c\leq \left(\beta \cdot \rad\left(B_{i}\right)\right)^c .
$$
The same inequality is true for the collections $A^*_{i+1,k}$. Also, by construction, either Alice wins by default or $\bigcap_i B_i \subseteq E_1\cap E_2$.
\end{proof}

\subsection{The strong and the weak Schmidt games}
\label{subsectionstrongweak}

One can define further variants of Schmidt's original game. The \textit{strong game} (introduced by McMullen in \cite{McMullen_absolute_winning}) coincides with Schmidt's game except that the equalities~(\ref{scmcond}) are replaced by the inequalities
\begin{equation}
\rad(A_{i+1}) \geq \alpha \cdot \rad(B_i) \textrm{  and  } \rad(B_{i+1}) \geq \beta \cdot \rad(A_{i+1}) \textrm{  for all  } i\geq 0.
\end{equation}
%One can show that \textit{strong winning sets} satisfy (W1) and (W2), and in the case of (W3) are preserved by quasisymmetric homeomorphisms.

Similarly, we define the \emph{weak game} so that Alice can choose her radii with an inequality but Bob must use equality. Obviously, every weak winning set is winning but the converse is not necessarily true (however, the weak game does not have the intersection property---see later).

One could also define a \emph{very strong game}, where Alice must use equality but Bob may use inequality.

\begin{rem}
It can be shown that a set is strong winning if and only if it is very strong winning, but this fact will not be needed hence we skip its proof.
\end{rem}
    % \comdavid{are we going to prove this?} \comstephen{I think we should just mention it can be shown and have this sentence as a remark. We don't actually use it do we? Adding a proof of something we don't use would probably just make the section more clunky}, but, if $\alpha$ and $\beta$ are fixed, then a set may be strong $(\alpha,\beta)$ winning without being very strong $(\alpha,\beta)$ winning.

Later in the paper we will need the following proposition which relates the
very strong winning property of a set $E$ and the weak winning property of its complement $X\butnot E$.

\begin{proposition}\label{prop_enote}
Let $E\sub X$ satisfy the following property: for any $\beta\in(0,1)$ there exists $\alpha>0$
such that $E$ is $(\alpha,\beta)$ very strong winning. Then $X\butnot E$ is not weak winning.
\end{proposition}

\begin{proof}
We need to show that for any $\beta>0$ the set $X\butnot E$ is not $\beta$ weak
winning. In order to do that, we will provide a winning strategy for Bob that
ensures that $\bigcap_{i=0}^\infty B_i \sub E$. On his first turn Bob chooses
$\alpha>0$ such that $E$ is $(\alpha, \beta)$ very strong winning, so that he is now playing the $(\beta,\alpha)$ weak game, and chooses an arbitrary
ball $B_0$. In his subsequent moves Bob follows Alice's winning strategy for the
$(\alpha, \beta)$ very strong game.
\end{proof}

\subsection{Cantor winning}

For the sake of clarity we introduce a specialisation of the
framework presented in \cite{BadziahinHarrap}, tailored to suit our
needs. We first define the notion of a ``splitting structure'' on
$X$. This is in some sense the minimal amount of structure the
metric space needs to allow
%for
the construction of
%Generalised
generalised Cantor sets (cf. \cite{BadziahinHarrap}). Denote by
$\BBB(X)$ the set of all closed balls. We assume that by definition, to specify a closed ball, it is necessary to specify its centre and radius. Note that this means that in some cases two distinct balls may be set-theoretically equal in the sense of containing the same points (see also Remark \ref{rem:closedBalls}). %\comdavid{Added footnote clarifying center-radius vs set-theoretic balls} %\comerez{moved the footnote into the main text and referenced a previous remark}
 in  $X$ and by $\PPP(\BBB(X))$
the set of all subsets of $\BBB(X)$. A \textit{splitting structure}
is a quadruple~$(X,\SSS, U,f)$, where
\begin{itemize}
    \item $U\sub \NN$ is an infinite multiplicatively closed set such that if
    $u,v\in U$ and $u\mid v$ then $v/u\in U$;
    \item $f\;:\; U\to \NN$ is a totally multiplicative function; % \comdavid{``totally multiplicative'' seems to be the standard terminology; and technically an arithmetic function must have all of $\NN$ as its domain}
    \item $\SSS\;:\BBB(X)\times U \to \mathcal P(\BBB(X))$ is a map defined in such a way that for every ball $B\in\BBB(X)$ and
    $u\in U$, the set $\SSS(B,u)$ consists solely of balls $S\sub B$ of
    radius $\rad(B)/u$.
\end{itemize}
Additionally, we require all these objects to be linked by the following properties:
\begin{enumerate}
    \item[(S1)] $\#\SSS(B,u) = f(u)$;
    \item[(S2)] If $S_1,S_2\in \SSS(B,u)$ and $S_1\neq S_2$ then $S_1$ and $S_2$ may only intersect on the
    boundary, i.e. the distance between their centres must be at least $\rad(S_1) + \rad(S_2) = 2\rad(B)/u$;

    \item[(S3)] For all $u,v\in U$,
    $$
    \SSS(B,uv) = \bigcup_{S \,\in \,\SSS(B,u)} \SSS(S,v).
    $$
\end{enumerate}

Broadly speaking, the set $U$ determines the set of possible ratios of radii of balls in the successive levels of the upcoming Cantor set construction. For any fixed $B \in \BBB(X)$, the function $f$ tells us how many balls inside $B$ of a fixed radius we may choose as candidates for the next level in the Cantor set construction, and the sets $\SSS(B, u)$ describe the position of these candidate balls.

If $f \equiv 1$ we say the splitting structure $(X,\SSS, U,f)$ is
trivial. Otherwise, fix a sequence $(u_i)_{i\in\NN}$ with $u_i\in U$
such that $u_i\mid u_{i+1}$ and
%$u_i\stackrel{i\to\infty}{\longrightarrow} \infty$
$u_i\underset{i\to\infty}{\longrightarrow} \infty$. Then, one can
show~\cite[Theorem 2.1]{BadziahinHarrap} that for any $B_0 \in\BBB(X)$ the limit set
$$
A_\infty(B_0):= \bigcap_{i=1}^\infty \:  \bigcup_{B\, \in\, \SSS(B,u_i)} B
$$
is compact and independent of the choice of sequence $(u_i)_{i\in\NN}$.

%We give an example to demonstrate the definition above, which is also the most important case for the purpose of our main results. For further examples and discussion see \cite{BadziahinHarrap}.

\begin{example}
The standard splitting structure on $X=\RR^N$ is with the metric that comes from the sup norm: $\vd(x,y):=\|x-y\|_\infty$. The collection of splittings $\SSS(B,u)$ of any closed ball $B$ according to the integer $u\in U = \NN$, into $f(u)=u^N$ parts, is defined as follows: Cut~$B$ into $u^N$ equal square boxes which edges have length $u$ times less than the edges length of~$B$.
One can
%easily
check that $(\RR^N,\SSS, \NN, f)$ satisfies properties (S1)~--~(S3) and is the unique splitting structure for $\left(\RR^N, \vd\right)$ with $U=\NN$ such that $A_\infty(B) = B$ for every $B\in \BBB(X)$.

In a similar manner, the standard splitting structure on $\{0,1\}^\bbn$ is with the standard metric $\vd(x,y):=2^{-\min\left\{i\geq0\sep x_i\neq y_i\right\}}$, $U=\{2^i\sep i\geq0\}$, $f(u)=u$, and
\[
\SSS([x_1,\ldots,x_n],2^i)=\left\{[x_1,\ldots,x_n,y_1,\ldots,y_i]\sep (y_1,\ldots,y_i)\in\{0,1\}^i \right\}
%\SSS(\{y\sep y_k=x_k \textrm{ for all } 1\leq k\leq n\},2^i)=\left\{ \right\}
\]
where for any $x\in X$ and $n\geq 0$ we define $[x_1,\ldots,x_n]:=\{y\sep y_k=x_k \textrm{ for all } 1\leq k\leq n\}=B(x,2^{-n})$ to be the cylinder with prequel $(x_1,\ldots,x_n)$, thought of as a ball of radius $2^{-n}$ (the choice of centre is arbitrary).

%Whenever we mention generalised Cantor sets not in the context of a splitting
%structure we will refer to a standard splitting structure.
%\comstephen{this sentence was already included later, in perhaps a better place}
\end{example}

For any collection $\BBB \sub \mathcal P(\BBB(X))$ of balls and for
any $R\in U$ we will write
$$
\frac{1}{R} \AAA:= \bigcup_{A \in \AAA} \SSS(A, R).
$$
Of course, for higher powers of $R$ we can write $\frac1{R^n} \AAA = \frac1R\left(\cdots\left(\frac1R\AAA\right)\cdots\right)$.

%Let $R \geq 3$ be an integer. Then, given a collection $\BBB$ of closed (metric) balls in $\RR^N$ let $\frac 1 R\BBB$ denote the collection of balls obtained by dividing each $B \in \BBB$ into $R^N$ closed balls of equal radius $\rad(B)/R$.
%For ease of notation,
%A \emph{Cantor-type construction} is any sequence of collections $\{\calb_n\}_{n=0}^{\infty}$ satisfying $\calb_0=\{B_0\}$ and $\calb_{n+1}\sub\frac{1}{R}\calb_n$, where $B_0\sub X$ is a closed ball. The \emph{limit set} of such a construction is the set $\bigcap_{i=0}^\infty \bigcup_{b\in\BBB_n} b$.

We can now recall the
%precise
definition of generalised Cantor sets. % (Can change to Erez' shorthand defn if preferred.)
Fix some closed ball $B_0 \sub X$ and some $R \in U$,
%. Then,
and let
$$
\vr:= (r_{m,n}),\; m,n\in\ZZ_{\ge 0},\; m\le n
$$
be a two-parameter sequence of nonnegative real numbers. %\comdavid{I changed ``positive'' to ``nonnegative'' here because we have $r_{m,n} = 0$ sometimes later}
 Define $\BBB_0:=\{B_0\}$. We start by considering the set
$\frac 1 R\BBB_0$. The first step in the construction of a generalised
Cantor set involves the removal of a collection $\AAA_{0,0} $ of at most $r_{0,0}$ balls from $\frac 1 R\BBB_0$. We label the surviving collection as $\BBB_1$. %Balls in $\BBB_1$ will be referred as (level one) survivors.
Note that we do not specify here the removed balls, we just give an upper bound for their number.
In general, for a fixed $n\ge 0$, given the collection $\BBB_n$ we construct a
nested collection $\BBB_{n+1}$ as follows: Let
\begin{equation}\label{eq:cantorwinning}
\BBB_{n+1}: = \left(\frac{1}{R} \BBB_n\right) \setminus
\bigcup_{m=0}^n \AAA_{m,n}, \quad \text{where} \quad \AAA_{m,n} \sub
\frac{1}{R^{n-m+1}} \BBB_m  \quad (0 \leq m \leq n)
\end{equation}
are collections of removed balls satisfying for every $B \in \BBB_m$ the
inequality
$$\# \AAA_{m,n}(B)  \leq r_{m,n}, \quad \text{where } \quad \AAA_{m,n}(B):=\left\{ A \in \AAA_{m,n}: \, A \sub B   \right\}.$$
For a pictorial demonstration  in $\R^2$ of such a construction, see Figure \ref{fig:cantorWinning},  in which boxes removed in the current step are coloured red.
%\comstephen{Label fixed and figure resized/renamed/repositioned. My printer/pdf viewer sometimes struggles to compile it perfectly, especially the final step where some odd gray lines are added and some dotted/thick lines appear to overlap (they disappear on zooming in). If this persists it may turn out better just to save as png and attatch inline.}

%Cantor winning set construction (with redness, 3 steps)

%Parameters: rad(B) = 3, R = 3
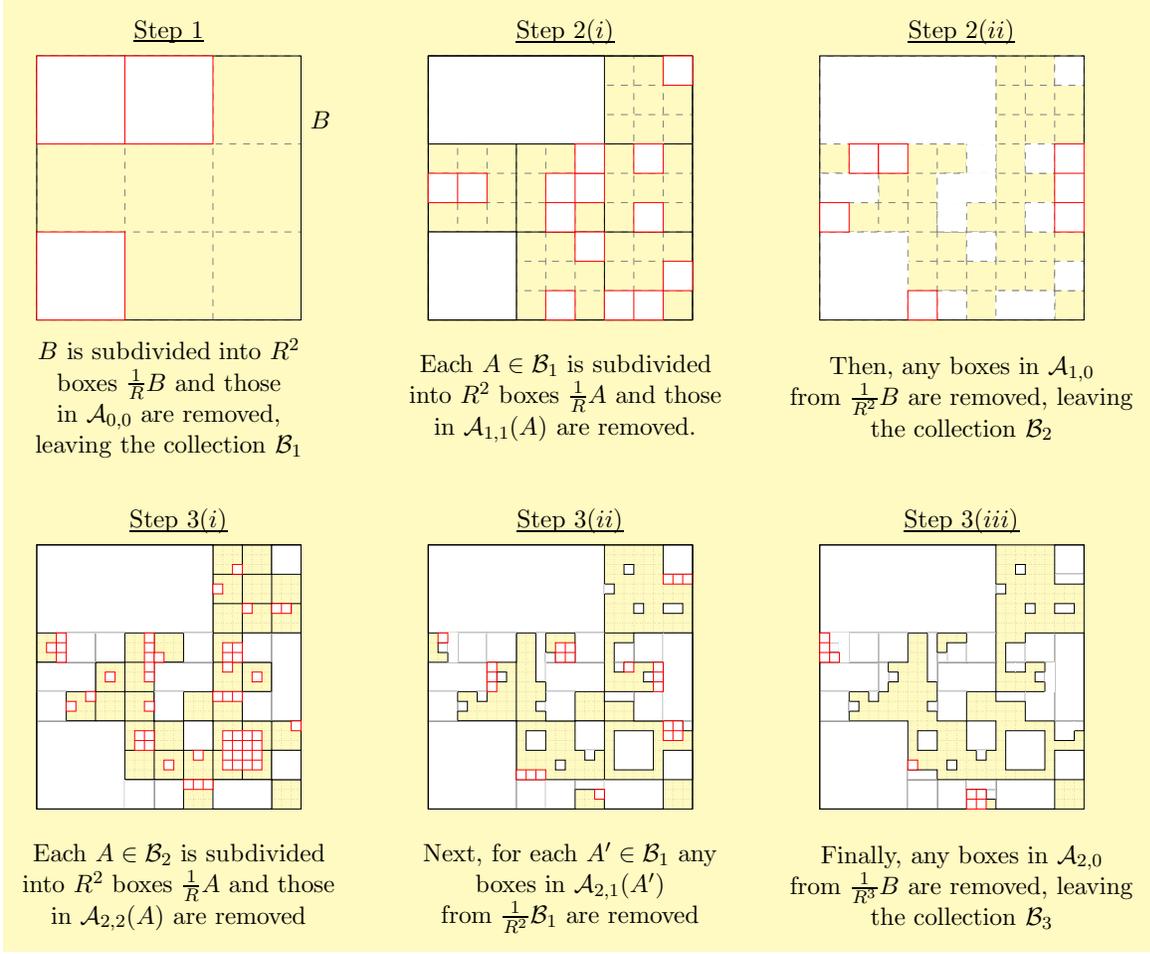
\begin{figure}[!t]\resizebox{\linewidth}{!}{
\begin{tikzpicture}[background rectangle/.style={fill=yellow!30}, show background rectangle, scale = 0.75]
%\begin{figure}
%\begin{center}
%\begin{tikzpicture}[background rectangle/.style={fill=yellow!30}, show background rectangle, scale = 0.75]

%have to do each drawing in reverse for colouring to work!

%------------------

%step 1

%B grids
\draw[help lines, gray, very thin, dashed, step=1.8] (0,0) grid (5.4,5.4);
%\draw[fill] (-0.2,0) circle [radius=0.03]; %B_0
%\draw[fill] (0.3, 2.3) circle [radius=0.03]; %B_1

%B box
\draw (0,0) rectangle (5.4, 5.4); %B

%\mathcal A_{0,0} boxes
\draw[red, fill=white] (0,0) rectangle (1.8, 1.8); %A_{0, 1}
\draw[red, fill=white] (0,3.6) rectangle (1.8, 5.4); %A_{0, 2}
\draw[red, fill=white] (1.8,3.6) rectangle (3.6, 5.4); %A_{0, 3}

%box labels
\node at (5.8, 4.1) {$B$}; %B

%titles/explanation
\node at (2.7, 5.9) {\underbar{Step $1$}};
\node[align=center] at (2.7, -1.6) {$B$ is subdivided into $R^2$ \\ boxes $\frac 1 R B$ and those \\ in $\mathcal A_{0,0}$ are removed, \\ leaving the collection $\mathcal B_1$};

%--------------

%step 2(i)

%B grids
\draw[help lines, gray, very thin, dashed, step=0.6, xshift=-1cm] (9,0) grid (14.4,5.4);
\draw[help lines, black, very thin, step=1.8, xshift=-1cm] (9,0) grid (14.4,5.4);
%dash removal \mathcal A_{0,0}
%\draw[help lines, white,  step=0.6, xshift=-1cm] (9, 3.6) grid (12.6,5.4);
%\draw[gray, very thin, dashed] (8, 3.6) -- (13.6, 3.6);
%\draw[help lines, white,  step=0.6, xshift=-1cm] (9,0) grid (10.04,2);
%dash removal \mathcal A_{1,0}
%\draw[help lines, white,  step=0.67, xshift=0.96cm] (12.04,0) grid (14.04,2);
%\draw[gray, very thin, dashed] (13, 0) -- (13, 2);

%\mathcal A_{0,0} boxes
\draw[black, very thin, fill=white] (8,0) rectangle (9.8, 1.8); %A_{0, 1}
\draw[black, very thin, fill=white] (8,3.6) rectangle (11.6, 5.4); %A_{0, 2} & A_{0, 3}

%B box
\draw (8,0) rectangle (13.4, 5.4); %B

%0 (0.6, 1.2) 1.8 (2.4, 3) 3.6 (4.2, 4.8) 5.4
%8 (8.6, 9.2) 9.8 (10.4, 11) 11.6 (12.2, 12.8) 13.4

%\mathcal A_{1,1} boxes
\draw[red, fill=white] (10.4,0) rectangle (11, 0.6); %A_{1, 1, 1}
\draw[red, fill=white] (11, 1.2) rectangle (11.6, 1.8); %A_{1, 1, 2}
\draw[red, fill=white] (8, 2.4) rectangle (9.2, 3); %A_{1, 1, 3} & %A_{1, 1, 4}
\draw[red] (8.6, 2.4) -- (8.6, 3);
\draw[red, fill=white] (10.4,2.4) rectangle (11.6, 3); %A_{1, 1, 5} & %A_{1, 1, 6}
\draw[red] (11, 2.4) -- (11, 3);
\draw[red, fill=white] (11, 3) rectangle (11.6, 3.6); %A_{1, 1, 7}
\draw[red, fill=white] (10.4, 1.8) rectangle (11, 2.4); %A_{1, 1, 8}
\draw[red, fill=white] (12.2,1.8) rectangle (12.8, 2.4); %A_{1, 1, 9}
\draw[red, fill=white] (12.2, 3) rectangle (12.8, 3.6); %A_{1, 1, 10}
\draw[red, fill=white] (12.8, 4.8) rectangle (13.4, 5.4); %A_{1, 1, 11}
\draw[red, fill=white] (11.6, 0) rectangle (12.8, 0.6); %A_{1, 1, 12} & %A_{1, 1, 13}
\draw[red] (12.2, 0) -- (12.2, 0.6);
\draw[red, fill=white] (12.8, 0.6) rectangle (13.4, 1.2); %A_{1, 1, 14}

%box labels
%\node at (23, 4.1) {$B$}; %B

%titles/explanation
\node at (10.8, 5.9) {\underbar{Step $2(i)$}};
\node[align=center] at (10.8, -1.6) {Each $A \in \mathcal B_1$ is subdivided
\\ into $R^2$ boxes $\frac 1 R A$ and those  \\ in $\mathcal A_{1,1}(A)$ are removed. %\\ Next, any boxes in $\mathcal A_{1,0}$ \\ from $\frac 1 {R^2} B$ are removed, leaving \\ the collection $\mathcal B_2$
};

%------------------------

%step 2(ii)

%B grids
\draw[help lines, gray, very thin, dashed, step=0.6, xshift=-2cm] (18,0) grid (23.4,5.4);
%\draw[help lines, gray, very thin, step=0.6, xshift=-2cm] (18,0) grid (23.4,5.4);
%dash removal \mathcal A_{0,0}
%\draw[help lines, white,  step=0.6, xshift=-1cm] (9, 3.6) grid (12.6,5.4);
%\draw[gray, very thin, dashed] (8, 3.6) -- (13.6, 3.6);
%\draw[help lines, white,  step=0.6, xshift=-1cm] (9,0) grid (10.04,2);
%dash removal \mathcal A_{1,0}
%\draw[help lines, white,  step=0.67, xshift=0.96cm] (12.04,0) grid (14.04,2);
%\draw[gray, very thin, dashed] (13, 0) -- (13, 2);

%\mathcal A_{0,0} boxes
\draw[gray, dashed, ultra thin, fill=white] (16,0) rectangle (17.8, 1.8); %A_{0, 1}
\draw[gray, dashed, ultra thin, fill=white] (16,3.6) rectangle (19.6, 5.4); %A_{0, 2} & A_{0, 3}

%\mathcal A_{1,1} boxes
\draw[gray, dashed, ultra thin, fill=white] (18.4,0) rectangle (19, 0.6); %A_{1, 1, 1}
\draw[gray, dashed, ultra thin, fill=white] (19, 1.2) rectangle (19.6, 1.8); %A_{1, 1, 2}
\draw[gray, dashed, ultra thin, fill=white] (16, 2.4) rectangle (17.2, 3); %A_{1, 1, 3} & %A_{1, 1, 4}
\draw[gray, dashed, ultra thin, fill=white] (18.4,2.4) rectangle (19.6, 3); %A_{1, 1, 5} & %A_{1, 1, 6}
\draw[gray, dashed, ultra thin, fill=white] (19, 3) rectangle (19.6, 3.6); %A_{1, 1, 7}
\draw[gray, dashed, ultra thin, fill=white] (18.4, 1.8) rectangle (19, 2.4); %A_{1, 1, 8}
\draw[gray, dashed, ultra thin, fill=white] (20.2,1.8) rectangle (20.8, 2.4); %A_{1, 1, 9}
\draw[gray, dashed, ultra thin, fill=white] (20.2, 3) rectangle (20.8, 3.6); %A_{1, 1, 10}
\draw[gray, dashed, ultra thin, fill=white] (20.8, 4.8) rectangle (21.4, 5.4); %A_{1, 1, 11}
\draw[gray, dashed, ultra thin, fill=white] (19.6, 0) rectangle (20.8, 0.6); %A_{1, 1, 12} & %A_{1, 1, 13}
%\draw[gray, ultra thin] (20.2, 0) -- (20.2, 0.6);
\draw[gray, dashed, ultra thin, fill=white] (20.8, 0.6) rectangle (21.4, 1.2); %A_{1, 1, 14}

%remove dashes
\draw[white, thick] (17.2, 3.6) -- (17.8, 3.6);
\draw[white, thick] (19, 3.6) -- (19.6, 3.6);
\draw[white, thick] (18.4, 2.4) -- (19, 2.4);
\draw[white, thick] (19, 3) -- (19.6, 3);

%B box
\draw (16,0) rectangle (21.4, 5.4); %B

%\mathcal A_{1,0} boxes
\draw[red, fill=white] (16, 1.8) rectangle (16.6, 2.4); %A_{1, 0, 1}
\draw[red, fill=white] (16.6, 3) rectangle (17.2, 3.6); %A_{1, 0, 2}
\draw[red, fill=white] (17.8, 0) rectangle (18.4, 0.6); %A_{1, 0, 3}
\draw[red, fill=white] (17.2, 3) rectangle (17.8, 3.6); %A_{1, 0, 4}
\draw[red, fill=white] (20.8, 1.8) rectangle (21.4, 3.6); %A_{1, 0, 5} & %A_{1, 0, 6} %A_{1, 0, 7}
\draw[red] (20.8, 2.4) -- (21.4, 2.4);
\draw[red] (20.8, 3) -- (21.4, 3);

%box labels
%\node at (23, 4.1) {$B$}; %B

%titles/explanation
\node at (18.9, 5.9) {\underbar{Step $2(ii)$}};
\node[align=center] at (18.9, -1.6) {Then, any boxes in $\mathcal A_{1,0}$ \\ from $\frac 1 {R^2} B$ are removed, leaving \\ the collection $\mathcal B_2$};

%---------------
%step 3(i)

%B grids
\draw[help lines, gray, ultra thin, step=0.2, densely dotted] (0,-10) grid (5.4,-4.6);
\draw[help lines, black, very thin, step=0.6, yshift=-1cm] (0,-9) grid (5.4,-3.6);
%%dash removal \mathcal A_{0,0}
%%\draw[help lines, white,  step=0.6, xshift=-1cm] (9, 3.6) grid (12.6,5.4);
%%\draw[gray, very thin, dashed] (8, 3.6) -- (13.6, 3.6);
%%\draw[help lines, white,  step=0.6, xshift=-1cm] (9,0) grid (10.04,2);
%%dash removal \mathcal A_{1,0}
%%\draw[help lines, white,  step=0.67, xshift=0.96cm] (12.04,0) grid (14.04,2);
%%\draw[gray, very thin, dashed] (13, 0) -- (13, 2);
%
%\mathcal A_{0,0} boxes
\draw[black, ultra thin, fill=white] (0,-10) rectangle (1.8, -8.2); %A_{0, 1}
\draw[black, ultra thin, fill=white] (0, -6.4) rectangle (3.6, -4.6); %A_{0, 2} & A_{0, 3}
%
%
%
%\mathcal A_{1,1} boxes
\draw[black, ultra thin, fill=white] (2.4,-10) rectangle (3, -9.4); %A_{1, 1, 1}
\draw[black, ultra thin, fill=white] (3, -8.8) rectangle (3.6, -8.2); %A_{1, 1, 2}
\draw[black, ultra thin, fill=white] (0, -7.6) rectangle (1.2, -7); %A_{1, 1, 3} & %A_{1, 1, 4}
\draw[black, ultra thin, fill=white] (2.4, -7.6) rectangle (3.6, -7); %A_{1, 1, 5} & %A_{1, 1, 6}
\draw[black, ultra thin, fill=white] (3, -7) rectangle (3.6, -6.4); %A_{1, 1, 7}
\draw[black, ultra thin, fill=white] (2.4, -8.2) rectangle (3, -7.6); %A_{1, 1, 8}
\draw[black, ultra thin, fill=white] (4.2, -8.2) rectangle (4.8, -7.6); %A_{1, 1, 9}
\draw[black, ultra thin, fill=white] (4.2, -7) rectangle (4.8, -6.4); %A_{1, 1, 10}
\draw[black, ultra thin, fill=white] (4.8, -5.2) rectangle (5.4, -4.6); %A_{1, 1, 11}
\draw[black, ultra thin, fill=white] (3.6, -10) rectangle (4.8, -9.4); %A_{1, 1, 12} & %A_{1, 1, 13}
\draw[black, ultra thin, fill=white] (4.8, -9.4) rectangle (5.4, -8.8); %A_{1, 1, 14}
%
%
%
%\mathcal A_{1,0} boxes
\draw[black, ultra thin, fill=white] (0, -8.2) rectangle (0.6, -7.6); %A_{1, 0, 1}
\draw[black, ultra thin, fill=white] (0.6, -7) rectangle (1.2, -6.4); %A_{1, 0, 2}
\draw[black, ultra thin, fill=white](1.8, -10) rectangle (2.4, -9.4); %A_{1, 0, 3}
\draw[black, ultra thin, fill=white] (1.2, -7) rectangle (1.8, -6.4); %A_{1, 0, 4}
\draw[black, ultra thin, fill=white] (4.8, -8.2) rectangle (5.4, -6.4); %A_{1, 0, 5} & %A_{1, 0, 6} %A_{1, 0, 7}

%remove dashes
\draw[white, thin] (1.2, -6.4) -- (1.8, -6.4);
\draw[white, thin] (3, -6.4) -- (3.6, -6.4);
\draw[white, thin] (2.4, -7.6) -- (3, -7.6);
\draw[white, thin] (3, -7) -- (3.6, -7);
\draw[white, thin] (0, -8.2) -- (0.6, -8.2);
\draw[white, thin] (0, -7.6) -- (0.6, -7.6);
\draw[white, thin] (0.6, -7) -- (1.2, -7);
\draw[white, thin] (0.6, -6.4) -- (1.2, -6.4);
\draw[white, thin] (1.2, -7) -- (1.2, -6.4);
\draw[white, thin] (1.8, -10) -- (1.8, -9.4);
\draw[white, thin] (2.4, -10) -- (2.4, -9.4);
\draw[white, thin] (4.8, -8.2) -- (4.8, -7.6);
\draw[white, thin] (4.8, -7) -- (4.8, -6.4);

%B box
\draw (0,-10) rectangle (5.4, -4.6); %B

%\mathcal A_{2,2} boxes
\draw[red, fill=white] (0.6,-8) rectangle (0.8, -7.8); %A_{2, 2, 1}
\draw[red, fill=white] (1, -7.8) rectangle (1.2, -7.6); %A_{2, 2, 2}
\draw[red, fill=white] (0.2, -6.8) rectangle (0.4, -6.6); %A_{2, 2, 3}
\draw[red, fill=white] (0.4, -7) rectangle (0.6, -6.4); %A_{2, 2, 4} & %A_{2, 2, 5} & %A_{2, 2, 6}
\draw[red] (0.4, -6.8) -- (0.6, -6.8);
\draw[red] (0.4, -6.6) -- (0.6, -6.6);
\draw[red, fill=white] (1.4, -7.4) rectangle (1.6, -7.2); %A_{2, 2, 7}

\draw[red, fill=white] (2,-8.8) rectangle (2.4, -8.4); %A_{2, 2, 7} & %A_{2, 2, 8} & A_{2, 2, 9} & %A_{2, 2, 10}
\draw[red] (2.2, -8.8) -- (2.2, -8.4);
\draw[red] (2, -8.6) -- (2.4, -8.6);
\draw[red, fill=white] (2.6, -9.2) rectangle (2.8, -9); %A_{2, 2, 11}
\draw[red, fill=white] (2.2, -8) rectangle (2.4, -7.8); %A_{2, 2, 12}
\draw[red, fill=white] (3.2, -9) rectangle (3.4, -8.8); %A_{2, 2, 13}
\draw[red, fill=white] (3, -9.6) rectangle (3.6, -9.4); %A_{2, 2, 14} & %A_{2, 2, 15} & A_{2, 2, 16}
\draw[red] (3.2, -9.6) -- (3.2, -9.4);
\draw[red] (3.4, -9.6) -- (3.4, -9.4);

\draw[red, fill=white] (3.8, -7) rectangle (4.2, -6.6); %A_{2, 2, 17} & %A_{2, 2, 18} & A_{2, 2, 19} & %A_{2, 2, 20}
\draw[red] (4, -7) -- (4, -6.6);
\draw[red] (3.8, -6.8) -- (4.2, -6.8);
\draw[red, fill=white] (4.4, -7.4) rectangle (4.6, -7.2); %A_{2, 2, 21}
\draw[red, fill=white] (5.2, -8.4) rectangle (5.4, -8.2); %A_{2, 2, 22}
\draw[red, fill=white] (3.8, -7.2) rectangle (4, -7); %A_{2, 2, 23}
\draw[red, fill=white] (3.6, -7.8) rectangle (4.2, -7.6); %A_{2, 2, 24} & %A_{2, 2, 25} & A_{2, 2, 26}
\draw[red] (3.8, -7.8) -- (3.8, -7.6);
\draw[red] (4, -7.8) -- (4, -7.6);
\draw[red, fill=white] (4, -5.2) rectangle (4.2, -5);
%A_{2, 2, 27}
\draw[red, fill=white] (3.6, -5.6) rectangle (3.8, -5.4); %A_{2, 2, 28}
\draw[red, fill=white] (4.2, -6) rectangle (4.4, -5.8); %A_{2, 2, 29}
\draw[red, fill=white] (4.8, -6) rectangle (5.2, -5.8); %A_{2, 2, 30} & %A_{2, 2, 31}
\draw[red] (5, -6) -- (5, -5.8);
\draw[red] (5.2, -6) -- (5.2, -5.8);

\draw[red, fill=white] (2.2, -7.4) rectangle (2.4, -6.4); %A_{2, 2, 32} to %A_{2, 2, 36}
\draw[red] (2.2, -7.2) -- (2.4, -7.2);
\draw[red] (2.2, -7) -- (2.4, -7);
\draw[red] (2.2, -6.8) -- (2.4, -6.8);
\draw[red] (2.2, -6.6) -- (2.4, -6.6);
\draw[red, fill=white] (2.4, -7) rectangle (2.6, -6.8); %A_{2, 2, 37}

\draw[red, fill=white] (3.8, -9.2) rectangle (4.6, -8.4); %A_{2, 2, 38} to %A_{2, 2, 53}
\draw[red] (3.8, -9) -- (4.6, -9);
\draw[red] (3.8, -8.8) -- (4.6, -8.8);
\draw[red] (3.8, -8.6) -- (4.6, -8.6);
\draw[red] (4, -9.2) -- (4, -8.4);
\draw[red] (4.2, -9.2) -- (4.2, -8.4);
\draw[red] (4.4, -9.2) -- (4.4, -8.4);

%%box labels
%%\node at (23, 4.1) {$B$}; %B

%titles/explanation
\node at (2.9, -4.1) {\underbar{Step $3(i)$}};
\node[align=center] at (2.9, -11.6) {Each $A \in \mathcal B_2$ is subdivided
\\ into $R^2$ boxes $\frac 1 R A$ and those \\ in $\mathcal A_{2,2}(A)$ are removed};

%---------------
%step 3(ii)

%B grids
\draw[help lines, gray, ultra thin, step=0.2,densely dotted, xshift=-1cm] (9,-10) grid (14.4,-4.6);
\draw[help lines, black, very thin, step=1.8, xshift=-1cm, yshift=-1cm] (9,-9) grid (14.4,-3.6);
%%dash removal \mathcal A_{0,0}
%%\draw[help lines, white,  step=0.6, xshift=-1cm] (9, 3.6) grid (12.6,5.4);
%%\draw[gray, very thin, dashed] (8, 3.6) -- (13.6, 3.6);
%%\draw[help lines, white,  step=0.6, xshift=-1cm] (9,0) grid (10.04,2);
%%dash removal \mathcal A_{1,0}
%%\draw[help lines, white,  step=0.67, xshift=0.96cm] (12.04,0) grid (14.04,2);
%%\draw[gray, very thin, dashed] (13, 0) -- (13, 2);

%\mathcal A_{0,0} boxes
\draw[black, ultra thin, fill=white] (8,-10) rectangle (9.8, -8.2); %A_{0, 1}
\draw[black, ultra thin, fill=white] (8, -6.4) rectangle (11.6, -4.6); %A_{0, 2} & A_{0, 3}
%
%
%
%\mathcal A_{1,1} boxes
\draw[black, ultra thin, fill=white] (10.4,-10) rectangle (11, -9.4); %A_{1, 1, 1}
\draw[black, ultra thin, fill=white] (11, -8.8) rectangle (11.6, -8.2); %A_{1, 1, 2}
\draw[black, ultra thin, fill=white] (8, -7.6) rectangle (9.2, -7); %A_{1, 1, 3} & %A_{1, 1, 4}
\draw[black, ultra thin, fill=white] (10.4, -7.6) rectangle (11.6, -7); %A_{1, 1, 5} & %A_{1, 1, 6}
\draw[black, ultra thin, fill=white] (11, -7) rectangle (11.6, -6.4); %A_{1, 1, 7}
\draw[black, ultra thin, fill=white] (10.4, -8.2) rectangle (11, -7.6); %A_{1, 1, 8}
\draw[black, ultra thin, fill=white] (12.2, -8.2) rectangle (12.8, -7.6); %A_{1, 1, 9}
\draw[black, ultra thin, fill=white] (12.2, -7) rectangle (12.8, -6.4); %A_{1, 1, 10}
\draw[black, ultra thin, fill=white] (12.8, -5.2) rectangle (13.4, -4.6); %A_{1, 1, 11}
\draw[black, ultra thin, fill=white] (11.6, -10) rectangle (12.8, -9.4); %A_{1, 1, 12} & %A_{1, 1, 13}
\draw[black, ultra thin, fill=white] (12.8, -9.4) rectangle (13.4, -8.8); %A_{1, 1, 14}
%
%
%
%\mathcal A_{1,0} boxes
\draw[black, ultra thin, fill=white] (8, -8.2) rectangle (8.6, -7.6); %A_{1, 0, 1}
\draw[black, ultra thin, fill=white] (8.6, -7) rectangle (9.2, -6.4); %A_{1, 0, 2}
\draw[black, ultra thin, fill=white](9.8, -10) rectangle (10.4, -9.4); %A_{1, 0, 3}
\draw[black, ultra thin, fill=white] (9.2, -7) rectangle (9.8, -6.4); %A_{1, 0, 4}
\draw[black, ultra thin, fill=white] (12.8, -8.2) rectangle (13.4, -6.4); %A_{1, 0, 5} & %A_{1, 0, 6} %A_{1, 0, 7}

%remove dashes
\draw[white, thin] (9.2, -6.4) -- (9.8, -6.4);
\draw[white, thin] (11, -6.4) -- (11.6, -6.4);
\draw[white, thin] (10.4, -7.6) -- (11, -7.6);
\draw[white, thin] (11, -7) -- (11.6, -7);
\draw[white, thin] (8, -8.2) -- (8.6, -8.2);
\draw[white, thin] (8, -7.6) -- (8.6, -7.6);
\draw[white, thin] (8.6, -7) -- (9.2, -7);
\draw[white, thin] (8.6, -6.4) -- (9.2, -6.4);
\draw[white, thin] (9.2, -7) -- (9.2, -6.4);
\draw[white, thin] (9.8, -10) -- (9.8, -9.4);
\draw[white, thin] (10.4, -10) -- (10.4, -9.4);
\draw[white, thin] (12.8, -8.2) -- (12.8, -7.6);
\draw[white, thin] (12.8, -7) -- (12.8, -6.4);

%\mathcal A_{2,2} boxes
\draw[black, ultra thin, fill=white] (8.6,-8) rectangle (8.8, -7.8); %A_{2, 2, 1}
\draw[black, ultra thin, fill=white] (9, -7.8) rectangle (9.2, -7.6); %A_{2, 2, 2}
\draw[black, ultra thin, fill=white] (8.2, -6.8) rectangle (8.4, -6.6); %A_{2, 2, 3}
\draw[white, ultra thin, fill=white] (8.4, -7) rectangle (8.6, -6.4); %A_{2, 2, 4} & %A_{2, 2, 5} & %A_{2, 2, 6}
\draw[black, ultra thin, fill=white] (9.4, -7.4) rectangle (9.6, -7.2); %A_{2, 2, 7}

\draw[black, ultra thin, fill=white] (10,-8.8) rectangle (10.4, -8.4); %A_{2, 2, 7} & %A_{2, 2, 8} & A_{2, 2, 9} & %A_{2, 2, 10}
\draw[black, ultra thin, fill=white] (10.6, -9.2) rectangle (10.8, -9); %A_{2, 2, 11}
\draw[black, ultra thin, fill=white](10.2, -8) rectangle (10.4, -7.8); %A_{2, 2, 12}
\draw[black, ultra thin, fill=white] (11.2, -9) rectangle (11.4, -8.8); %A_{2, 2, 13}
\draw[black, ultra thin, fill=white] (11, -9.6) rectangle (11.6, -9.4); %A_{2, 2, 14} & %A_{2, 2, 15} & A_{2, 2, 16}

\draw[black, ultra thin, fill=white] (11.8, -7) rectangle (12.2, -6.6); %A_{2, 2, 17} & %A_{2, 2, 18} & A_{2, 2, 19} & %A_{2, 2, 20}
\draw[black, ultra thin, fill=white] (12.4, -7.4) rectangle (12.6, -7.2); %A_{2, 2, 21}
\draw[black, ultra thin, fill=white] (13.2, -8.4) rectangle (13.4, -8.2); %A_{2, 2, 22}
\draw[black, ultra thin, fill=white] (11.8, -7.2) rectangle (12, -7); %A_{2, 2, 23}
\draw[black, ultra thin, fill=white] (11.6, -7.8) rectangle (12.2, -7.6); %A_{2, 2, 24} & %A_{2, 2, 25} & A_{2, 2, 26}
\draw[black, ultra thin, fill=white] (12, -5.2) rectangle (12.2, -5);
%A_{2, 2, 27}
\draw[black, ultra thin, fill=white] (11.6, -5.6) rectangle (11.8, -5.4); %A_{2, 2, 28}
\draw[black, ultra thin, fill=white] (12.2, -6) rectangle (12.4, -5.8); %A_{2, 2, 29}
\draw[black, ultra thin, fill=white] (12.8, -6) rectangle (13.2, -5.8); %A_{2, 2, 30} & %A_{2, 2, 31}

\draw[black, ultra thin, fill=white] (10.2, -7.4) rectangle (10.4, -6.4); %A_{2, 2, 32} to %A_{2, 2, 36}
\draw[black, ultra thin, fill=white] (10.4, -7) rectangle (10.6, -6.8); %A_{2, 2, 37}

\draw[black, ultra thin, fill=white] (11.8, -9.2) rectangle (12.6, -8.4); %A_{2, 2, 38} to %A_{2, 2, 53}

%remove dashes/lines
\draw[black, ultra thin] (8.4, -7) -- (8.4, -6.8);
\draw[white, thin] (8.6, -8) -- (8.6, -7.8);
\draw[white, thin] (9, -7.6) -- (9.2, -7.6);
\draw[white, thin] (10.2, -6.4) -- (10.4, -6.4);
\draw[white, thin] (10.4, -7.4) -- (10.4, -6.8);
\draw[white, thin] (10.4, -7) -- (10.6, -7);
\draw[white, thin] (10.4, -8) -- (10.4, -7.8);
\draw[white, thin] (11, -9.6) -- (11, -9.4);
\draw[white, thin] (11.6, -9.6) -- (11.6, -9.4);
\draw[white, thin] (11.6, -5.6) -- (11.6, -5.4);
\draw[white, thin] (11.2, -8.8) -- (11.4, -8.8);
\draw[white, thin] (12.2, -7) -- (12.2, -6.6);
\draw[white, thin] (11.8, -7) -- (12, -7);
\draw[white, thin] (12.2, -7.8) -- (12.2, -7.6);
\draw[white, thin] (13.2, -8.2) -- (13.4, -8.2);

%\mathcal A_{2,1} boxes
\draw[red, fill=white] (8.2,-6.6) rectangle (8.4, -6.4); %A_{2, 1, 1}
\draw[red, fill=white] (9.2, -7.6) rectangle (9.4, -7); %A_{2, 1, 2} & %A_{2, 1, 3} & %A_{2, 1, 4}
\draw[red] (9.2, -7.4) -- (9.4, -7.4);
\draw[red] (9.2, -7.2) -- (9.4, -7.2);

\draw[red, fill=white] (10.6, -7) rectangle (11, -6.6); %A_{2, 1, 5}  to %A_{2, 1, 8}
\draw[red] (10.6, -6.8) -- (11, -6.8);
\draw[red] (10.8, -7) -- (10.8, -6.6);

\draw[red, fill=white] (9.8, -9.4) rectangle (10.4, -9.2); %A_{2, 1, 9}  to %A_{2, 1, 11}
\draw[red] (10, -9.4) -- (10, -9.2);
\draw[red] (10.2, -9.4) -- (10.2, -9.2);
\draw[red, fill=white] (11.4,-9.8) rectangle (11.6, -9.6); %A_{2, 1, 12}

\draw[red, fill=white] (12.8, -5.4) rectangle (13.4, -5.2); %A_{2, 1, 13}  to %A_{2, 1, 15}
\draw[red] (13, -5.4) -- (13, -5.2);
\draw[red] (13.2, -5.4) -- (13.2, -5.2);

\draw[red, fill=white] (12.6, -7.6) rectangle (12.8, -7); %A_{2, 1, 16} & %A_{2, 1, 17} & %A_{2, 1, 18}
\draw[red] (12.6, -7.2) -- (12.8, -7.2);
\draw[red] (12.6, -7.4) -- (12.8, -7.4);
\draw[red, fill=white] (12, -7.2) rectangle (12.2, -7); %A_{2, 1, 19}

\draw[red, fill=white] (12.8,-8.6) rectangle (13.2, -8.2); %A_{2, 1, 20} to A_{2, 1, 23}
\draw[red] (12.8, -8.4) -- (13.2, -8.4);
\draw[red] (13, -8.6) -- (13, -8.2);

%B box
\draw (8,-10) rectangle (13.4, -4.6); %B

%
%%box labels
%%\node at (23, 4.1) {$B$}; %B

%titles/explanation
\node at (10.9, -4.1) {\underbar{Step $3(ii)$}};
\node[align=center] at (10.9, -11.6) {Next, for each $A' \in \mathcal B_1$ any \\ boxes  in $\mathcal A_{2,1}(A')$  \\  from $ \frac 1 {R^2} \mathcal B_1$ are removed};

%--------------
%step 3(iii)

%B grids
\draw[help lines, gray, ultra thin, densely dotted, step=0.2, xshift=-2cm] (18,-10) grid (23.4,-4.6);
%\draw[help lines, black, very thin, step=1.8, xshift=-1cm, yshift=-1cm] (17,-9) grid (22.4,-3.6);
%%dash removal \mathcal A_{0,0}
%%\draw[help lines, white,  step=0.6, xshift=-1cm] (9, 3.6) grid (12.6,5.4);
%%\draw[gray, very thin, dashed] (8, 3.6) -- (13.6, 3.6);
%%\draw[help lines, white,  step=0.6, xshift=-1cm] (9,0) grid (10.04,2);
%%dash removal \mathcal A_{1,0}
%%\draw[help lines, white,  step=0.67, xshift=0.96cm] (12.04,0) grid (14.04,2);
%%\draw[gray, very thin, dashed] (13, 0) -- (13, 2);

%\mathcal A_{0,0} boxes
\draw[black, ultra thin, fill=white] (16,-10) rectangle (17.8, -8.2); %A_{0, 1}
\draw[black, ultra thin, fill=white] (16, -6.4) rectangle (19.6, -4.6); %A_{0, 2} & A_{0, 3}
%
%
%
%\mathcal A_{1,1} boxes
\draw[black, ultra thin, fill=white] (18.4,-10) rectangle (19, -9.4); %A_{1, 1, 1}
\draw[black, ultra thin, fill=white] (19, -8.8) rectangle (19.6, -8.2); %A_{1, 1, 2}
\draw[black, ultra thin, fill=white] (16, -7.6) rectangle (17.2, -7); %A_{1, 1, 3} & %A_{1, 1, 4}
\draw[black, ultra thin, fill=white] (18.4, -7.6) rectangle (19.6, -7); %A_{1, 1, 5} & %A_{1, 1, 6}
\draw[black, ultra thin, fill=white] (19, -7) rectangle (19.6, -6.4); %A_{1, 1, 7}
\draw[black, ultra thin, fill=white] (18.4, -8.2) rectangle (19, -7.6); %A_{1, 1, 8}
\draw[black, ultra thin, fill=white] (20.2, -8.2) rectangle (20.8, -7.6); %A_{1, 1, 9}
\draw[black, ultra thin, fill=white] (20.2, -7) rectangle (20.8, -6.4); %A_{1, 1, 10}
\draw[black, ultra thin, fill=white] (20.8, -5.2) rectangle (21.4, -4.6); %A_{1, 1, 11}
\draw[black, ultra thin, fill=white] (19.6, -10) rectangle (20.8, -9.4); %A_{1, 1, 12} & %A_{1, 1, 13}
\draw[black, ultra thin, fill=white] (20.8, -9.4) rectangle (21.4, -8.8); %A_{1, 1, 14}
%
%
%
%\mathcal A_{1,0} boxes
\draw[black, ultra thin, fill=white] (16, -8.2) rectangle (16.6, -7.6); %A_{1, 0, 1}
\draw[black, ultra thin, fill=white] (16.6, -7) rectangle (17.2, -6.4); %A_{1, 0, 2}
\draw[black, ultra thin, fill=white](17.8, -10) rectangle (18.4, -9.4); %A_{1, 0, 3}
\draw[black, ultra thin, fill=white] (17.2, -7) rectangle (17.8, -6.4); %A_{1, 0, 4}
\draw[black, ultra thin, fill=white] (20.8, -8.2) rectangle (21.4, -6.4); %A_{1, 0, 5} & %A_{1, 0, 6} %A_{1, 0, 7}

%remove dashes
\draw[white, thin] (17.2, -6.4) -- (17.8, -6.4);
\draw[white, thin] (19, -6.4) -- (19.6, -6.4);
\draw[white, thin] (18.4, -7.6) -- (19, -7.6);
\draw[white, thin] (19, -7) -- (19.6, -7);
\draw[white, thin] (16, -8.2) -- (16.6, -8.2);
\draw[white, thin] (16, -7.6) -- (16.6, -7.6);
\draw[white, thin] (16.6, -7) -- (17.2, -7);
\draw[white, thin] (16.6, -6.4) -- (17.2, -6.4);
\draw[white, thin] (17.2, -7) -- (17.2, -6.4);
\draw[white, thin] (17.8, -10) -- (17.8, -9.4);
\draw[white, thin] (18.4, -10) -- (18.4, -9.4);
\draw[white, thin] (20.8, -8.2) -- (20.8, -7.6);
\draw[white, thin] (20.8, -7) -- (20.8, -6.4);

%
%
%\mathcal A_{2,2} boxes
\draw[black, ultra thin, fill=white] (16.6,-8) rectangle (16.8, -7.8); %A_{2, 2, 1}
\draw[black, ultra thin, fill=white] (17, -7.8) rectangle (17.2, -7.6); %A_{2, 2, 2}
\draw[black, ultra thin, fill=white] (16.2, -6.8) rectangle (16.4, -6.6); %A_{2, 2, 3}
\draw[white, ultra thin, fill=white] (16.4, -7) rectangle (16.6, -6.4); %A_{2, 2, 4} & %A_{2, 2, 5} & %A_{2, 2, 6}
\draw[black, ultra thin, fill=white] (17.4, -7.4) rectangle (17.6, -7.2); %A_{2, 2, 7}

\draw[black, ultra thin, fill=white] (18,-8.8) rectangle (18.4, -8.4); %A_{2, 2, 7} & %A_{2, 2, 8} & A_{2, 2, 9} & %A_{2, 2, 10}
\draw[black, ultra thin, fill=white] (18.6, -9.2) rectangle (18.8, -9); %A_{2, 2, 11}
\draw[black, ultra thin, fill=white](18.2, -8) rectangle (18.4, -7.8); %A_{2, 2, 12}
\draw[black, ultra thin, fill=white] (19.2, -9) rectangle (19.4, -8.8); %A_{2, 2, 13}
\draw[black, ultra thin, fill=white] (19, -9.6) rectangle (19.6, -9.4); %A_{2, 2, 14} & %A_{2, 2, 15} & A_{2, 2, 16}

\draw[black, ultra thin, fill=white] (19.8, -7) rectangle (20.2, -6.6); %A_{2, 2, 17} & %A_{2, 2, 18} & A_{2, 2, 19} & %A_{2, 2, 20}
\draw[black, ultra thin, fill=white] (20.4, -7.4) rectangle (20.6, -7.2); %A_{2, 2, 21}
\draw[black, ultra thin, fill=white] (21.2, -8.4) rectangle (21.4, -8.2); %A_{2, 2, 22}
\draw[black, ultra thin, fill=white] (19.8, -7.2) rectangle (20, -7); %A_{2, 2, 23}
\draw[black, ultra thin, fill=white] (19.6, -7.8) rectangle (20.2, -7.6); %A_{2, 2, 24} & %A_{2, 2, 25} & A_{2, 2, 26}
\draw[black, ultra thin, fill=white] (20, -5.2) rectangle (20.2, -5);
%A_{2, 2, 27}
\draw[black, ultra thin, fill=white] (19.6, -5.6) rectangle (19.8, -5.4); %A_{2, 2, 28}
\draw[black, ultra thin, fill=white] (20.2, -6) rectangle (20.4, -5.8); %A_{2, 2, 29}
\draw[black, ultra thin, fill=white] (20.8, -6) rectangle (21.2, -5.8); %A_{2, 2, 30} & %A_{2, 2, 31}

\draw[black, ultra thin, fill=white] (18.2, -7.4) rectangle (18.4, -6.4); %A_{2, 2, 32} to %A_{2, 2, 36}
\draw[black, ultra thin, fill=white] (18.4, -7) rectangle (18.6, -6.8); %A_{2, 2, 37}

\draw[black, ultra thin, fill=white] (19.8, -9.2) rectangle (20.6, -8.4); %A_{2, 2, 38} to %A_{2, 2, 53}

%\mathcal A_{2,1} boxes
\draw[white, ultra thin, fill=white] (16.2,-6.6) rectangle (16.4, -6.4); %A_{2, 1, 1}
\draw[black, ultra thin, fill=white] (17.2, -7.6) rectangle (17.4, -7); %A_{2, 1, 2} & %A_{2, 1, 3} & %A_{2, 1, 4}
\draw[black, ultra thin, fill=white] (18.6, -7) rectangle (19, -6.6); %A_{2, 1, 5}  to %A_{2, 1, 8}
\draw[black, ultra thin, fill=white] (17.8, -9.4) rectangle (18.4, -9.2); %A_{2, 1, 9}  to %A_{2, 1, 11}
\draw[black, ultra thin, fill=white] (19.4,-9.8) rectangle (19.6, -9.6); %A_{2, 1, 12}
\draw[black, ultra thin, fill=white]  (20.8, -5.4) rectangle (21.4, -5.2); %A_{2, 1, 13}  to %A_{2, 1, 15}
\draw[white, ultra thin, fill=white] (20.6, -7.6) rectangle (20.8, -7); %A_{2, 1, 16} & %A_{2, 1, 17} & %A_{2, 1, 18}
\draw[black, ultra thin, fill=white] (20, -7.2) rectangle (20.2, -7); %A_{2, 1, 19}
\draw[black, ultra thin, fill=white] (20.8,-8.6) rectangle (21.2, -8.2); %A_{2, 1, 20} to A_{2, 1, 23}

%remove dashes/lines
\draw[black, ultra thin,] (16.4, -7) -- (16.4, -6.8);
\draw[black, ultra thin,] (16.2, -6.6) -- (16.2, -6.4);
\draw[white, thin] (16.6, -8) -- (16.6, -7.8);
\draw[white, thin] (17, -7.6) -- (17.2, -7.6);
\draw[white, thin] (18.2, -6.4) -- (18.4, -6.4);
\draw[white, thin] (18.4, -7.4) -- (18.4, -6.8);
\draw[white, thin] (18.4, -7) -- (18.6, -7);
\draw[white, thin] (18.4, -8) -- (18.4, -7.8);
\draw[white, thin] (19, -9.6) -- (19, -9.4);
\draw[white, thin] (19.6, -9.6) -- (19.6, -9.4);
\draw[white, thin] (19.6, -5.6) -- (19.6, -5.4);
\draw[white, thin] (19.2, -8.8) -- (19.4, -8.8);
\draw[white, thin] (20.2, -7) -- (20.2, -6.6);
\draw[white, thin] (19.8, -7) -- (20, -7);
\draw[white, thin] (20.2, -7.8) -- (20.2, -7.6);
\draw[white, thin] (21.2, -8.2) -- (21.4, -8.2);

\draw[white, thin] (17.2, -7.6) -- (17.2, -7);
\draw[white, thin] (17.2, -7) -- (17.4, -7);
\draw[white, thin] (17.4, -7.4) -- (17.4, -7.2);
\draw[white, thin] (17.8, -9.4) -- (17.8, -9.2);
\draw[white, thin] (17.8, -9.4) -- (18.4, -9.4);
\draw[white, thin] (18.6, -7) -- (18.6, -6.8);
\draw[white, thin] (18.6, -7) -- (19, -7);
\draw[white, thin] (19, -7) -- (19, -6.6);
\draw[white, thin] (19.4, -9.6) -- (19.6, -9.6);
\draw[white, thin] (19.6, -9.6) -- (19.6, -9.8);
\draw[white, thin] (20.8, -5.2) -- (21.4, -5.2);
\draw[white, thin] (20, -7.2) -- (20, -7);
\draw[white, thin] (20, -7) -- (20.2, -7);
\draw[black, ultra thin,] (20.6, -7) -- (20.6, -7.2);
\draw[black, ultra thin,] (20.6, -7.4) -- (20.6, -7.6);
\draw[white, thin] (20.8, -8.2) -- (21.2, -8.2);
\draw[white, thin] (21.2, -8.2) -- (21.2, -8.4);

%B box
\draw (16,-10) rectangle (21.4, -4.6); %B

%\mathcal A_{2,0} boxes
\draw[red, fill=white] (16.2,-7) rectangle (16.4, -6.8); %A_{2, 0, 1}
\draw[red, fill=white] (16, -7) rectangle (16.2, -6.4); %A_{2, 0, 2} & %A_{2, 0, 3} & %A_{2, 0, 4}
\draw[red] (16, -6.8) -- (16.2, -6.8);
\draw[red] (16, -6.6) -- (16.2, -6.6);

\draw[red, fill=white] (19, -10) rectangle (19.4, -9.6); %A_{2, 0, 5} to %A_{2, 0, 8}
\draw[red] (19, -9.8) -- (19.4, -9.8);
\draw[red] (19.2, -10) -- (19.2, -9.6);

\draw[red, fill=white] (17.8,-9.2) rectangle (18, -9);  %A_{2, 0, 9}

%
%
%%box labels
%%\node at (23, 4.1) {$B$}; %B

%titles/explanation
\node at (18.9, -4.1) {\underbar{Step $3(iii)$}};
\node[align=center] at (18.9, -11.6) {Finally, any boxes in  $\mathcal A_{2,0}$ \\ from $ \frac 1 {R^3} B$ are removed, leaving \\ the collection $\mathcal B_3$};

\end{tikzpicture}
}% end resizebox
\caption{Construction of a Cantor winning set % in $\mathbb R^2$ (3 steps, red indicating removal)
	}
		\label{fig:cantorWinning}
\end{figure}

%For each ball $B_n\in \BBB_n$ we remove at
%   most $r_{n,n}$ balls $B_{n+1}\in\frac 1 R\{ B_n \}$ from $\frac 1 R\BBB_n$. Let
%   $\BBB_{n+1}^n\sub \frac 1 R\BBB_n$ be the collection of balls that
%   remain. Next, for each ball $B_{n-1}\in \BBB_{n-1}$ we remove at
%   most $r_{n-1,n}$ balls $B_{n+1}\in \frac 1 {R^2}\{ B_{n-1} \}\cap
%   \BBB_{n+1}^n$. Let $\BBB_{n+1}^{n-1}$ be the collection of balls
%   that remain. In general for each $B_{n-k}\in\BBB_{n-k}$ ($1\le k \le
%   n$) we remove at most $r_{n-k,n}$ balls $B_{n+1}\in
%   \frac 1 {R^{k+1}}\{ B_{n-k} \}\cap
%   \BBB_{n+1}^{n-k+1}$ and
%   define $\BBB_{n+1}^{n-k}\sub\BBB_{n+1}^{n-k+1}$ to be the
    %collection of balls that remain. Finally, $\BBB_{n+1}:=
    %\BBB_{n+1}^0$ then becomes the desired collection of (level $n+1$) survivors.

Define the \emph{limit set} of such a construction to be 
\begin{equation}\label{eq:limitset}
\bigcap_{n=0}\bigcup_{B\in\calb_n}B
\end{equation}
and denote it by $\KKK(B_0, R,\vr)$. %, and call it a \textit{generalised $(B_0,R,\vr)$-Cantor set}.
Note that the triple $(B_0,R,\vr)$ does not uniquely determine $\KKK(B_0,R,\vr)$ as there is a large degree of freedom in the sets of balls $\AAA_{m,n}$ removed in the construction procedure. We say a set $\KKK$ is a \textit{generalised Cantor set} if it can be constructed by the procedure described above for some triple $(B_0,R,\vr)$. In this case, we may refer to $\KKK$ as being \textit{$(B_0,R,\vr)$ Cantor} if we wish to make such a triple explicit and write $\KKK=\KKK(B_0,R,\vr)$.

\begin{example}
\label{examplebasic}
In the case that $r_{m,n}=0$ for every $m < n$, we refer to a $(B_0,R,\vr)$ Cantor set as a \textit{local Cantor set} and denote it simply by $\KKK(B_0,R, \left\{r_n\right\})$, where $r_n:=r_{n,n}$. The simplest class of local Cantor sets occur when the sequence $\left\{r_n\right\}$ is taken to be constant; i.e., $r_n:=r$. Such $(B_0,R, r)$ Cantor sets are well studied and their measure theoretic properties are well known.
%For example, the famous middle-third Cantor set in the unit interval $I$ is an $(I,3, 1)$ Cantor set.
\end{example}

\begin{example}
The %famous
middle-third Cantor set $\calk$ is an $(I,3,1)$ Cantor set in $\RR$, and $\calk^2$ is a $\left(I^2,3,5\right)$ Cantor set in $\RR^2$.
\end{example}

%\begin{fact}\label{regCandim}
%   For any $R \geq \left\lfloor r\right\rfloor$, the Hausdorff dimension of a $\left(B_{0},R,r\right)$ Cantor set $\KKK$
%   in $\RR^N$ (with respect to the canonical splitting structure) satisfies
%   \[
%   \frac{\log\left(R^N-\left\lfloor r\right\rfloor \right)}{\log R} \:
%   \leq \: \dim{\KKK} \: \leq \:N.
%   \]
%   If during the construction of $\KKK$ each step is comprised of the removal of the maximum possible number $\left\lfloor r\right\rfloor $ of intervals, then the dimension of $\KKK$ coincides with the lower bound.

    %Furthermore (see e.g., Theorem $5$ of \cite{BadziahinHarrap}), for a generic splitting structure $(X,\SSS, U,f)$ on a doubling\footnote{A metric space $X$ is \emph{doubling} if there exists a constant $N\in\NN$ such that every ball of radius $2\rho$ can be covered by at most $N$ balls of radius $\rho$.} metric space we have
    %   \[
    %   \frac{\log\left(f(R)}{\log R} - \log_R 2\:
    %   \leq \: \dim{\KKK\left(B_{0},R,r\right)} \: \leq \:     \frac{\log f(R)}{\log R}
    %   \]
    %   so long as $f(R) is sufficiently large$.
%\end{fact}

\begin{defn}\label{def:cantorWinning}
Fix a ball $B_0 \sub X$. Given a parameter $\eps_0 \in (0,1]$ we say a
set $E \sub X$ is  \emph{$\eps_0$ Cantor winning on $B_0$ (with respect to the splitting structure $(X, \SSS, U, f)$)} if for every
$0<\eps<\eps_0$ there exists $R_\eps \in U$ such that for every
$R_\eps\leq R\in U$ the set $E$ contains
%some
a $(B_0,R,\vr)$ Cantor set satisfying
\begin{equation*}
    r_{m,n} \:\leq \: f(R)^{(n-m+1)(1-\eps)}\quad\mbox{for every }m,n\in\NN,\: m\le n.
\end{equation*}

If a set $E \sub X$ is $\eps_0$ Cantor winning on~$B_0$ for every
ball $B_0\sub X$ then we say $E$ is \emph{$\eps_0$ Cantor winning}, and simply that $E$ is \emph{Cantor winning} if
it is $\eps_0$ Cantor winning for some $\eps_0 \in (0,1]$.
\end{defn}

It is obvious that if a set is $\eps_0$ Cantor winning then it is also $\eps$ Cantor winning for any $\eps \in (0, \eps_0)$. %i.e.,
In particular, the property of being  $1$ Cantor winning is the strongest possible Cantor winning property.

\begin{rem}
As for the generalised Cantor sets, whenever we mention Cantor winning sets
not in the context of a splitting structure we will refer to a standard
splitting structure. In particular, this applies to all theorems in
Section~\ref{newresults} where the standard splitting structure in $\RR^N$ is
considered. Some of those results remain true in a more general context of
complete metric spaces.
\end{rem}

\subsection{Cantor rich}

The class of \textit{Cantor rich} subsets of $\RR$ was introduced in~\cite{Beresnevich_BA}. As was noted in \cite{BadziahinHarrap}, this notion can be generalised to the context of metric spaces endowed with a splitting structure:
%. However, we will show in \S\ref{sec:rich} that this notion is equivalent to Cantor winning.
%\comdavid{Is there any particular reason for this? We don't seem to be using the assumption that $X=\R$ anywhere.}
%\comstephen{I think it's simply because Victor first introduced them as subsets of $\RR$. Presumably we can easily just extend to X with splitting structure, though might be worth checking which doubling/(S4)-type condition is most appropriate and guarantees (W1) - the proof of equivalence later doesn't seem to use $\RR$ I agree}
%\comerez{Even with David's improvement of Theorem \ref{thm:potentialCantor} where we no longer need to assume that $X$ is doubling, I couldn't easily generalise Corollary \ref{cor:cantor-winning-implies-rich} to general metric spaces. If we assume doubling, then both the definition of Cantor richness and the relation between $M$ and $\eps$ seem to depend on $\delta$. Since the main point of this corollary is that there was no good reason to define Cantor richness, I think we should do history a favor and undefine Cantor richness. The current applications of richness are in Victor's and Lei's papers, where its use contributes significantly to the difficulty of reading those papers. Cantor winning should be used everywhere instead - it is a much nicer and more practical property}

\begin{defn}
Let $M\geq 4$ be a real number and $B_0$ be a ball in
%$\RR$.
$X$.
A set
%$E\sub\bbr$
$E\sub X$
is \emph{$\left(B_{0},M\right)$ Cantor rich} if for every $R$ such that
\begin{equation}\label{eq:RinMrich}
f(R) > M
\end{equation}
and $0 < y < 1$, $E$ contains a $\left(B_{0},R,\br\right)$
Cantor set, where $\br = (r_{m,n})_{0\leq m\leq n}$ is a two-parameter sequence satisfying
\[
%\sum_{m=0}^{n}\left(\frac{4}{R}\right)^{n-m+1}r_{m,n}\leq
\sum_{m=0}^{n}\left(\frac{4}{f(R)}\right)^{n-m+1}r_{m,n}\leq
y\quad\mbox{ for every }n\geq0.
\]
$E$ is \emph{Cantor rich in $B_{0}$} if it is $\left(B_{0},M\right)$ Cantor rich in
$B_{0}$ for some $M\geq 4$. $E$ is \emph{Cantor rich} if it is Cantor
rich in some $B_{0}$.
\end{defn}

Note that, following \cite{Beresnevich_BA}, Cantor rich sets are Cantor rich in some ball, while Cantor winning sets are Cantor winning in every ball.
Also note that unlike in \cite{Beresnevich_BA}, the inequality \eqref{eq:RinMrich} is strict. This enables a clearer correspondence with the parameter $\eps_0$ in Definition \ref{def:cantorWinning} (cf. Corollary \ref{cor:cantor-winning-implies-rich}). %\comdavid{What is the logic behind this difference?}

\subsection{Diffuse sets and doubling spaces}\label{sec:doublingAndDiffuse}
%,  and Ahlfors regularity dimension
Schmidt's game was originally played on a complete metric space. However, in order to conclude that winning sets have properties such as large Hausdorff dimension, certain assumptions on the underlying space are required. The first time that large intersection with certain subspaces of $\bbr^n$ was proved using Schmidt's game was in \cite{Fishman}. Later it was realised in \cite{BFKRW} that more generally some spaces form a natural playground for the absolute game. The following notion generalises the definition of a diffuse set in $\bbr^n$ as appeared in \cite{BFKRW}.
%This notion is in fact equivalent to the following The definition of diffuse sets was first given in \cite{BFKRW} in the case $X=\bbr^{d}$.

\begin{defn}\label{def:diffuse}
For $0<\beta<1$, a closed set $K\sub X$ is \emph{$\beta$ diffuse} if
there exists $r_{0}>0$ such that for every $x\in K$, $y\in X$ and
$0<r\leq r_{0}$, we have
\begin{equation}
\big(B(x,(1-\beta)r)\setminus B(y,2\beta r)\big)\cap K\neq\emptyset.\label{eq:diffuse}
\end{equation}
If $K$ is $\beta$ diffuse for some $\beta$ then
we say it is \emph{diffuse}.
\end{defn}

Condition (\ref{eq:diffuse}) is satisfied if and only if there exists $z\in K$ such that
\[
B\left(z,\beta r\right)\sub B\left(x,r\right)\setminus B\left(y,\beta r\right).
\]
The definition of diffuse sets is designed so that the $\beta$ absolute game on a $\beta$ diffuse metric space cannot end after finitely many steps (as long as the first ball of Bob is small enough). So, if Alice plays the game in such way that she is able to force the outcome to lie in the target set, then the reason for her victory is her strategy and not the limitations of the space in which the game is played. %\comstephen{Reworded.} \comdavid{Reworded again, switched back to present tense (I am not sure why you wanted to switch to past tense?)}

\begin{rem}\label{rem:diffuse}
In the definition that appears in \cite{BFKRW}, sets as in Definition \ref{def:diffuse} are called $0$-dimensionally diffuse. Moreover, equation (\ref{eq:diffuse})
is replaced by
\begin{equation}
\big(B\left(x,r\right)\setminus B\left(y,\beta r\right)\big)\cap K\neq\emptyset. \label{eq:diffuse2}
\end{equation}
Note that
\[
B\left(x,r\right)\setminus B\left(y,\beta r\right)=B\left(x,\left(1-\beta'\right)\rho\right)\setminus B\left(y,2\beta'\rho\right),
\]
for $\beta'=\frac{\beta}{2+\beta}$ and $\rho=\left(1+\frac{\beta}{2}\right)r$,
so both (\ref{eq:diffuse}) and (\ref{eq:diffuse2}) give rise to
the same class of diffuse sets. We will use (\ref{eq:diffuse}) and
(\ref{eq:diffuse2}) interchangeably without further notice.
\end{rem}

A more standard definition that is equivalent to diffuseness (see \cite[Lemma 4.3]{FSU4}) is the following:
\begin{defn}\label{def:uniformlyPerfect}
For $0<c<1$, a closed set $K\sub X$ is \emph{$c$ uniformly perfect}
if there exists $r_{0}>0$ such that for every $x\in K$ and $0<r\leq
r_{0}$ we have
\begin{equation}
\big(B(x,r)\setminus B(x,cr)\big)\cap K\neq\emptyset.\label{eq:uniformlyPerfect}
\end{equation}
If $K$ is $c$ uniformly perfect for some $c$ then we say it is \emph{uniformly perfect}.
\end{defn}

Being diffuse is equivalent to having no singletons as microsets
\cite[Lemma 4.4]{BFKRW}. Instead of expanding the discussion on
microsets, we refer the interested reader to \cite{Furstenberg4},
and mention that Ahlfors regular sets are diffuse.
In the other direction, diffuse sets support Ahlfors regular measures (see Proposition \ref{propositiondiffusedimR} below).
%suggest another equivalent property. It turns out that diffuse
%sets must have positive Hausdorff dimension. In fact, they must
%support an Ahlfors regular measure.
We will now give the relevant definitions and provide the results which support the statements
above. %\comerez{reworded}

\begin{defn}
A Borel measure $\mu$ on $X$ is called \emph{Ahlfors regular} if there exist $C,\delta,r_{0}>0$ such
that every $x\in X$ and $0<r\leq r_{0}$ satisfy
\begin{equation*}
\frac{1}{C}r^\delta\leq\mu\left(B\left(x,r\right)\right)\leq Cr^\delta.%\label{eq:Ahlfors regular}
\end{equation*}
In this case we call $\delta$ the \emph{dimension} of $\mu$ and say that
$\mu$ is \emph{Ahlfors regular of dimension $\delta$}.
\end{defn}

We say that a closed set  $K\sub X$ is \emph{Ahlfors regular} if it is the support of an Ahlfors
regular measure. More exactly, $K\sub X$ is \emph{$\delta$ Ahlfors regular} if it
is the support of an Ahlfors regular measure of dimension~$\delta$. This
terminology gives rise to a dimensional property of $K$:

\begin{defn}
The \emph{Ahlfors regularity dimension} of $K\sub X$ is
\begin{equation*}
\dim_R K=\sup \left\{\delta \sep \text{$K$ supports an Ahlfors regular measure of dimension $\delta$}\right\}. %\label{eq:Ahlfors regular dimension}
\end{equation*}
\end{defn}

Note that it is a direct consequence of the Mass Distribution Principle
(see~\cite{Falconer_book}) that $\dim_R K\leq \dim_H K$, where
$\dim_H K$ stands for the \emph{Hausdorff dimension} of $K$.

\begin{proposition}
\label{propositiondiffusedimR}
A diffuse set $K\sub X$ satisfies $\dim_R K>0$. Conversely, if $\dim_R K > 0$ then $K$ contains a diffuse set.
\end{proposition}
\begin{proof}
Suppose $K$ is $\beta$ diffuse; then every ball in $K$ of radius $\rho$ contains two disjoint balls of radius $\beta \rho$. Construct a Cantor set by recursively replacing every ball of radius $\rho$ with two disjoint subballs of radius $\beta\rho$. Standard calculations show that this Cantor set is Ahlfors regular of dimension $\frac{\log 2}{-\log \beta} > 0$, so $\dim_R K > 0$.

Suppose $\dim_R K > 0$; then $K$ contains an Ahlfors regular subset of positive dimension, which is
%clearly
diffuse (cf. e.g. \cite[Lemma 5.1]{BFKRW}).
\end{proof}

The following property of Ahlfors regular sets will be used later in the paper.

\begin{proposition}\label{prop_ahlfors}
If $(X,\mu)$ is $\delta$ Ahlfors regular then there exists $c>0$ such that
for any ball $B\sub X$ and any $\alpha<1$ there exist at least
$c\alpha^{-\delta}$ balls $B_i=B(x_i,\alpha\cdot \rad(B))\sub B$ $(i\in \{1,\ldots, I\},
I\ge c\alpha^{-\delta})$
%of radius $\alpha\cdot \rad(B)$, and separated in the following sense:
such that $\vd (x_i,x_j)\ge 3\alpha\cdot\rad(B)$ for any $1\le i\neq j\le I$.
\end{proposition}

\begin{proof}
It is enough to prove the proposition for all $\alpha\leq1/2$.
%then the proposition is trivial. Therefore, we may assume that $\alpha\le 1/2$.
Let $B'$ be a ball with the same centre as $B$, but with radius
$(1-\alpha)\rad(B)$, and let $D\sub X$ be an arbitrary ball of
radius $\rad(D) = 3\alpha\cdot \rad(B)$. Then:
%its $\delta$ Ahlfors regular measure can be estimated as:
$$
\mu (D)\le C\cdot (3\alpha)^\delta (\rad(B))^\delta \le C^2\cdot
\left(\frac{3\alpha}{1-\alpha}\right)^\delta \mu (B') \leq C_1
\alpha^\delta \mu (B').
$$
%assuming that $\alpha \leq 1/2$. \comdavid{$=$ changed to $\leq$, since $1-\alpha\geq 1/2$ but $1-\alpha$ is not itself constant. If $\alpha > 1/2$, then the proposition is trivial.}
Therefore, for any $k$ balls $D_i$ of radius $3\alpha \cdot \rad(B)$ with
$k<C_1^{-1} \alpha^{-\delta}$, there exists another ball~$D$ whose
centre lies inside $B'\backslash \bigcup_{i=1}^k D_i$. Hence there are
at least $I = \lceil C_1^{-1} \alpha^{-\delta}\rceil$ balls $D_i$
such that their centres lie inside $B'$ and are distanced by at
least $3 \alpha\cdot\rad(B)$ from each other. For each $1 \leq i \leq I$, let $B_i$ be a ball with the same centre as of $D_i$ but with the radius
$\alpha \cdot \rad(B)$. Clearly the balls $B_i$ satisfy all the conditions
of the proposition, where $c>0$ is chosen to ensure the inequality
$$
c\alpha^{-\delta} \le \lceil C_1^{-1} \alpha^{-\delta}\rceil
$$
for all $\alpha\in[0,1]$.
\end{proof}

Although winning sets are usually dense, the large dimension
property~(W1) for some types of winning sets depends on the space
$X$ being large enough. For example, absolute winning sets may be
empty if the Ahlfors regularity dimension is zero (cf. Remark
\ref{rem:emptysetIsWinning}). Various conditions are introduced in
this context; for example, see \cite[Section 11]{Schmidt1} and
\cite[Theorem 3.1]{Fishman}. More recently, the following condition
was introduced in \cite{BadziahinHarrap}:
\begin{itemize}
    \item[(S4)] There exists an absolute constant $C(X)$ such that any
    ball $B \in \BBB(X)$ cannot intersect more than $C(X)$ disjoint open balls of
    the same radius as $B$.
\end{itemize}

\begin{proposition}[{\cite[Corollary to Theorem
4.1]{BadziahinHarrap}}] If $X$ satisfies \textup{(S4)} then for any
$B \in \BBB(X)$ and any Cantor winning set $E\sub X$ we have
$$
\dim_H(E \cap A_\infty(B))=\dim_H A_\infty(B).
$$
\end{proposition}

Condition (S4) is similar to the following definition:

\begin{defn}\label{def:doubling}
A metric space $X$ is \emph{doubling} if there exists a constant $N\in\NN$ such that for every $r>0$, every ball of radius $2 r$ can be covered by a collection of at most $N$ balls of radius $r$.
\end{defn}

Doubling spaces satisfy condition (S4), but the converse does not hold. For example, in
an ultrametric space any two non-disjoint balls of the same radius are set-theoretically equal.
So an ultrametric space satisfies (S4) with $C=1$, but is not necessarily doubling.
%However, the doubling condition is more natural and it seems to be satisfied in all natural examples that satisfy (S4).
As shown in~\cite{BadziahinHarrap}, the space $X$ needs to
satisfy~(S4) for Cantor winning sets in $X$ to satisfy the large
dimension~(W1) and large intersection~(W2) conditions, and if
$X$ is doubling then Cantor winning sets satisfy (W3) with
respect to bi-Lipschitz homeomorphisms of $X$. This makes doubling spaces a natural setting for generalised Cantor sets,
and it is under this condition that many of the results of the next section hold.
%is under this condition that Theorem \ref{thm:potentialCantor} and Theorem \ref{potentialequiv} are proved.
%\comerez{Is this not true in ultrametric spaces?}\comdzmitry{ No!}
%%%%%% I comment this because we do not need it in the paper. And if we write this we need to provide some justification
%We point out that Ahlfors
%regularity implies doubling, but there are doubling spaces with
%Ahlfors regularity dimension zero.

%\begin{rem}
%A metric space $X$ is \emph{doubling} if there exists a constant $N\in\NN$ such that every ball of radius $2\rho$ can be covered by at most $N$ balls of radius $\rho$
%The condition that a metric space is doubling is similar to condition (S4) \comerez{ ADD DEFINITION OF DOUBLING}. It implies condition (S4), but it is not clear whether the reverse holds. However, the doubling condition is more natural and it seems to be satisfied in all natural examples that satisfy (S4).
%\end{rem}

%\comerez{ WRITE SOMETHING...}

\section{Connections}

\subsection{The Cantor game}\label{sec:cantor_game}

%The Cantor game is played on complete metric spaces that are endowed with a splitting structure. %Cantor games are the analogues of Schmidt games, when the players are confined to choosing balls from the splitting structure. The most natural among these games is what we will call the Cantor game.
To facilitate a more natural exposition,  we introduce a new game. The ``Cantor game'' described below is similar to the
games described earlier. The main difference is that this game is designed to be
played on a metric space  endowed with a splitting
structure. Our nomenclature is so chosen because (as we shall demonstrate) under natural
conditions the winning sets for this game are precisely the Cantor
winning sets defined in Definition \ref{def:cantorWinning}. %\comstephen{Reworded}

Let $X$ be a complete metric space and let $(X,\SSS, U,f)$ be a
splitting structure. Given $\eps\in(0,1]$, the \emph{$\eps$ Cantor
game} is defined as follows: Bob starts by choosing an integer
$2\leq R\in U$ and a closed ball $B_0\sub X$. For any $i\geq0$,
given Bob's $i$th choice $B_i$, Alice chooses a collection
$\cala_{i+1}\sub\frac{1}{R}\left\{B_i\right\}$ satisfying
\[
\#\cala_{i+1}\leq f(R)^{1-\eps}.
\]
Given $\cala_{i+1}$, Bob chooses a ball $B_{i+1}\in
\frac{1}{R}\left\{B_i\right\}\setminus\cala_{i+1}$, and the game continues. We say that
$E\sub X$ is \emph{winning for the $\eps$ Cantor game} if
Alice has a strategy which guarantees that (\ref{eq:outcome}) is
satisfied. The \emph{Cantor game} is defined by allowing Alice first
to choose $\eps\in(0,1]$ and then continuing as in the $\eps$ Cantor
game. We say that $E$ is \emph{winning for the Cantor game} if Alice has a winning strategy for the Cantor game. Note that $E$ is winning for the Cantor game if and only if it is winning for the $\eps$ Cantor game for some
$\eps\in(0,1]$. We say that $E$ is \emph{winning for the
$\eps$ Cantor game on $B_0$} if Alice has a winning strategy in the
$\eps$ Cantor game given that Bob's first ball is $B_0$.

\begin{rem}
A version of the Cantor game may be played in metric spaces with no prescribed splitting structure. Given $c\geq0$ and $0<\beta<1$, Bob will choose a sequence of decreasing balls with radii that shrink at rate $\beta$, while Alice will remove a collection of at most $\beta^{-c}$ closed balls of radius $\beta r$ where $r$ is the radius of Bob's previous choice, instead of just one such ball as in the absolute game. When $c=0$ this game coincides with the absolute game.
\end{rem}

It is clear that winning sets for the Cantor game contain Cantor
winning sets. Indeed, by the construction, for any $R\in U$ a
winning set for the $\eps$ Cantor game on $B_0$ contains a $(B_0, R,
f(R)^{1-\eps})$ Cantor set which is comprised of all Bob's possible moves while playing against a fixed winning strategy of Alice.
Therefore, it is $\eps$ Cantor winning on $B_0$.
%(cf. Example \ref{examplebasic})
%\comdavid{Is it worth putting in more details here?} \comstephen{I think it's OK, maybe just say at each step all the moves Bob could play comprises each $\mathcal{B}_i$ in local construction}
% (as long as the splitting structure is nontrivial).
We show that the converse is also true: %\comdavid{I think I removed the need for doubling by changing the definition of $\delta$ to depend on $R$...}

\begin{theorem}\label{thm:potentialCantor}
Let $E\sub X$. Let $B_{0}\sub X$ be
a closed ball and $0<\eps_0\leq1$. If $E$ is $\eps_0$ Cantor winning
on $B_0$ then $E$ is winning for the $\eps_0$ Cantor game on $B_0$.
\end{theorem}

\begin{proof}
We will define a winning strategy for Alice. Fix $R\geq2$, and let $\delta\geq 0$ be chosen so that $f(R) = R^\delta$. Choose $0<\eta<\eps_0$ so small that

\begin{equation*}%\label{eq:potentialEta}
    \frac{R^{\delta\left(1-\eps_0 + \eta\right)}}{\lfloor R^{\delta(1-\eps_0)}\rfloor+1} < 1.
\end{equation*}
%\comdavid{Note that I added a floor here, due to the fact that Alice must delete an integer number of balls}
 Set
\begin{align}
    \eps_1 &=\eps_0 - \frac{\eta}{2},\label{eq:potentialEps1}\\
    \eps_2 &=\eps_0 - \eta,\label{eq:potentialEps2}
\end{align}
and let $R_{\eps_1}$ be as in Definition \ref{def:cantorWinning}.
Then for any $\ell$ such that $R^\ell\geq R_{\eps_1}$, since $f\left(R^\ell\right)=f\left(R\right)^\ell=R^{\delta\ell}$, there exists a
$\left(B_{0},R^\ell,\left(R^{\left(n-m+1\right)\ell\delta(1-\eps_1)}\right)_{0\leq
m\leq n}\right)$ Cantor set in $E$. %\comdavid{changed $R$ to $R^\ell$, since $R$ is already fixed and so the previous version did not make sense}
Let $\ell$ be large enough so
that $R^{\ell}\geq R_{\eps_1}$,

\begin{equation}\label{eq:potentialEll1}
\left(\frac{R^{\delta\left(1-\eps_2\right)}}{\lfloor R^{\delta(1-\eps_0)}\rfloor+1}\right)^\ell<\frac{1}{2},
\end{equation}
and
\begin{equation}\label{eq:potentialEll2}
R^{-\ell\delta\frac{\eta}{2}}<\frac{1}{3} \cdot
\end{equation}
Such a value exists because of the choice of $\eta$.

Choose a $\left(B_{0},R^{\ell},\left(R^{\left(n-m+1\right)\ell\delta(1-\eps_1)}\right)_{0\leq m\leq n}\right)$ Cantor
set in $E$. Let $\cala=\left\{ \cala_{m,n}\right\} _{0\leq m\leq n}$
be the removed balls, so that $\calb_{n+1}=\frac{1}{R^{\ell}}\calb_{n}\setminus\bigcup_{m=0}^{n}\cala_{m,n}$ for all $n\geq0$, where
$\calb_{0}=\left\{ B_{0}\right\} $.

%Define a winning Construct a $\left(B_{0},R,R^{1-\eps}\right)$ Cantor set by defining the sequence of collections $\calb_{n}'$ recursively:
Define a winning strategy for Alice in the $\eps_0$ Cantor game as follows: For each $i\geq 0$ define a potential function on balls:
\begin{equation}\label{eq:potential}
\varphi_{i}\left(B\right)=\sum_{m=0}^{\left\lfloor i/\ell\right\rfloor }\sum_{n=m}^{\infty}\#\left\{ A\in\cala_{m,n}\sep A\cap B \text{ has nonempty interior}\right\} R^{-\left(n+1\right)\ell\delta(1 - \eps_2)}.
\end{equation}

%\comdavid{changed ``$A\cap B\neq \emptyset$'' to ``$A\cap B$ has nonempty interior'', since otherwise it would count cases where $A$ and $B$ intersect in the boundary, which would cause problems later}

Given Bob's $i$th move $B_i$, Alice's strategy is to choose the subcollection $\AAA_{i+1} \sub \frac{1}{R}\{B_i\}$ consisting of the $\lfloor R^{\delta(1-\eps_0)} \rfloor$ balls with the largest corresponding values of $\varphi_{i+1}$. We will show that this is a winning strategy for Alice. Assume Bob chose the sequence of balls $\left\{B_i\right\}_{i=0}^\infty$ on his turns, while Alice responded according to the above strategy. It is enough to prove that for every $i\geq 0$
\begin{equation}\label{eq:inductionPotential}
\varphi_{i\ell}\left(B_{i\ell}\right) < R^{-i\ell\delta(1- \eps_2)}.
\end{equation}
Indeed, this inequality would imply that $B_{i\ell}$ does not
intersect any of the balls from $\AAA_{m,n}$ with $n<i$ except at the boundary, which implies that
$B_{i\ell}\in \BBB_i$.

Fix $i\geq 0$ such that \eqref{eq:inductionPotential} holds. For
every $0\leq j\leq\ell-1$, note that $\left\lfloor
\frac{i\ell+j}{\ell}\right\rfloor = i$, so by the induction
hypothesis (\ref{eq:inductionPotential}), we get
\[
\varphi_{i\ell+j}\left(B_{i\ell+j}\right)\le
\varphi_{i\ell}\left(B_{i\ell}\right)<R^{-i\ell\delta(1- \eps_2)}.
\]
Therefore, for any $n\geq m \geq 0$, if $A\in\cala_{m,n}$ is such that $A\cap B_{i\ell+j+1}$ has nonempty interior, then necessarily
$n+1 > i$, and thus by (S2), we have $A \sub B_{i\ell+j+1}$. So
\begin{equation}
\label{fromS2}
%\sum_{A\in\frac{1}{R}\left\{B_{i\ell+j}\right\}}\varphi_{i\ell+j+1}(A)=\varphi_{i\ell+j}(B_{i\ell+ j}),
\sum_{A\in\frac{1}{R}\left\{B_{i\ell+j}\right\}}\varphi_{i\ell+j}(A)=\varphi_{i\ell+j}(B_{i\ell+ j}),
\end{equation}
and the definition of Alice's strategy yields
%\[
%\varphi_{i\ell+j+1}\left(B_{i\ell+j+1}\right)\leq\frac{1}{R^{\eps_0}+1}\varphi_{i\ell+j+1}\left(B_{i\ell+j}\right).
%\]
\begin{equation}\label{th32_eq1}
%\varphi_{i\ell+j+1}\left(B_{i\ell+j+1}\right)\leq\frac{1}{\lfloor R^{\delta(1-\eps_0)}\rfloor+1}\varphi_{i\ell+j}\left(B_{i\ell+j}\right).
\varphi_{i\ell+j}\left(B_{i\ell+j+1}\right)\leq\frac{1}{\lfloor R^{\delta(1-\eps_0)}\rfloor+1}\varphi_{i\ell+j}\left(B_{i\ell+j}\right)
\end{equation}
for every $0\leq j\leq l-1$. Indeed, there are at least
$\lfloor R^{\delta(1-\eps_0)}\rfloor$ balls $A$, chosen by Alice in her turn, with
%$\varphi_{i\ell+j+1}(A)\ge \varphi_{i\ell+j+1}(B_{i\ell+j+1})$.
$\varphi_{i\ell+j}(A)\ge \varphi_{i\ell+j}(B_{i\ell+j+1})$.
Therefore, if \eqref{th32_eq1} fails for for some ball
$B_{i\ell+j+1}$, we have
$$
%\varphi_{i\ell+j+1}\left(B_{i\ell+j+1}\right) + \sum_{A\mbox{ chosen by Alice}} \varphi_{i\ell+j+1}(A) > \varphi_{i\ell+j}\left(B_{i\ell+j}\right)
\varphi_{i\ell+j}\left(B_{i\ell+j+1}\right) + \sum_{A\mbox{ chosen by Alice}} \varphi_{i\ell+j}(A) > \varphi_{i\ell+j}\left(B_{i\ell+j}\right)
$$
which contradicts \eqref{fromS2}.

%For $j=\ell-1$, the potential of Bob's next move decreases
%by the same factor, but we have to consider new potential
%coming from $\cala_{i+1}$. I.e., we have:
In order to estimate $\varphi_{(i+1)\ell}\left(B_{(i+1)\ell}\right)$ we need
to additionally consider the potential coming from the sets
$(\cala_{i+1,n})_{n\ge i+1}$. Clearly, for any $A\in \cala_{i+1,n}$ the
condition that $A\cap B_{(i+1)\ell}$ has nonempty interior is equivalent to the inclusion $A\sub
B_{(i+1)\ell}$. Hence we have:
\begin{align*}
\varphi_{(i+1)\ell}\left(B_{(i+1)\ell}\right) \leq &
\left(\frac{1}{\lfloor R^{\delta(1-\eps_0)}\rfloor+1}\right)^{\ell}\varphi_{i\ell}\left(B_{i\ell}\right) \\
& + \sum_{n=i+1}^{\infty}\#\left\{ A\in\cala_{i+1,n}\sep A\sub B_{(i+1)\ell}\right\} R^{-\left(n+1\right)\ell\delta(1-\eps_2)}.
\end{align*}

Since $\#\left\{ A\in\cala_{m,n}\sep A\sub B_{m\ell}\right\} \leq R^{\left(n-m+1\right)\ell\delta(1-\eps_1)}$ for any $n\geq m\geq0$, by \eqref{eq:potentialEps1}, \eqref{eq:potentialEps2}, and \eqref{eq:inductionPotential}, we get:
\begin{align*}
\varphi_{(i+1)\ell}\left(B_{(i+1)\ell}\right)
& \leq \left(\frac{1}{\lfloor R^{\delta(1-\eps_0)}\rfloor+1}\right)^{\ell}R^{-i\ell\delta(1-\eps_2)}+\sum_{n=i+1}^{\infty}R^{\left(n-i\right)\ell\delta(1-\eps_1)}R^{-\left(n+1\right)\ell\delta(1-\eps_2)}\\
& =    \left(\frac{R^{\delta\left(1-\eps_2\right)}}{\lfloor R^{\delta(1-\eps_0)}\rfloor+1}\right)^{\ell}R^{-(i+1)\ell\delta\left(1-\eps_2\right)}+\sum_{n=1}^{\infty}R^{-n\ell\delta\frac{\eta}{2}}R^{-(i+1)\ell\delta(1-\eps_2)}\\
& =    \left(\left(\frac{R^{\delta\left(1-\eps_2\right)}}{\lfloor R^{\delta(1-\eps_0)}\rfloor+1}\right)^{\ell}+\frac{R^{-\ell\delta\frac{\eta}{2}}}{1-R^{-\ell\delta\frac{\eta}{2}}}\right)R^{-(i+1)\ell\delta(1-\eps_2)}.
\end{align*}
By (\ref{eq:potentialEll1}) and (\ref{eq:potentialEll2}) we get:
\[
\varphi_{(i+1)\ell}\left(B_{(i+1)\ell}\right)<R^{-(i+1)\ell\delta(1-\eps_2)},
\]
which is what we need to complete the induction.
\end{proof}

%Theorem~\ref{thm:potentialCantor} in many cases allows us to simplify the
%proofs which involve $\eps$ Cantor winning sets. Indeed, given an $\eps$ Cantor winning set we may now assume much simpler and seemingly stronger conditions for it than those in Definition~\ref{def:cantorWinning} of Cantor winning sets. We illustrate this point by the following corollary:
Theorem~\ref{thm:potentialCantor} can be rephrased in terms of local Cantor sets (cf. Example \ref{examplebasic}):

\begin{corol}\label{corl:cantor}
Let $E\sub X$. Let $B_0\sub X$ be a closed ball and $0<\eps\le 1$. If $E$ is $\eps$ Cantor winning on $B_0$ then for any $R\ge 2$ it contains a $\left(B_0,R,f(R)^{1-\eps}\right)$ Cantor set.
\end{corol}

Corollary \ref{corl:cantor} means that in Definition~\ref{def:cantorWinning} it is enough to use local Cantor sets. This will be useful in the proof of Theorem \ref{thm:cantor_rich}.

For the sake of completeness, let us define a version of the Cantor game
called the ``Cantor potential game''. Its name is justified via the
analogy with the potential game, and the key role that potential
functions, like the one in (\ref{eq:potential}), play in the proof above that relates it to the Cantor game.

Given $\eps_0\in(0,1]$, the \emph{$\eps_0$ Cantor potential game} is defined
as follows: Bob starts by choosing $0<\eps<\eps_0$, and Alice replies by choosing $R_\eps\geq2$. Then Bob chooses an integer $R\geq R_\eps$, and
a closed ball $B_0\sub X$. For any $i\geq0$, given Bob's $i$th choice $B_i$,
Alice chooses collections $\left\{\cala_{i+1,k}\right\}_{k=0}^\infty$ such
that $\cala_{i+1,k}\sub\frac{1}{R^k}\left\{B_i\right\}$ and
\[
\#\cala_{i+1,k}\leq f(R)^{(k+1)(1-\eps)}
\]
for all $k\geq 0$. Given $\left\{\cala_{i+1,k}\right\}_{k=0}^\infty$, Bob chooses a ball $B_{i+1}\in \frac{1}{R}\left\{B_i\right\}\setminus\bigcup_{k=0}^{i}\cala_{i+1-k,k}$ (if there's no such ball, we say that Alice wins by default). We say that $E\sub X$ is $\eps$ Cantor potential winning if Alice has a strategy which guarantees that either she wins by default or (\ref{eq:outcome}) is satisfied.
The \emph{Cantor potential game} is defined by allowing Alice first to choose $\eps_0\in(0,1]$ and then continuing as in the $\eps_0$ Cantor potential game.

By definition, $\eps_0$ Cantor winning is the same as $\eps_0$ Cantor potential winning. Moreover, there's a natural way to pass between the collections that define a Cantor winning set to the winning strategy in the Cantor potential game. Indeed, if $E\sub X$ is $\eps_0$ Cantor winning in $B_0$ given by the collections $\calb_n$ and $\cala_{m,n}$ for any $n\geq 0$ and any $0\leq m\leq n$, then a winning strategy for Alice for the $\eps_0$ Cantor potential game may be given by $\cala'_{i+1,k} = \cala_{i+1,k+i+1}(B_i)$ where $B_i$ is Bob's $i$th move. In the other direction, let $\cala'$ denote a winning strategy for Alice for the $\eps_0$ Cantor potential game on $B_0$ for $E$, i.e., if $B$ is the $i$th move of Bob, denote by $\cala'_{i+1,k}(B_i)$ Alice's $i+1$st move according to this winning strategy. Define the collections $\calb_n$ for $n\geq0$ and $\left\{\cala_{m,n}\right\}_{0\leq m \leq n}$ for $n\geq1$ by recursion as follows: $\calb_0=\left\{B_0\right\}$,
\[
\cala_{m,n} = \bigcup_{B\in\calb_m} \cala'_{m+1,n-m}(B),
\]
for every $n\geq1$ and $0\leq m\leq n$, and define $\calb_n$ as usual via \eqref{eq:cantorwinning}. Since $\cala'$ is a winning strategy, it implies that the limit set as in \eqref{eq:limitset} is contained in $E$.
%upon rewriting $(m,n)$ as $(i+1,k+i+1)$.
%\comdavid{If this is true, shouldn't Alice choose $R_\eps\geq 2$ and then Bob chooses $R \geq R_\eps$?}
Theorem~\ref{thm:potentialCantor} may then be reinterpreted as saying that the winning sets for the $\eps$ Cantor potential game are the same as the winning sets for the $\eps$ Cantor game.
%\comdavid{Maybe more details are needed here?}\comerez{Here I moved to use freely the notion of Cantor winning as the winning sets in the Cantor game, so, by definition, $R_\eps$ is not needed}
%Indeed, by definition of the Cantor potential game, a set is $\eps$ Cantor potential winning if and only if it is $\left(B_0,\eps\right)$ Cantor winning for every ball $B_0\sub X$.
%\comerez{Is this enough details?} \comdavid{I meant more like, explain the relation between Cantor potential winning and Cantor winning in more detail. Like, maybe we should say that the $\AAA_{m,n}$ in the definition of Cantor winning is equal to the union of all possible $\AAA_{m,n}$ in the definition of Cantor potential winning, or something like that?}\comerez{I think this is similar to the paragraph just before Theorem \ref{thm:potentialCantor} so can be skipped}
%\subsection{Cantor winning , absolute winning and potential winning}
\subsection{Cantor, absolute, and potential winning}\label{sec:CantorPotentialconnections}

We prove Theorem \ref{potentialequivR} as a special case of a more general result connecting potential winning sets with Cantor winning sets, corresponding to  the standard splitting structure on $\RR^N$.

%\begin{theorem}\label{thmabsolutewinning}
%   Assume a complete metric space $X$ is endowed with a nontrivial
%   splitting structure~$(X,\SSS,U,f)$ and that condition (S4) holds
%   with constant $C(X)$. If $E \sub X$ is $1$-Cantor-winning then $E$ is absolute winning.
%\end{theorem}

\begin{theorem}\label{potentialequiv}
Let $(X,\SSS,U,f)$ be a splitting structure on a doubling metric space $X$, fix $B_0\in\BBB(X)$ and denote
\[
\delta = \HD(A_\infty(B_0)).
\]
%so that $\delta = \frac{\log f(u)}{\log(u)}$ for all $u\in U$.
Fix $\eps_0 \in \OC 01$, and let $c_0 = \delta(1 - \eps_0)$. Then a set $E\sub A_\infty(B_0)$ is $\eps_0$ Cantor winning on $B_0$ with respect to $(X,\SSS,U,f)$ if and only if $E$ is $c_0$ potential winning on $A_\infty(B_0)$.

In particular, $E$ is Cantor winning on $B_0$ if and only if $E$ is potential winning on $A_\infty(B_0)$.
\end{theorem}

Since potential winning is defined on metric spaces without using a splitting structure, we get the following observation:

\begin{corol}
%The
Cantor winning
%property
in a doubling metric space $X$ is independent of the choice of splitting structure, provided that the limit set $A_\infty(B)$ is fixed.
\end{corol}

A combination of Proposition~\ref{thm:potential_implies_absolute} and
Theorem~\ref{potentialequiv} 
%immediately 
yields the following 
%additional 
corollary:

\begin{corol}\label{CWabs}
Assume a complete metric space $X$ is doubling and is endowed with a nontrivial splitting structure~$(X,\SSS,U,f)$, and let $B_0 \in \BBB(X)$. Then, a set $E \sub X$ is $1$ Cantor winning on $B_0$ if and only if $E$ is absolute winning on $A_\infty(B_0)$. In particular, for any splitting structure such that $B=A_\infty(B)$ for every ball $B$, a set $E \sub X$ is $1$ Cantor winning if and only if $E$ is absolute winning.
\end{corol}

\begin{rem}
The statements in \cite[Section~7]{BadziahinHarrap} immediately provide us
with applications of Corollary~\ref{CWabs}. In particular, various sets from
\cite[Theorems 14 and 15]{BadziahinHarrap} are in fact absolutely winning.
\end{rem}

%\begin{rem}
%    We can now see that Proposition \ref{propositionfullHD} is strictly stronger than the result that Cantor winning sets have full dimension, since spaces with splitting structures satisfying (S4) must have dimensions of the form $\frac{\log(n)}{\log(m)}$ ($m,n\in\NN$) (cf. \cite[Theorem 3]{BadziahinHarrap}), whereas no such restriction holds on potential winning sets.
%\end{rem}
%%%%%% I MOVE THIS REMARK AFTER THEOREM 4.4 (DMITRY)

We now prove Theorem \ref{potentialequiv}:

\begin{proof}[Proof of Forward Direction]
%As in the proof of Theorem \ref{thm:potentialCantor},
Since $X$ is doubling, we may assume that $f(u) = u^{\delta}$ for every $u\in U$ (cf. \cite[Corollary 1]{BadziahinHarrap}). %\footnote{\comdavid{By the way, here is a counterexample to the analogue of \cite[Corollary 1]{BadziahinHarrap} when $X$ is not doubling: let $X = \ell^\infty$, $B = B(\mathbf 0,1)$, and for each $u\in\NN$ let $\SSS(B,u)$ be the set of all balls of the form $B(\eps/u,1/u)$, where $\eps_i = \pm 1$ for all $i \leq m$ and $\eps_i = 0$ for all $i > m$, where $m$ is the largest integer such that $p^m$ divides $u$, where $p$ is a fixed prime. For $B'\in \SSS(B,v)$, let $\SSS(B,u) = \{B''\in \SSS(B,uv) : B''\sub B\}$. Then $f(u) = 2^m$ cannot be expressed as $f(u) = u^\delta$ for any $\delta$. (Technically one should define $\SSS(B',u)$ for all balls $B'$ rather than just for $B'\in\SSS(B,v)$, but this example illustrates the basic idea.)}}
Suppose that $E$ is $\eps_0$ Cantor winning, and fix $\beta > 0$, $c > c_0$, and $\rho_0 >
0$. Fix $0 < \eps < \eps_0$ such that $\delta(1 - \eps) \in (c_0,c)$. Fix
large $R\in U$. Its precise value will be determined later. Potentially $R$
may depend on $\eps$, $\beta$, $c$, and $\rho_0$, and in particular it
satisfies $R \geq \max(R_\eps,1/\beta)$ and
$R^{\delta(1-\varepsilon)-c}<\frac12$. Then by the definition of Cantor winning,
$E$ contains some $\left(B_0,R,\vr\right)$ Cantor set $\KKK$, where
    \[
    r_{m,n} = f(R)^{(n - m + 1)(1 - \eps)} = R^{\delta(n - m + 1)(1 - \eps)} \text{ for all $m,n$ with $m\leq n$.}
    \]
We now describe a strategy for Alice to win the $(c,\beta)$ potential game on
$A_\infty(B_0)$, assuming that Bob's starting move has radius $\rho_0$.
%For each $m,n\in\NN$ with $m\leq n$ and for each $B_m\in\BBB_m$, let $\RRR(B_m,R^{n - m + 1})$ denote the collection of elements of $\SSS(B_m,R^{n - m + 1})$ which are removed in the construction of the $(B,R,\vr)$-Cantor set, so that $\#(\RRR(B_m,R^{n - m + 1}) \leq r_{m,n}$.
Recall that for any $B\in \BBB_m$,
$$\AAA_{m,n}(B):=\left\{ A \in \AAA_{m,n}:\, A \sub B   \right\}.$$
Let $\rho$ denote the radius of $B_0$, let $D_k$ denote Bob's $k$th move,\footnote{We are denoting Bob's $k$th move by $D_k$ instead of $B_m$ so as to reserve the letters $B$ and $m$ for the splitting structure framework.}
%\footnote{We are denoting Bob's $k$th move by $D_k$ because $B_m$ already denotes a ball in the splitting structure's framework.}
and for each
$k\in\NN$ let $m = m_k\in\NN$ denote the unique integer such that $\beta\cdot
\rad\left(D_k\right) < R^{-m}\rho \leq \rad\left(D_k\right)$, assuming such an integer exists. If
such a number does not exist then we say that $m_k$ is undefined. Then
Alice's strategy is as follows: On turn $k$, if $m_k$ is defined, then remove all
elements of the set
    \[
\bigcup_{n\geq m} \AAA_{m,n} = \bigcup_{B\in \BBB_m} \bigcup_{n\geq m} \AAA_{m,n}(B)
    \]
that intersect Bob's current choice. If $m_k$ is undefined, then delete
nothing. Obviously, this strategy, if executable, will make Alice win since
the intersection of Bob's balls will satisfy $\bigcap_{k \in \NN}D_k \sub
\KKK \sub E$. To show that it is legal, we need to show that
\begin{equation}\label{NTS}
        \sum_{\substack{B\in \BBB_m \\ B\cap D_k \neq \smallemptyset}} \sum_{n \geq m} r_{m,n} \left(R^{-(n + 1)}\rho\right)^c \leq \left(\beta\cdot \rad(D_k)\right)^c.
\end{equation}
This is because elements of $\AAA_{m,n}(B)$ all have radius $R^{-(n + 1)}\rho$.
%, and Bob's $k$th move always has radius $\beta^k \rho_0$.

To estimate the sum in question, we use the fact that $X$ is doubling.
%Fix $m\in W_k$.
Since the elements of~$\BBB_m$ have disjoint interiors and have radius $R^{-m} \rho
\asymp_\beta \rad(D_k)$, the number of them that intersect~$D_k$ is bounded
by a constant depending only on $\beta$. Call this constant $C_1$. Then the
left-hand side of \eqref{NTS} is less than
    \begin{align*}
        C_1 \sum_{n\geq m} R^{\delta(n - m + 1)(1 - \eps)} \left(R^{-(n + 1)}\rho\right)^c
        &= C_1 \left(R^{-m} \rho\right)^c \sum_{\ell = 1}^\infty R^{\delta(1 - \eps)\ell} R^{-c\ell}\\
        &\le 2 C_1 \left(\rad(D_k)\right)^c R^{\delta(1 - \eps) - c}. \qquad\left(\mbox{since $R^{\delta(1 - \eps) - c}\le \frac12$}\right)
    \end{align*}
    By choosing $R$ large enough so that
    \[
    2C_1 R^{\delta(1 - \eps) - c} < \beta^c,
    \]
    we guarantee that the move is legal.
\end{proof}

\begin{proof}[Proof of Backward Direction]
Suppose that $E \sub A_\infty(B_0)$ is $c_0$ potential winning, and fix $0 <
\eps < \eps_0$. Let $R_\eps = 2$, and fix $R\in U$ such that $R\geq R_\eps$.
Fix a large integer $q\in\NN$ to be determined, and let $\beta = 1/R^q$ and
$c = \delta(1 - \eps)$. Then $E$ is $(c,\beta)$ potential winning. Now for
each $k\in\NN$ and $B\in\BBB_{qk}$, let $\AAA(B)$ denote the
collection of sets that Alice deletes in response to Bob's $k$th move
$B_k=B$, assuming that it has been preceded by the moves
$B_0,B_1,\ldots,B_{k-1}$ satisfying $B_{i + 1}\in
\SSS(B_i,R^q)$ for all $i = 0,\ldots,k - 1$.\footnote{It is necessary to
specify the history in order to uniquely identify Alice's response because
unlike the other variants of Schmidt's game, the potential game is not
a \textit{positional} game (see \S\ref{winprelim} and \cite[Theorem 7]{Schmidt1}). This is because the
regions that were deleted at previous stages of the game affect the overall
situation and therefore may affect Alice's strategy.} Since Alice's strategy
is winning, we must have
    \[
    E\subp A_\infty(B_0) \butnot \bigcup_{k\in\NN} \bigcup_{B\in \BBB_{qk}} \bigcup_{A\in \AAA(B)} A.
    \]
    Indeed, for every point on the right-hand side, there is some strategy that Bob can use to force the outcome to equal that point (and also not lose by default). So to complete the proof, it suffices to show that the right-hand side contains a $\left(B_0,R,\vr\right)$ Cantor set, where
    \[
    r_{m,n} = f(R)^{(n - m + 1)(1 - \eps)} = R^{\delta(n - m + 1)(1 - \eps)} = R^{c(n - m + 1)} \text{ for all $m,n$ with $m\leq n$.}
    \]
For each $m = qk \leq n$ and $B\in\BBB_m$ let $\CCC(n)$ be the collection
of elements of $\AAA(B)$ whose radius is between $R^{-n}\rho$ and $R^{-(n +
1)}\rho$,
    % , let $\AAA_{m,n}(B_m)$ be the collection of all elements of $\SSS(B_m,R^{n - m + 1})$ that intersect some element of $\AAA(B_m)$ with radius between $R^{-n}\rho$ and $R^{-(n + 1)}\rho$,
where $\rho$ is the radius of $B_0$. Define
$$
\AAA_{m,n} (B):= \left\{ A\in \SSS(B,R^{n - m + 1})\;:\; \exists C\in\CCC(n) \mbox{ such that } A\cap C\neq\emptyset\right\}.
$$
By \eqref{eq:potentialLegal} and the definition of $\CCC(n)$, we have
    \[
    \sum_{C\in\CCC(n)} \left(R^{-(n+1)}\rho\right)^c \leq \left(\beta R^{-m}\rho\right)^c;
    \]
    i.e.
    \[
    \#\CCC(n) \leq \beta^c \left(R^{n - m+1}\right)^c.
    \]
This implies that for $m>n$ we get $\#\CCC(n)<1$, i.e. $\CCC(n) = \emptyset$
and in turn $\AAA_{m,n}(B) = \emptyset$. Now we construct a Cantor set
$\KKK\in E$ by removing the collections $\AAA_{m,n}(B)$ ($m\leq n$, $B\in
\BBB_m$). Then to complete the proof, we just need to show that
    \[
    \#\AAA_{m,n}(B) \leq r_{m,n} = R^{c(n - m + 1)}.
    \]
for every $B\in\BBB_m$. By the doubling property each element of $\CCC(n)$ intersects a bounded
number of elements of $\SSS\left(B,R^{n - m + 1}\right)$. Thus
    \[
    \#\AAA_{m,n}(B) \leq C_3 \#\CCC(n)\leq  C_3 \beta^c \left(R^{n - m+1}\right)^c
    \]
    for some constant $C_3 > 0$ depending on $R$. By letting $\beta$ be sufficiently small (i.e. letting $q$ be sufficiently
    large), we can get $C_3 \beta^c \leq 1$ and therefore
    \[
    C_3 \beta^c \left(R^{n - m+1}\right)^c \leq R^{(n - m + 1)c},
    \]
    which completes the proof.
\end{proof}

Finally, we briefly mention how the second statement of
Theorem~\ref{potentialequivR} follows from the first statement. If $E$ is 1
Cantor winning then, by the first part of the theorem, it is $c$ potential
winning for any $c>0$. The statement follows as an immediate corollary of
Proposition~\ref{thm:potential_implies_absolute}.

\subsection{Schmidt's winning and Cantor winning}\label{sec:SchmidtCantor}

We now move our attention to the links between Schmidt's original game and generalised Cantor set constructions. Thanks to Theorem~\ref{potentialequiv}, in $\delta$ Ahlfors regular spaces we can treat the notions of Cantor winning sets and potential winning sets as interchangeable. The latter notion will be more convenient in this section. We prove Theorem \ref{potentialschmidtRN} as a special case of a more general result:

%\subsubsection{Proof of Theorems \ref{potentialschmidtRN} \& \ref{theoremPWWintersectionR}}

\begin{theorem}\label{potentialschmidt}
\label{theoremschmidtgamerelation} Let $X$ be a $\delta$ Ahlfors regular
space and let $E\sub X$ be a $c_0$ potential winning set. Then for all $c >
c_0$, there exists $\gamma > 0$ such that for all $\alpha,\beta > 0$ with
$\alpha < \gamma(\alpha\beta)^{c/\delta}$, the set $E$ is $(\alpha,\beta)$
very strong winning.
\end{theorem}

To see that Theorem \ref{potentialschmidtRN} follows from Theorem
\ref{potentialschmidt},
%simply
notice that $\RR^N$ is
%obviously
$N$ Ahlfors regular, and that very strong $(\alpha,\beta)$ winning sets are
%Schmidt
$(\alpha,\beta)$ winning. When $E$ is $\eps$ Cantor winning, the desired statement follows upon taking $c_0=N(1-\eps)$ in the above. The reader may wish to compare the following proof with the proof of Theorem \ref{thm:potentialCantor}.
%\comstephen{Reworded}

%We first prove Theorem \ref{potentialschmidt} and Corollary \ref{theoremPWWintersection}.
\begin{proof}
Consider a strategy for Alice to win the $(c,\w\beta)$ potential game for the
set $E$, where $\w\beta = (\alpha\beta)^q$ for some large $q\in\NN$. We want
to show that Alice has a winning strategy in the $(\alpha,\beta)$ very strong game.
Given any sequence of Bob's choices in the $(\alpha,\beta)$ game, say
$B_0,\ldots,B_{qk}$, we can consider the corresponding sequence of choices in
the $(c,\w\beta)$ potential game which correspond to Bob choosing the balls
$B_0,B_q,\cdots,B_{qk}$. Alice's response to this sequence of moves is to
delete a collection $\AAA(B_{qk})$. Now Alice's corresponding response in the
$(\alpha,\beta)$ very strong game will be: if Bob has made the sequence of moves
$B_0,\ldots,B_m$, then she will choose her ball $A_{m+1} \sub B_m$ so as to
minimize the value at $A_{m+1}$ of the potential function %\comdavid{Reworded}
    \[
    \varphi_m(A) := \sum_{\substack{k\in\NN \\ qk \leq m}} \sum_{\substack{C\in\AAA(B_{qk}) \\ C\cap A \neq \smallemptyset}} \diam^c(C).
    \]
    We will now prove by induction that the inequality
    \[
    \varphi_m\left(A_{m+1}\right) \leq (\eps \cdot \rad\left(B_m\right))^c
    \]
holds for all $m$, where $\eps > 0$ will be chosen later (small but
independent of $\w\beta$).
%Indeed,
Suppose that it holds for some
$m$, and let $B_{m + 1}$ be Bob's next move.
%, of radius $(\alpha\beta)^{m + 1} \rho$.
Define $r = \alpha\cdot \rad(B_{m+1})\ge \alpha^2\beta\cdot \rad(B_m)$.
By Proposition~\ref{prop_ahlfors}, there exists a set
$$
\DDD:= \{B\left(x_i,r\right)\sub B_{m+1}\;:\; i\in \{1,\ldots,I\}\},
$$
which satisfies the following conditions: $I = \#\DDD \gg \alpha^{-\delta}$,
and for any distinct values $1\le i\neq j\le I$ one has $d(x_i,x_j)\ge 3r$,
i. e. all balls in $\DDD$ are separated by distances of at least $r$.

%Since $X$ is $\delta$ Ahlfors regular, there exists a set $\DDD$ of $d\gg
%\alpha^{-\delta}$ balls of radius $\alpha (\alpha\beta)^{m + 1} \rho$,
%separated by distances of at least $\alpha (\alpha\beta)^{m + 1} \rho$.
By the induction hypothesis, for all $C\in \bigcup_{qk\leq m} \AAA(B_{qk})$
such that $C\cap A_{m+1} \neq \emptyset$, we have $\diam(C) \leq \eps\cdot \rad(B_m)$. So by letting $\eps$ be small enough (i.e. $\eps \leq
\alpha^2\beta$), we can guarantee that all balls in $\DDD$ make disjoint
contributions to $\varphi_m(A_{m+1})$. In other words,
$$
\sum_{D\in \DDD} \varphi_m(D) \le \varphi_m(A_{m+1}).
$$
By choosing $A_{m + 2}$ to be the ball from $\DDD$ which minimises the
function $\varphi_m$, we get
\[
\alpha^{-\delta} \varphi_m(A_{m + 2}) \ll \varphi_m(A_{m+1}).
\]
By choosing $\gamma$ small enough, the inequality $\alpha<
\gamma(\alpha\beta)^{c/\delta}$ implies that
\[
\varphi_m(A_{m + 2}) \leq (\alpha\beta)^c \varphi_m(A_{m+1})/2 = \frac12 (\eps \cdot\rad(B_{m+1}))^c.
\]
Note that for the base of induction, $m=-1$, the last estimate is straightforward, since $\varphi_m(A) = 0$ for all balls $A$, and we can set $A_0=B_0$. %\comdavid{Isn't the base of induction $m=0$? There is no such thing as $A_{-1}$.}

When we consider $\varphi_{m+1}(A_{m+2})$ there may be a new term coming from
a new collection $\AAA(B_{m + 1})$. By the definition of a $(c,\w\beta)$
potential game, that term is at most $(\w\beta \cdot\rad(B_{m+1}))^c$,
and thus by choosing $q$ large enough we can guarantee that
$$
\varphi_{m+1}(A_{m+2}) \le (\eps (\alpha\beta)^{m+1}\rho)^c.
$$
This finishes the induction.

Since Alice followed the winning strategy for $(c,\w\beta)$ potential game,
we have either
$$
\bigcap_{k=1}^\infty B_{kq} \sub E\quad\mbox{or}\quad \bigcap_{k=1}^\infty B_{kq} \sub \bigcup_{k=1}^\infty \bigcup_{A\in\AAA(B_{kq})} A.
$$
To finish the proof of the theorem, we need to show that the second inclusion
is impossible. If we assume the contrary, then there exists $A\in \AAA(B_{kq})$ for
some $k$ such that every ball $B_m$ and in turn every ball $A_m$ has
nonempty intersection with $A$. But in that case we have $\varphi_m(A_{m+1}) \geq
\diam^c(A)$, which contradicts the fact that $\varphi_m(A_{m+1}) \to 0$ as $m\to\infty$.
\end{proof}

As a corollary of Theorem \ref{potentialschmidt}, we get a general form of Theorem \ref{theoremPWWintersectionR}:

\begin{corol}\label{theoremPWWintersection}
    In an Ahlfors regular space, the intersection of a weak winning set and a potential winning set has full Hausdorff dimension. %\comdavid{The proof was confusing in how it switched between Cantor winning and potential winning (which are of course equivalent), so I rewrote it to just talk about potential winning}
\end{corol}

\begin{proof}
Let $E\sub X$ be a $c_0$ potential winning set for some $c_0<\delta$. Fix $c_0 < c <
\delta$. By Theorem~\ref{potentialschmidt}, there exists $\gamma > 0$ such
that for all $\alpha,\beta>0$ with $\alpha < \gamma
(\alpha\beta)^{c/\delta}$, $E$ is $(\alpha,\beta)$ very strong winning.
%In fact, the
%proof shows that $E$ is $(\alpha,\beta)$ very strong winning for these values
%of $(\alpha,\beta)$; the only place that information is needed about the
%radius of Bob's ball is in bounding the potential function in the inductive
%step, and this can be done with an inequality just as well as with an
%equality.
In particular, for all $\beta > 0$ there exists $\alpha > 0$ such that $E$ is
$(\alpha,\beta)$ very strong winning. Hence we can apply
Proposition~\ref{prop_enote} which says that $X\butnot E$ is not weak winning (and
thus does not contain a weak winning set). So every potential winning set
intersects every weak winning set nontrivially.

Now by contradiction suppose that $T$ is a weak winning set such that $\HD(E\cap
T) < \delta$. This implies that for any $c \in (\HD(E\cap T), \delta)$ and
for any $\eps>0$ one can construct a cover of $E\cap T$ by a countable
collection $(A_i)_{i\in \NN}$ of balls such that
$$
\sum_{i=1}^\infty \diam^c(A_i) <\eps.
$$
In other words, this cover can be chosen by Alice in a single move of the $c$
potential game, so $X\butnot (E\cap T)$ is potential winning. So by the
intersection property of potential winning sets, $E\butnot T = E\cap (X\butnot (E\cap T))$ is potential
winning. But then $X\butnot T$ is a potential winning set whose complement is weak
winning, contradicting what we have just shown.
\end{proof}

\subsection{Proof of Theorem \ref{thmwinningmainN}}\label{winprelim}

%We first prove Theorem \ref{thmwinningmainN}.
%<<Old N-dim proof currently in some state in separate tex file>>
%\subsubsection{Preliminaries}\label{winprelim}
%To
%do this
%prove Theorem \ref{thmwinningmainN}
In this section we prove that $1/2$ winning implies absolute winning for subsets of the real line. The same notation, terminology, and observations as in~\cite{Schmidt1} are employed.
Recall that
%Denote by
$\BBB(\RR)$ stands for the set of all closed intervals in~$\RR$.

In Schmidt's $(\alpha, \beta)$ game, the method Alice uses to determine where to play her intervals~$A_i$
%(in view of Bob's preceding inverval $B_{i-1})
is called a \textit{strategy}. Formally, a strategy $F:=\{f_1, f_2, \ldots \}$ is a sequence of functions $f_i: \BBB(\RR)^i \rightarrow \BBB(\RR)$ satisfying $\rad(f_i(B_0, B_1, \ldots, B_{i-1})) = \alpha \cdot \rad(B_{i-1})$ and $f_i(B_0, B_1, \ldots, B_{i-1}) \sub B_{i-1}$. A set $E$ is a winning set for the $(\alpha, \beta)$ game if and only if there exists a strategy determining where Alice should place her intervals $A_{i}:=f_i(B_0, B_1, \ldots, B_{i-1})$ so that, however Bob chooses his intervals $B_{i} \sub A_{i}$, the intersection point $\bigcap_{i} B_i$ lies in $E$.  Such a strategy is referred to as a \textit{winning strategy (for $E$)}. Schmidt \cite[Theorem $7$]{Schmidt1} observed that the existence of a winning strategy guarantees the existence of a \textit{positional} winning strategy; that is, a winning strategy for which the placement of a given interval by Alice needs only to depend upon the
position of Bob's immediately preceding interval, not on the entirety of the game so far. To be precise, a winning strategy~$F:=(f_1, f_2, \ldots )$ is \textit{positional} if there is a function $f_0: \BBB(X) \rightarrow \BBB(\RR)$
%, independent of $i$,
such that each function $f_i \in F$ satisfies $f_i(B_0, \ldots, B_{i-1})=f_0(B_{i-1})$.  Employing a positional strategy will not give us any real advantage in the proof that follows; however, it will
allow us to simplify our notation somewhat. Without confusion we will write $A_i=F(B_{i-1}) = f_0(B_{i-1})$ where $F$ is a given positional winning strategy for Alice and $f_0$ is the function witnessing $F$'s positionality. %\comdavid{rewrote this line, since the previous version was confusing}
%to place her $i$th interval during the game, once the preceding interval $B_{i-1}$ has been played by Bob.
%If the associated function $f_0$ needs to be made clear we will write $F=F(f_0)$.

%As a final piece of terminology form \cite{Schmidt1},
We refer to a sequence~$\{B_0, B_1, \ldots \}$ of intervals as an \textit{$F$-chain} if it consists of the moves Bob has made during the playing out of an $(\alpha, \beta)$ game with target set $E$ in which Alice has followed the winning strategy $F$. By definition we must have $\bigcap_{i} B_i \, \sub E$. Furthermore, we say a finite sequence $\{B_0, \ldots B_{n-1}\}$ is an \textit{$F_n$-chain} if there exist $B_{n}, B_{n+1}, \ldots$ for which the infinite sequence $\{B_0, \ldots B_{n-1}, B_{n}, B_{n+1} \ldots \}$ is an $F$-chain. One can readily verify (see \cite[Lemma~$1$]{Schmidt1}) that if $\{B_0, B_1, \ldots \}$ is a sequence of intervals such that for every $n \in \NN$ the finite sequence~$\{B_0, B_1, \ldots B_{n-1}\}$ is an $F_n$-chain, then $\{B_0, B_1, \ldots \}$ is an $F$-chain.

%\subsubsection{Outline of the strategy for Theorem \ref{thmwinningmainN}}
\textit{Outline of the strategy.}
We begin with some interval $b_0 \sub \RR$ of length $2\cdot\rad(b_0)$, some $\eps \in (0, 1)$, some $R$ sufficiently large, and a $1/2$ winning set $E \sub \RR$. Fix $\alpha=1/2$ and $\beta=2/R$, so that $\alpha \beta = 1/R$. We will construct for every $R \geq 10^{1/(1-\eps)}$  a local $(b_0, R, 10)$ Cantor set~$\KKK$ lying inside $E$. %, where $c_N>0$ is some absolute constant depending only upon $N$.
This is sufficient to prove that $E$ is $1$ Cantor winning. %, assuming we may take~$R_\eps$ large enough so that $R_\eps \geq (c_N)^{1/(1-\eps N)}$. Such a choice is always possible since $\eps < 1/N$.

By assumption the set $E$ is $(1/2, 2/R)$ winning; thus, there exists
%a winning strategy, and therefore there also exists
a positional winning strategy~$F$ such that however we choose to place Bob's intervals $B_i$ in the game, the placement of Alice's subsequent intervals $A_i=F(B_{i-1})$ guarantees that the unique element of $\bigcap_i B_i$ will fall inside $E$. We will in a sense play as Bob in many simultaneous $(1/2, 2/R)$ games, each of which Alice will win by following her prescribed strategy~$F$.

Our procedure is as follows. The construction of the local Cantor set $\KKK$ comprises the construction of subcollections~$\BBB_n$ of intervals in $\frac 1 {R^n} \BBB_0$ for each $n \in \NN$, where $\BBB_0:=\{ b_0\}$. Construction will be carried out iteratively in such a way that for any interval $b \in \BBB_n$ there exists an $F_n$-chain $\{B_0, B_1, \ldots, B_{n-1}\}$ satisfying $b \sub B_{n-1}$. Moreover, for each \textit{ancestor} $b'$ of $b$ (that is, the unique interval $b' \in \BBB_k$  with $b \sub b'$, for some given $0 \leq k < n$) the subsequence $\{B_0, B_1, \ldots, B_{k-1}\} $ will coincide with the  $F_k$-chain constructed for $b'$ satisfying $b' \sub B_{k-1}$ from the earlier inductive steps. We write $B_n(b)$ if the dependence of the interval $B_n$ on an interval~$b$ is not clear from context. In view of the previous discussion, upon completion of the iterative procedure it follows that for any sequence of intervals $\{b_i\}_{i \in \NN}$ with $b_i \in \BBB_i$ the associated sequence $\{B_0, B_1, \ldots\}$ must be an $F$-chain satisfying $b_i \sub B_{i-1}$. %, and so $\bigcap_{i \in \NN}B_i \sub E$.
In this way, for every point $\mathbf{x}$ in the Cantor set $\KKK$ we will establish the existence of an $F$-chain $\{B_0, B_1, \ldots\}$ for which $\{\mathbf{x}\}= \bigcap_{i}B_i$; namely, we may simply choose the $F$-chain associated with a sequence of intervals $\{b_i\}_{i \in \NN}$ with $b_i \in \BBB_i$ satisfying $\bigcap_{i \in \NN}b_i=\{\mathbf{x}\}$. It follows that $\mathbf x$ must fall within $E$ and, since $\mathbf{x}\in \KKK$ was arbitrary, that $\KKK \sub E$ as required.

%\subsubsection{Proof of Theorem \ref{thmwinningmainN}}\label{sec:N1}

%\comdavid{removed ``{\tt \textbackslash begin\{proof\}}'' here since the proof has essentially begun already in the paragraph titled ``Outline of the strategy', and anyway the title of the subsection is ``Proof of Theorem \ref{thmwinningmainN}''}

We begin the first step of the inductive process. Given $\BBB_0=\{b_0\}$, the idea is to construct the collection $\BBB_1 \sub \frac 1 R \BBB_0$ in such a way that for every $b \in \BBB_1$ there exists a place Bob may play his first interval $B_0$ in a $(1/2, 2/R)$ game so that Alice's first interval $A_1:=F(B_0)$
(as specified by the winning strategy~$F$)
contains~$b$; that is, we construct $\BBB_1$ so that for every $b \in \BBB_1$ there exists an $F_1$-chain $\left\{ B_0 \right\}$ satisfying $b \sub A_1 := F(B_0)$. For technical reasons, in practice we will ensure the stronger condition that $b \sub (1-2/R)A_1$, where $\kappa B$ denotes the interval with the same centre as $B$ but with radius multiplied by $\kappa$.
%This will ensure that the placement of Bob's intervals are always legal in the following inductive procedure: % and us to construct a $(b_0, R, 10)$-Cantor set~$\KKK$ lying inside $E$ for every sufficiently large $R$.

Recall that for his opening move in a $(1/2, 2/R)$ game Bob may choose any interval in~$\RR$. In particular, he may choose any interval of radius $2 \cdot \rad(b_0)$ with centre lying in~$b_0$, therefore covering $b_0$. Let $\WWW_1(b_0) \sub \BBB(\RR)$ be the set of all ``winning'' intervals for Alice (as prescribed by the winning strategy~$F$) corresponding to any opening interval of this kind that Bob might choose to begin the game with; that is, let
\begin{equation*}
    \WWW_1(b_0): \: =  \: \left\{F(B): \: B \in \BBB(\RR),\: \rad(B) = 2\cdot\rad(b_0), \, \cent(B) \in b_0  \right\}.
\end{equation*}
Note that this set could either be uncountable or, since for each $A \in \WWW_1(b_0)$ we have  $\rad (A)=\alpha \cdot (2\cdot\rad(b_0)) = \rad(b_0)$, it could simply coincide with the singleton $\BBB_0$. %Moreover, for every ball $B \in \BBB(\RR^N)$ with $\rad(B) = 2\cdot\rad(B_0)$ and $\cent(B) \in B_0$ there must exist a ball $A^B \in \WWW_1(B_0, R)$ for which $A^B=f_1(B)$.
Moreover, since $\alpha=1/2$ we must have for every interval $B$ that $\cent(B)$ is contained in $F(B)$, and so $\WWW_1(b_0)$ is a cover of $b_0$. This property is unique to the case $\alpha=1/2$.

\begin{lemma} \label{prop:cover1}
    Let $b \in \BBB(\RR)$ be any closed interval and let $\mathcal{W} \sub \BBB(\RR)$ be a cover of~$b$ by closed intervals of radius $\rad(b)$. Then, for any strictly positive $\eps < \rad(b)$ there exists a subset $\mathcal W^\ast(\eps) \sub \WWW$ of cardinality at most two that covers $b$ except for an open interval of length at most $\eps$.
\end{lemma}
%\comdavid{Here is a much shorter proof}
\begin{proof}[Proof of Lemma \ref{prop:cover1}]
Fix $n$ such that $n^{-1} \rad(b) \leq \eps$, and let $\cent(b) - \rad(b) = x_0,\ \ldots,\ x_n = \cent(b) + \rad(b)$ be an evenly spaced sequence of points starting and ending at the two endpoints of $b$. For each $i=0,\ldots,n$ let $W_i$ be an interval in $\WWW$ containing $x_i$. Now let $j$ be the smallest integer such that $x_n\in W_j$. If $j=0$, then $W_0 = b$ and we are done. Otherwise, since $x_n\notin W_{j-1}$, we have $x_0\in W_{j-1}$ and thus $[x_0,x_{j-1}] \sub W_{j-1}$, whereas $[x_j,x_n] \sub W_j$. Thus the intervals $W_{j-1}$ and $W_j$ cover $b$ except for the interval $(x_{j-1},x_j)$, which is of length at most $\eps$.
\end{proof}

We continue the proof of Theorem \ref{thmwinningmainN}. Take $b=b_0$, $\eps=\rad(b_0)/R$, and $\WWW= \WWW_1$ in Lemma~\ref{prop:cover1}. Then there exists a subcollection $\WWW_1^\ast:=\{A_1^L, A_1^R\}$ of $\WWW_1$ of cardinality at most two covering $b_0$ except for an open interval $I(b_0)$ of length at most $\rad(b)/R$. It is possible that $I(b_0)$ is empty. If $\WWW_1^\ast$ contains only one interval, say $\WWW_1^\ast = \{ b'\}$, then set $A_1^L=A_1^R=b'$. Let $B_0^L$ and $B_0^R$ be two choices for Bob's opening move with centres in $b_0$ that satisfy $A_1^L=F(B_0^L)$ and $A_1^R=F(B_0^R)$; such choices exist by the definition of $\WWW_1$. If we had $A_1^L=A_1^R$ then assume $B_0^L=B_0^R$.

We now discard the ``bad'' intervals that will not appear in the generalised Cantor set~$\KKK$. Let
$$\bad_1:= \left\{b' \in \frac{1}{R}\BBB_0: \, b'\cap I(b_0) \neq \emptyset\;\mbox{or}\;  b' \not\sub \left( (1-2/R)A_1^L \cup (1-2/R)A_1^R \right) \right\}. $$
Since any interval of length $\rad(b_0)/R$ may intersect at most two intervals from $\frac{1}{R}\BBB_0$, and there are at most five intervals (of length at most $\rad(b_0)/R$) that need to be intersected by $b'$ for it to be in $\bad_1$,  it is immediate that $\#\bad_1 \leq 10$. It follows that $\BBB_1:= \frac{1}{R}\BBB_0 \setminus \bad_1$ satisfies the required conditions for our generalised Cantor set. For each $b \in \BBB_1$ we associate one of the $F_1$-chains $\left\{ B_0^L \right\}$ or $\left\{ B_0^R \right\}$, depending on whether $b$ is a subset of $A_1^L$ or $A_1^R$ respectively. If $b$ lies in both $A_1^L$ and $A_1^R$ (which is possible if $I(b_0)$ was empty) then we will assign as a precedent whichever of the intervals $B_0^L$ or $B_0^R$ has the `leftmost' centre.  This completes the first step of the inductive construction process.

Assume now that we have constructed the collection $\BBB_i$ so that for each $b \in\BBB_i$ we have an $F_i$-chain $\left\{B_0, \ldots, B_{i-1} \right\}$ with $\rad(B_{i-1})=2R \cdot\rad(b)$ and $b \sub (1-2/R)F(B_{i-1})$. We outline the procedure to construct the subsequent collection $\BBB_{i+1}$ and the corresponding $F_{i+1}$-chains:

Fix $b \in \BBB_i$ with associated $F_i$-chain $\left\{B_0, \ldots, B_{i-1} \right\}$. Assume Alice has just played her $i$th move~$A_i:=F(B_{i-1})$, an interval of radius $R \cdot \rad(b)$, in a $(1/2, 2/R)$ game corresponding to this $F_i$-chain. Bob may then choose any interval of radius $2 \cdot \rad(b)$ inside $A_i$. Since $b \sub (1-2/R)A_i$, Bob is free to choose as $B_i$ any interval of radius  $2 \cdot \rad(b)$ centred in~$b$; every such choice $b'$ will constitute a legal move since $\rad(b')=\frac2R \rad(A_i)$ and $b \sub (1-2/R)A_i$ together imply $b' \sub A_i$. In particular, the collection
\begin{equation*}
\WWW_{i+1}(b): \: =  \: \left\{F(b'): \: b' \in \BBB(\RR),\: \rad(b') = 2\cdot\rad(b_0)/R^i, \, \cent(b') \in b  \right\}
\end{equation*}
is a cover of $b$ and consists of candidate ``winning'' moves that could be played by Alice after any legal move $B_i$ (with centre in $b$) played by Bob. As before, we apply Lemma~\ref{prop:cover1} to find a subset of $\WWW_{i+1}(b)$ of cardinality at most two that covers $b$ except for an open subinterval~$I(b)$ of length at most~$\rad(b)/R$. If the subcover contains two intervals denote them by $A_{i+1}^L(b)$ and $A_{i+1}^R(b)$; otherwise, if the subcover consists of a single ball $b'$ then let $A_{i+1}^L(b)=A_{i+1}^R(b)=b'$. Denote by $B_{i}^L(b)$ and $B_{i}^R(b)$ some choice of intervals centred in $b$ for which $A_{i+1}^L(b)=F(B_{i}^L(b))$ and $A_{i+1}^R(b)=F(B_{i}^R(b))$ respectively. Define
$$\bad_{i+1}(b):= \left\{b' \in \frac{1}{R}\left\{b\right\}: \, b'\cap I(b) \neq \emptyset\;\mbox{or}\;  b' \not\sub \left( (1-2/R)A_{i+1}^L(b) \cup (1-2/R)A_{i+1}^R(b) \right) \right\} $$
and let $\BBB_{i+1}(b):= \frac{1}{R}\left\{b\right\} \setminus \bad_{i+1}(b)$. By the same arguments as before it is clear that $\#\bad_{i+1}(b) \leq 10$ for each $b \in \BBB_i$.

Finally, for each $b' \in \BBB_{i+1}(b)$ we assign the $F_{i+1}$-chain $\left\{B_0, \ldots, B_{i-1}, B_{i}^L(b) \right\}$ or $\left\{B_0, \ldots, B_{i-1}, B_{i}^R(b_i) \right\}$ depending on whether $b'$ lies in  $A_{i+1}^L(b_i)$ or $A_{i+1}^R(b_i)$ respectively. Again, if $I(b)$ was empty and $b$ was contained in both $A_{i+1}^L(b)$ and $A_{i+1}^R(b)$ then we choose whichever of $B_{i}^L(b)$ and $B_{i}^R(b)$ has the leftmost centre as a convention.

Upon setting $$\BBB_{i+1}:= \bigcup_{b \in \BBB_i}\BBB_{i+1}(b)$$ we see that the conditions of our generalised Cantor set $\KKK(b_0,R,10)$ are satisfied. This completes the inductive procedure.

\subsection{Cantor rich and Cantor winning}\label{sec:rich}

In this section we prove the following statement which implies Theorem~\ref{thm:cantor_rich}: %\comdavid{removed the phrase ``more general result'' because Theorem \ref{cor:cantor-winning-implies-rich} does not appear to be any more general than Theorem \ref{thm:cantor_rich}, it is just stronger}

%will use the characterisation of Cantor winning sets proved in Theorem \ref{thm:potentialCantor} to prove the following:

\begin{theorem}\label{cor:cantor-winning-implies-rich}
Let $(X,\SSS,U,f)$ be a splitting structure, let
%$B_0\sub \RR$
$B_0\sub X$
be any ball, and let $\eps>0$ be a real number. If
%$E\sub\bbr$
$E\sub X$
is $\eps$ Cantor winning in $B_0$ then it is $\left(B_0,M\right)$ Cantor rich with
\[
%M = 4^{\frac{1}{\eps}}.
M = 4^{\frac{1}{\eps}}.
\]
Conversely, assume that $X$ is doubling
%, $f(u)=u^\delta$ for every $u\in U$, 
and let $M\ge 4$ be a real number. If $E$ is $\left(B_0,M\right)$ Cantor rich then it is $\eps$ Cantor winning in $B_0$, for any $\eps$ for which there exists $R\in U$ such that
\begin{equation}\label{eq:fromMtoEpsilon}
M < f(R) \leq 4^{\frac{1}{\delta\eps}}.
%\lfloor M^\frac{1}{\delta}\rfloor +1 =4^{\frac{1}{\delta\eps}}.
\end{equation}
In particular, if $X=\bbr$ with the standard splitting structure, then $E$ is $M$ Cantor rich if and only if it is $\eps$ Cantor winning with
\[
\eps = \log_4 \left(\frac{1}{\lfloor M\rfloor +1}\right).
\]
\end{theorem}
\begin{proof}
Assume $E$ is $\eps$ Cantor winning in $B_0$. Then, by Corollary~\ref{corl:cantor} 
%applied for $X=\bbr$ with the standard splitting structure, 
for every $2\leq R\in U$ there exists a 
%$\left(B_{0},R,R^{1-\eps_0}\right)$
$\left(B_{0},R,f(R)^{1-\eps}\right)$
Cantor set contained in $E$. Fix an integer 
%$R>M$
$R$ such that $f(R)>M$ 
and a real number $y>0$. Since
%$\frac{4}{R^{\eps}}<1$, 
$\frac{4}{f(R)^{\eps}}<1$, 
we can choose $\ell>0$ such that 
%$\frac{4}{R}\cdot \left(\frac{4}{R^{\eps}}\right)^{\ell}<y$.
$\frac{4}{f(R)}\cdot \left(\frac{4}{f(R)^{\eps}}\right)^{\ell}<y$.
Consider a 
%$\left(B_{0},R^{\ell},R^{\ell(1-\eps)}\right)$ 
$\left(B_{0},R^{\ell},f(R)^{\ell(1-\eps)}\right)$ 
Cantor set contained in $E$ and denote its Cantor sequence by $\left\{ \calb_{n}\right\} _{n=0}^{\infty}$.
Define another Cantor sequence $\left\{ \calb_{n}'\right\} _{n=0}^{\infty}$
by setting $\calb'_{\ell k}=\calb_{k}$ for all $k\geq0$ and $\calb_{n}'=\frac{1}{R}\calb_{n-1}$
whenever $n\neq\ell k$. Then 
%obviously 
the limit set
$$
\bigcap_{n=1}^\infty \bigcup_{B\in \calb_{n}'} B
$$
is a $\left(B_{0},R,\left(r_{m,n}\right)_{0\leq m\leq n}\right)$
generalised Cantor set, with 
%$r_{\ell\left(k-1\right), \ell k}=R^{\ell(1-\eps)}$
$r_{\ell\left(k-1\right), \ell k}=f(R)^{\ell(1-\eps)}$
for all $k\geq0$, and $r_{m,n}=0$ for all other pairs. Then, for
any $k>0$ and $n = \ell k$, we get:
\[
%\sum_{m=0}^{\ell k}\left(\frac{4}{R}\right)^{\ell k-m +1}r_{m,\ell k}=\left(\frac{4}{R}\right)^{\ell+1}R^{\ell(1-\eps)}=\frac{4}{R}\cdot \left(\frac{4}{R^{\eps}}\right)^{\ell}<y.
\sum_{m=0}^{\ell k}\left(\frac{4}{f(R)}\right)^{\ell k-m +1}r_{m,\ell k}=\left(\frac{4}{f(R)}\right)^{\ell+1}f(R)^{\ell(1-\eps)}=\frac{4}{f(R)}\cdot \left(\frac{4}{f(R)^{\eps}}\right)^{\ell}<y.
\]
For every $n\neq\ell k$, we have
\[
%\sum_{m=0}^{n}\left(\frac{4}{R}\right)^{n-m+1}r_{m,n}=0<y.
\sum_{m=0}^{n}\left(\frac{4}{f(R)}\right)^{n-m+1}r_{m,n}=0<y.
\]

For the other direction, suppose that $E$ is $\left(B_0,M\right)$ Cantor rich. Let $c_0 = \delta(1-\eps)$. By Theorem \ref{potentialequiv}, it is enough to show that $E$ is $c_0$ potential winning in $B_0$. We argue as in the proof of Theorem \ref{potentialequiv}. Choose $R$ such that $M < f(R) \leq 4^{1/\eps}$ and fix any $0 < \beta \leq  \frac{1}{R}$, $c > c_0$.
%, and $\rho > 0$. 
Fix a small enough $y>0$ whose precise value will be determined later. By the definition of Cantor rich, $E$ contains some $\left(B_0,R,\vr\right)$ Cantor set $\KKK$, where $\vr$ satisfies
\[
%\sum_{m=0}^{n} \left(\frac{4}{R}\right)^{n-m+1}r_{m,n} < y
\sum_{m=0}^{n} \left(\frac{4}{f(R)}\right)^{n-m+1}r_{m,n} < y
\]
for every $n\in\bbn$. In particular, for every $m,n\in\bbn$ we have
\[
r_{m,n} < y\left(\frac{f(R)}{4}\right)^{n-m+1} \leq yR^{(n-m+1)c_0}.
\]
We now describe a strategy for Alice to win the $(c,\beta)$ potential game on
$A_\infty(B_0)$.
%, assuming that Bob's starting move has radius $\rho$.
%For each $m,n\in\NN$ with $m\leq n$ and for each $B_m\in\BBB_m$, let $\RRR(B_m,R^{n - m + 1})$ denote the collection of elements of $\SSS(B_m,R^{n - m + 1})$ which are removed in the construction of the $(B,R,\vr)$-Cantor set, so that $\#(\RRR(B_m,R^{n - m + 1}) \leq r_{m,n}$.
%Recall that for any $B\in \BBB_m$,
%$$\AAA_{m,n}(B):=\left\{ A \in \AAA_{m,n}: \, A \sub B   \right\}.$$
As in the proof of Theorem \ref{potentialequiv}, let $\rho$ denote the radius of $B_0$, let $D_k$ denote Bob's $k$th move, and for each
$k\in\NN$ let $m = m_k\in\NN$ denote an integer such that $\beta\cdot
\rad(D_k) < R^{-m}\rho \leq \rad(D_k)$. The inequality $\beta\le 1/R$ ensures that such $m$ exists.
%If such a number does not exist then we say that $m_k$ is undefined.
Then Alice's strategy is as follows: on turn $k$ remove all elements of the set
\[
\bigcup_{B\in \BBB_m} \bigcup_{n\geq m} \AAA_{m,n}(B)
\]
that intersect Bob's current choice. 
%Obviously, t
This strategy, if executable, will make Alice win since
the intersection of Bob's balls will satisfy $\bigcap_{k \in \NN}D_k \sub\KKK \sub E$. 
To show that it is legal, we need to show that
\begin{equation}\label{eq:NTS}
\sum_{\substack{B\in \BBB_m \\ B\cap D_k \neq \smallemptyset}} \sum_{n \geq m} r_{m,n} \left(R^{-(n + 1)}\rho\right)^c \leq (\beta\cdot \rad(D_k))^c.
\end{equation}
This is enough because elements of $\AAA_{m,n}(B)$ all have radius $R^{-(n + 1)}\rho$.
%, and Bob's $k$th move always has radius $\beta^k \rho_0$.
Since the elements of~$\BBB_m$ are disjoint and have radius $R^{-m} \rho
\asymp_\beta \rad(D_k)$, the number of them that intersect~$D_k$ is bounded
by a constant depending only on $\beta$. Call this constant $C_1$. Then the
left-hand side of \eqref{eq:NTS} is less than
    \begin{align*}
        C_1 \sum_{n\geq m} yR^{(n - m + 1)c_0} \left(R^{-(n + 1)}\rho\right)^c
        &= C_1 \left(R^{-m} \rho\right)^c y\sum_{\ell = 1}^\infty R^{\left(c_0 - c\right)\ell}\\
        &\le C_1(\rad(D_k))^c\frac{R^{\left(c_0 - c\right)}}{1 - R^{\left(c_0 - c\right)}} y.
    \end{align*}
    By choosing $y$ so that
    \[
    y < \frac{\beta^c \left(1 - R^{\left(c_0 - c\right)}\right)}{C_1R^{\left(c_0 - c\right)}},
    \]
    we guarantee that the move is legal.
\end{proof}

\section{Intersection with fractals}\label{sec:intersections}

The concept of winning sets gives a notion of largeness which is
orthogonal to both category and measure. %\comdzmitry{I understand how a notion of largeness can be orthogonal to measure, but how can it be orthogonal to category?} \comdavid{Same way as for measure; there is some set which is null with respect to category (i.e. meager) but full for the other notion of largeness, or vice versa. Of course measure and category are orthogonal in this sense.}
It was first noticed in
\cite{BFKRW} that absolute winning subsets of $\bbr^N$ are also
large in a dual sense: they intersect every nonempty diffuse set.
The Cantor and the potential games, together with the definition of
the Ahlfors regularity dimension, allow us to quantify this
observation. 
%In fact, b
Based on the Borel determinacy property for
Schmidt games \cite{FLS} we also prove a partial converse:
% result as well:
%for Borel sets:

\begin{theorem}\label{thm:potential_ahlfors}
If $E\sub X$ is $c_0$ potential winning then $E\cap
K\neq\emptyset$ for every closed set $K\sub X$ with $\dim_R K > c_0$. If $E$
is Borel then the converse holds.
\end{theorem}

Combining this with Propositions \ref{propositiondiffusedimR} and \ref{thm:potential_implies_absolute} yields the following corollary: %\comdzmitry{The condition $E\cap K\neq\emptyset$ for every diffuse set $K$ does not (straightforwardly) imply that $E\cap K\neq \emptyset$ for every set with $\dim_RK>0$. I added Remark 4.4 to fix this issue.} \comdavid{I think you are overcomplicating things. I deleted the reference to Remark 4.4 and instead modified the statement of Proposition \ref{propositiondiffusedimR} to include a converse.}

\begin{corol}
If $E\sub X$ is Borel and $E\cap K\neq\emptyset$ for any diffuse set $K$, then $E$ is absolute winning.
\end{corol}

\begin{rem}
One can
%easily
check that Theorem~\ref{thm:intersection_with_fractal_rn} follows from Theorem~\ref{thm:potential_ahlfors} together with Theorem~\ref{potentialequiv}. Indeed, by Theorem~\ref{potentialequiv} the set $E\in \RR^N$ if $\eps$ Cantor winning if and only if it is $c_0$ potential winning for $c_0 = N(1-\eps)$. Then, Theorem~\ref{thm:intersection_with_fractal_rn} is a straightforward corollary of Theorem~\ref{thm:potential_ahlfors}.
\end{rem}

We will need the following observation regarding the potential game:

\begin{lemma}\label{lem:restriction}
Let $X$ be a complete metric space, $K\sub X$ a closed subset, and let $c>0$. If $E\sub X$ is $c$ potential winning in $X$ then $E\cap K$ is $c$ potential winning in $K$.
\end{lemma}

\begin{proof}
Assume $c'>c$ and $\beta>0$ and fix a winning strategy for Alice on $E$ for the $(c',\beta)$ potential game on $X$. Define a winning strategy for Alice on $E\cap K$ for the $(c',3\beta)$ potential game on $K$ as follows: Assume $i\geq0$ and that $B_i=B(x_i,r_i)\sub K$ with $x_i\in K$ is chosen by Bob in his $i$st move. Consider the ball with the same center and radius in $X$ and apply the winning strategy of Alice on $X$ to get a collection of balls $\cala_{i+1}$. For every $A\in\cala_i$ for which $A\cap K\neq\varnothing$ fix a point $z(A)\in A\cap K$. The triangle inequality implies that $A\cap K\sub B(z(A),3\rad(A))$ for every $A\in\cala_{i+1}$.
For the game in $K$ Alice will choose the collection of balls in $K$ defined by
\[
\cala'_{i+1} = \left\{B(z(A),3\rad(A)) \sep A\in \cala_{i+1} \text{ and } A\cap K\neq\varnothing \right\}.
\]
On one hand
\[
\sum_{A\in\cala'_{i+1}} \rad(A)^{c'} \leq 3^{c'}\sum_{A\in\cala_{i+1}} \rad(A)^{c'} < 3^{c'}\left(\beta\cdot\rad(B_i)\right)^{c'} = \left(3\beta\cdot\rad(B_i)\right)^{c'},
\]
and, therefore, this move is legal. On the other hand,
\[
\bigcup_{A\in\cala_{i+1}}A\cap K\sub\bigcup_{A\in\cala'_{i+1}}A,
\]
Since $K$ is closed, applying this strategy for every $i$ guarantees that, if $\rad(B_i)\to 0$ then either
\[
\bigcap_{i=0}^\infty B_i\sub E\cap K
\]
or
\[
\bigcap_{i=0}^\infty B_i\sub \bigcup_{i=1}^\infty\bigcup_{A\in\cala'_i}A.
\]
This shows that $E\cap K$ is $(c',3\beta)$ potential winning for every $c'>c$ and $\beta>0$, hence, it is $c$ potential winning.
\end{proof}

\begin{proof}[Proof of Theorem \ref{thm:potential_ahlfors}]
%\comdavid{I rewrote the proof, anyway here are some answers to the specific questions Dzmitry raised in his email:
%\begin{itemize}
%\item The history includes all of Alice's moves, not just dummy moves.
%\item The different indices on the two different histories are intentional, and as far as I can tell there is not any redundancy here so nothing can be removed.
%\item In the previous version, $N$ was just the length of the sequence. In the definition of the recursion procedure I specified when the sequence should halt (i.e. when the cost got too big for Alice to delete), so its length was well-defined. I have changed it so that now $N$ is a predetermined number because it made it easier to explain the Ahlfors regularity.
%\item There is no need for the case $k=1$ of the recursion to be considered separately; when $k=1$, the sequence $g_1,\ldots,g_{k-1}$ is empty and so we have $g_1(B_i) = f(\emptyset)$.
%\item There were some errors in transferring from my file (which used its own macros) to this one; in particular, $\NNN$ was inadvertently replaced by $\NN$ (I have switched it back now). For me $\NNN$ always means neighborhood.
%\item I rewrote the proof to make it clearer what kind of tree construction is being used.
%\end{itemize}}

Assume that $E\sub X$ is $c_0$ potential winning, and let $K\sub X$ be a $c$ Ahlfors regular set with $c>c_0$. Lemma \ref{lem:restriction} implies that $E\cap K$ is winning for the $c_0$ potential game restricted to $K$. Therefore, by Theorem~\ref{theoremschmidtgamerelation}, $E\cap K$ contains an $(\alpha, \beta)$ winning set for certain pairs $(\alpha,\beta)$ and hence it is nonempty.

For the converse, assume $E$ is Borel, and that $E$ is not $c_0$ potential winning. By the Borel determinacy theorem~\cite[Theorem 3.1]{FLS},\footnote{In fact, Theorem~3.1 in~\cite{FLS} does not deal with potential games. However its proof can be easily modified to cover them. In the notation of that paper, it is clear from the definition of who wins the potential game that there exists a Borel set $A\sub Z\times X$ such that if $\omega\in E_\Gamma$ is a play of the game, then $\omega$ is a win for Alice if and only if either $\omega\notin Z$ or $(\omega,\iota(\omega)) \in B$. Combining with \cite[Lemma 3.3]{FLS} shows that the set of plays of the game that result in a win for Alice is Borel.} Bob has a winning strategy to make sure the $c_0$ potential game ends in $X\butnot E$. Fix this strategy and let $\beta > 0$ and $c > c_0$ be Bob's choices according to this strategy. We will show that $X\butnot E$ contains Ahlfors regular sets of dimension arbitrarily close to $c$.

Fix a small number $\gamma > 0$. Denote by $\AAA_{i+1}$ the collection of balls chosen by Alice in response to Bob's $i$th move $B_i$. We will consider strategies for Alice where she only chooses nonempty collections on some turns, namely $\AAA_{i+1}\neq \emptyset$ only if
\[
\rad(B_i) \leq \gamma^n \rad(B_0) < \rad(B_{i-1}) \text{ for some } n\in\NN.
\]
Such turns $i = i(n)$ will be called \emph{good turns}. If $B_0,\ldots,B_{i(n)}$ vs. $\AAA_1,\ldots,\AAA_{i(n)+1}$ is a history of the game, then we let $f(\AAA_{i(n)+1})$ denote the next ball that Bob plays on a good turn if the game continues with Bob playing his winning strategy and Alice playing dummy moves (i.e. choosing the empty collection), i.e. $f(\AAA_{i(n)+1}):= B_{i(n+1)}$.
Next, we let
\[
N = \left\lfloor \left(\frac{\beta^2}{3\gamma}\right)^c \right\rfloor.
\]
If $B_0,\ldots,B_{i(n)}$ vs. $\AAA_1,\ldots,\AAA_{i(n)}$ is a history of the game, then we define a sequence of balls $\left( g_k(B_{i(n)}) \right)_{k = 1}^N$ by recursion: we let $g_1(B_{i(n)}):= f(\emptyset)$, and if $g_1,\ldots,g_{k-1}$ are defined for some $k\geq 1$, then we let
\[
g_k( B_{i(n)} ) = f\big(\big\{\NNN(g_1(B_{i(n)}),\gamma^{n+1} \diam (B_0)),\ldots,\NNN(g_{k - 1}( B_{i(n)} ),\gamma^{n+1} \diam (B_0))\big\}\big).
\]
Here $\NNN(A,d)$ denotes the $d$-neighborhood of a set $A$, i.e. $\NNN(A,d):=\{x\in X\;:\; \vd(A,x)\le d\}$.
We will show later (using the definition of $N$) that it is legal for Alice to play the collection appearing in the right-hand side, hence the right-hand side is well-defined. Note that since $B_{i(n)}$ is the $n$th good turn in the history $B_0,\ldots,B_{i(n)}$ vs. $\AAA_1,\ldots,\AAA_{i(n)}$, for each $k=1,\ldots,N$, $g_k(B_{i(n)})$ is the $(n+1)$st good turn in its corresponding history. We call these balls the \emph{children} of $B_{i(n)}$. The children of $B_{i(n)}$ are disjoint, because if $g_j(B_{i(n)}) \cap g_k(B_{i(n)}) \neq \emptyset$ for some $j,k$ with $j < k$, then $g_k(B_{i(n)}) \sub \NNN(g_j(B_{i(n)}),\gamma^{n+1} \diam(B_0))$, and thus Alice wins by default in the game in which Bob played his winning strategy ending in move $g_k(B_{i(n)})$, a contradiction.

Now construct a Cantor set $\KKK$ as follows: let $\BBB_0 = \{B_0\}$, and if $\BBB_n$ is defined, then let $\BBB_{n+1}$ be the collection of children of elements of $\BBB_n$ according to the above construction. Note that $\BBB_n$ consists of moves for Bob that are the $n$th good turns with respect to their corresponding histories, and in particular $\beta\gamma^n\rad(B_0) < \rad(B) \leq \gamma^n \rad(B_0)$ for all $B\in \BBB_n$. Finally, let $\KKK = \bigcap_{n\in\NN} \bigcup_{B\in\BBB_n} B$. Any element of $\KKK$ is the outcome of some game in which Bob played his winning strategy, so $\KKK\subset X\butnot E$. Moreover, from the construction of $\KKK$ it is clear that $\KKK$ is Ahlfors regular of dimension $\frac{\log N}{-\log \gamma}$. From the definition of $N$, we see that this dimension tends to $c$ as $\gamma\to 0$.

To complete the proof, we need to show that it is legal for Alice to play the collection
\[
\big\{\NNN(g_1(B_{i(n)}),\gamma^{n+1} \diam (B_0)),\ldots,\NNN(g_{k - 1}( B_{i(n)} ),\gamma^{n+1} \diam (B_0))\big\}
\]
for any $k\leq N$. Indeed, the cost of this collection is
\begin{align*}
\sum_{j=1}^{k-1} \big(\rad(g_j(B_{i(n)})) + \gamma^{n+1} \diam(B_j)\big)^c
&\leq N (3\gamma^{n+1} \rad(B_0))^c \\
\leq (\beta^2 \gamma^n \rad(B_0))^c &\leq (\beta \cdot \rad(B_{i(n)}))^c.
\qedhere\end{align*}
\end{proof}

%\comdavid{I deleted the remark that was here, since I didn't understand it, and as I mentioned above the remark isn't necessary.}
%\begin{rem}\label{rem44}
%One can check that if $\gamma$ is small enough so that $N\ge 2$, the set $\KKK$ from the proof is also $\gamma$ diffuse. Therefore the converse of Theorem~\ref{thm:potential_ahlfors} can be made slightly stronger: $E\cap K\neq \emptyset$ for every closed diffuse set $K\subset X$ with $\dim_R K>c_0$.
%\end{rem}

\begin{rem}
A similar characterisation is possible in the generality of $\calh$ potential games.%, cf. \cite[Appendix C]{FSU4} for the definition and for examples.
\end{rem}

%\comerez{DOES THE REMARK REMAIN VALID IN THIS VAGUE FORM AS WELL? (It used to mention hyperplanes only)} \comdavid{Sure, it is fine}

In fact, potential winning sets not only have nonempty intersection with Ahlfors regular sets, but the intersection has full Hausdorff dimension. To prove this, let us first prove the following:

\begin{proposition}\label{propositionfirstturn}
    If $\HD(E) \leq c_0$, then $X\butnot E$ is
%$c_0$-points potential
$c_0$ potential winning.
\end{proposition}
\begin{proof}
    Fix $\beta > 0$, $c > c_0$, and
%$\rho_0 > 0$.
$r_0 > 0$. Since the $c$ dimensional Hausdorff measure of $E$ is zero,
%$\HHH^c(E) = 0$,
there exists a cover $\CCC$ of $E$ such that
    \[
    \sum_{C\in\calc} \diam^c(C) \leq (\beta r_0)^c.
    \]
    Thus if Bob's first ball has radius
%$\rho_0$
$r_0$, then Alice can legally remove the collection $\CCC$, thus deleting the entire set $E$ on her first turn.
\end{proof}

\begin{theorem}\label{propositionfullHD}
    If $c_0 < \delta=\dim_R X$, then every
%$c_0$-points potential
$c_0$ potential winning set in $X$ has Hausdorff dimension $\delta$.
\end{theorem}

\begin{proof}%[Proof of Theorem \ref{propositionfullHD}]
    By contradiction, suppose that $E\sub X$ is
%$c_0$-points potential
$c_0$ potential winning but $c_1 := \HD(E) < \delta$. Then by Proposition \ref{propositionfirstturn}, $X\butnot E$ is
%$c_1$-points potential
$c_1$ potential winning, so by the intersection property~(W2) (see Proposition~\ref{w2_potential}), the empty set $\emptyset$ is
%$c_2$-points potential
$c_2$ potential winning, where $c_2 = \max(c_0,c_1) < \delta$. But then by Theorem \ref{theoremschmidtgamerelation}, the empty set $\emptyset$ is $(\alpha,\beta)$ winning for some $\alpha,\beta > 0$, which is a contradiction.
\end{proof}

Theorem \ref{propositionfullHD} may be seen as a quantitative version of the full dimension intersection property of absolute winning sets \cite[Corollary 5.4]{BFKRW}.
%\comerez{IS THIS THE SIMPLEST PROOF?}

\begin{rem}
 In view of Theorem~\ref{potentialequiv}, if a complete doubling metric space
 admits a splitting structure $(X,\SSS,U,f)$ such that $B=A_\infty(B)$ for
 every ball $B$, then the notions of Cantor winning and potential winning sets
 are equivalent. However, not every metric space $X$ admits such a
 splitting structure. Indeed, spaces with splitting structures satisfying (S4)
 must have dimensions of the form $\frac{\log n}{\log m}$ ($m,n\in\NN$)
 (cf. \cite[Theorem 3]{BadziahinHarrap}), whereas one can construct a complete
 doubling metric space of arbitrary Hausdorff dimension. In that sense, the
 notion of potential winning sets is strictly more general than the notion of Cantor winning
 sets.
\end{rem}

\begin{rem}\label{rem:emptysetIsWinning}
In the above proof, we needed to be careful due to the fact that
%$c_0$-points potential
$c_0$ potential winning sets are not automatically nonempty (as the winning sets for Schmidt's game are); indeed, by Proposition \ref{propositionfirstturn}, the empty set is
%$\delta$-points potential
$\HD(X)$ potential winning.
\end{rem}

Lastly, the same observation that is used to prove Proposition \ref{propositionfirstturn} may be used to prove Theorem \ref{thm:dolgopyat}:

\begin{proof}[Proof of Theorem \ref{thm:dolgopyat}]
Recall that every closed ball in $X$ of radius $\leq 1$ is of the form
\[
B(x,r)=\{\left(y_k\right)_{k=1}^\infty \sep y_k=x_k \textrm{ for any } 1\leq k\leq n \}
\]
whenever $x\in X$ and $2^{-n}\leq r < 2^{-n+1}$.
Note that $X$ is 1 Ahlfors regular so it is enough to prove that the exceptional set defined in \eqref{eq:exceptionalSet}, i.e.
\[
    E=\left\{x\in X\sep \overline{\{T^i x\sep i\geq0\}}\cap K=\emptyset \right\},
\]
is $\dim_H K$ potential winning. We now define a winning strategy for Alice. Assume Bob choose $0<\beta<1$ and $c > \dim_H K$. Without loss of generality assume that $B_0=X$ and $r_0=1$. Then there exists a unique positive integer $\ell$ such that
\[
2^{-\ell}\leq \beta < 2^{-\ell+1}.
\]
Since $\dim_H K < c$, we may choose a collection of open balls $\calc$ such that $K\sub\bigcup_{C\in\calc}C$ and
\begin{equation}\label{eq:dolgopyatCover}
\sum_{C\in\calc}\rad(C)^c < \ell^{-1} (2^{-2\ell}\beta)^c.
\end{equation}
Let $\{B_i\}_{i=0}^\infty$ denote any sequence of balls that are chosen by Bob which satisfy $2^{-(i+2)\ell}\leq \rad(B_i) < 2^{-(i+1)\ell}$. If this sequence is not infinite then Alice wins by default. By abuse of notation, we will think of $B_i$ as Bob's $i$th choice, and it will be sufficient to define Alice's reaction to these moves. On her $(i+1)$st move, Alice will remove the collection %\comdavid{I added a union over $j$ here, since otherwise the displayed formula in the next paragraph is false. Correspondingly changed \eqref{eq:dolgopyatCover} and some other calculations.}
\begin{equation}\label{eq:DolgopyatStrategy}
\cala_{i+1}= \bigcup_{j=0}^{\ell-1} \left\{C \sep C\in T^{-(i\ell+j)}\calc \textrm{ and } C\cap B_i \neq \emptyset\right\}.
\end{equation}
Here, $T^{-j} \calc$ denotes the set of all images of elements of $\calc$ under the inverse branches of $T^j$. %\comdavid{The previous version was not correct, because the union of a collection is not the same thing as that collection. I also don't like the idea of redefining the notation $T^{-j} C$, which by default refers to the preimage of $C$ under $T^j$. Since it seems people thought the statement should be ``unpacked'' a little more I changed ``union of the images of $\calc$'' to ``set of all images of elements of $\calc$''. Please make further changes if things are not clear.}

Note that if Alice still doesn't win by default, i.e. if \eqref{eq:potentialDefault} is not satisfied, then necessarily
\[
\bigcap_{i=0}^\infty B_i\sub \left\{x\in X\sep \{T^i x\sep i\geq0\}\sub X\setminus \bigcup_{C\in\calc}C \right\} \sub E.
\]
Therefore, we are done once we show that the collection (\ref{eq:DolgopyatStrategy}) is a legal move for Alice, i.e., that it satisfies \eqref{eq:potentialLegal}. Firstly, for any $C\in\calc$ and $j = 0,\ldots,\ell-1$, we have that $T^{-(i\ell+j)}(\{C\})$ is a collection of balls of radius $2^{-(i\ell+j)}\rad(C)$, pairwise separated by distances of at least $2^{-(i\ell+j)} \geq 2^{-i\ell - \ell}$. Since $\rad(B_i) < 2^{-i\ell-\ell}$, at most one of them intersects $B_i$. Therefore,
\[
\sum_{C\in\cala_{i+1}} \rad\left(C\right)^c \leq \sum_{j=0}^{\ell-1} \sum_{C\in\calc}\left(2^{-(i\ell+j)}\cdot\rad(C)\right)^c \leq \ell 2^{2c\ell} \sum_{C\in\calc}\rad(C)^c\cdot \rad(B_i)^c,
\]
which is smaller than $\left(\beta\cdot\rad(B_i)\right)^c$ by \eqref{eq:dolgopyatCover}.
\end{proof}

\section{Counterexamples}\label{counters}

The original papers of Schmidt \cite[Theorem 5]{Schmidt1} and McMullen \cite{McMullen_absolute_winning} provide examples of sets that are winning (respectively, strong winning) but not absolute winning.
%
%For a given integer $g\geq3$, consider the set $$S_{g}=\left\{ x\in\left[0,1\right]\sep x=\sum_{i=1}^{\infty}\frac{x_{i}}{g^{i}},\; x_{i}=0\mbox{ i.o.}\right\}. $$
%Schmidt proved that the \textit{winning dimension} (that is, the supremum of the $\alpha$ for which the set is $\alpha$ winning) is achieved and equals
%\[
%\frac{1}{\left(g-1\right)^{2}+1}.
%\]
%On the other hand, it is easy to see that $S_{g}$ contains $\left(\left[0,1\right],g^{\ell},\left(g-1\right)^{\ell}\right)$ Cantor
%sets. Therefore, ignoring edge effects, $S_{g}$ is $\left(1-\frac{\log\left(g-1\right)}{\log g}\right)$ Cantor
%winning. This shows that if we had a general implication that $\eps$ Cantor
%winning implies $\alpha(\eps)$ winning, then certainly as $\eps$ tends to zero, so does $\alpha(\eps)$.
%
%On the other hand, let $x\in\left[0,1\right]$,
%denote by $Z\left(x\right)$ the set of indices for which the binomial
%coefficient is zero. Define:
%\[
%S=\left\{ x\in\left[0,1\right]: Z\left(x\right)\mbox{ contains an almost periodic sequence}\right\} .
%\]
%McMullen proved that $S$ is $\frac{1}{8}$ winning but not absolute
%winning (not even strong winning). In the context of Theorem~\ref{thmwinningmainN} this also provides an example of a set which is $\alpha$ winning for some $\alpha< 1/2$ but not $1$Cantor winning.
%Moreover, this set is unlikely
%to be $\eps$ Cantor winning for any $0<\eps<1$. If this were true,
%then one could not expect a general statement of the form $\alpha$ winning
%implies $\eps$ Cantor winning when $\alpha<\frac{1}{2}$.
%
We provide two additional counterexamples.

\subsection{Cantor winning but not winning}

%\begin{proposition}
%   \label{theoremPWnotW}
%   There is an Cantor-winning set in $\RR$ that is not winning.
%\end{proposition}
%
%
%\begin{proposition}
%   \label{theoremWnotPW}
%   There is a winning set in $\RR$ that is not Cantor-winning.
%\end{proposition}

We give %a precise
an explicit example of a set which is potential winning in $\RR$ (or equivalently, by Theorem \ref{potentialequivR}, is Cantor winning), but is not winning. % with respect to Schmidt's original game.
In fact,  we will prove more by showing that the set in question is not weak winning (see \S\ref{subsectionstrongweak}).

%\begin{rem}
%   Theorem \ref{theoremPWnotW} remains true if ``winning'' is replaced by ``weak winning'', and we will prove it in this context.
%\end{rem}

%\begin{thm}    \label{theoremPWnotWW}
%   There is a Cantor winning set in $\RR$ that is not weak winning.
%\end{thm}

\begin{proof}[Proof of Theorem \ref{theoremPWnotW}]
By an \emph{iterated function system} on $\RR$, or \emph{IFS}, we mean a finite collection $\{f_i\;:\RR\to\RR\}$, $i = 1,2,\ldots, N,$ of contracting similarities. By the \emph{limit set} of the IFS we mean the unique compact set $S$ which is equal to the union of its images $f_i(S)$ under the elements of the IFS. We call an IFS on $\R$ \emph{rational} if its elements all preserve the set of rationals. Note that there are only countably many rational IFSes.

  Let $E \sub \RR$ be the complement of the union of the limit sets of all rational IFSes whose limit sets have Hausdorff dimension
    $\leq 1/2$. We will show that $E$ is not winning for Schmidt's game. On the other hand, since $\dim(\RR \setminus E) = 1/2 < 1$, it follows from Proposition \ref{propositionfirstturn} that $E$ is $1/2$ potential winning. %\comdavid{changed $\leq$ back to $=$, is there some reason it was changed? Obviously, some rational IFSes have Hausdorff dimension 1/2.}

    Fix $\alpha > 0$, and let $\beta > 0$ be a small number, to be determined later. We will show that Alice cannot win the weak
    $(\alpha, \beta)$ game. %\comdavid{deleted explanation of the weak $(\alpha,\beta)$ game, since it is already explained in \S\ref{subsectionstrongweak}}
    Let $I$ be a rational
    interval contained in Alice's first move, whose length is at least half of the length of Alice's first move. Without
    loss of generality, suppose that $I = [-1, 1]$. Let $\lambda < 1$ be a rational number large enough so that $\lambda + \lambda\alpha \geq 1$, and let  $\gamma > 0$ be a rational number small enough so that
    \begin{equation}
    \label{gammaalpha}\lambda^{1/2} + 2\gamma^{1/2} \leq 1.
    \end{equation}
    Consider the IFS on the interval $I$ consisting of the following three contractions:
    $$u_0(x) = \lambda x, \quad u_1(x) = \gamma(x - 1) + 1, \quad u_{-1}(x) = \gamma(x + 1) - 1,$$
and note that condition (\ref{gammaalpha}) guarantees that the dimension of the limit set of this IFS is at most $1/2$ (see , for example,~\cite[Section~5.3, Theorem~(1)]{Hutchinson}%\comdzmitry{Maybe it can be found in Falconer? I do not have it at hand to check}
). Bob's strategy is as follows: on any given turn $k$, if Alice just made the move $A_k$, then find the largest interval $J_k \sub A_k$ which can be written as the image of $I$ under a composition of elements of the IFS,
and select the unique ball $B_k$ of radius $\beta\cdot\rad(A_{k})$ whose centre is the midpoint of this interval. We will show by induction that $B_k \sub J_k$, thus proving that this choice is legal. Indeed, suppose that this holds for $k$. By ``blowing up the picture'' using inverse images of the elements of the IFS, we may without loss of generality suppose that $J_k = I$. Then we have $B_k = [-r, r]$, where $r = \rad(B_k) \leq 1$. By ``blowing up the picture'' further using an iterate of the inverse image of the contraction $u_0$, we may assume without loss of generality that $r > \lambda$. Then Alice's next choice $A_{k+1}$ is an interval of length at least $2\lambda\alpha$ contained in $[-1, 1]$. But the set $\left\{\pm \lambda^q: \, q = 1, 2, \ldots, N \right\}$ intersects every interval of length $1 - \lambda$ contained
in $[-1, 1]$, where $N$ is a large constant. Since $\lambda \alpha \geq 1-\lambda$ by assumption, there exist $q \in \{1, \ldots, N\}$ and $\epsilon\in\{-1,1\}$ such that $\epsilon \lambda^q \in \frac12 A_{k+1}$. But then $J = u^q_0u_{\epsilon}(I)$ is an interval coming from a cylinder in the IFS construction which intersects $\frac12 A_{k+1}$, and whose length is $2\gamma\lambda^q \in [2\lambda^N\gamma, 2\lambda\gamma ]$. By choosing $\gamma \leq \frac14\alpha$, we can guarantee that the length of $J$ satisfies $2\rad(J) \leq \frac12\lambda\alpha \leq \frac12 \rad(A_{k+1})$. Combining with the fact that $J\cap \frac12 A_{k+1} \neq \emptyset$ shows that $J\sub A_{k+1}$. So then by the definition of $J_{k+1}$, we have $2\rad(J_{k+1}) \geq 2\rad(J) \geq 2\lambda^N\gamma$, and thus by letting $\beta = \lambda^N\gamma$ we get $B_{k+1} \sub J_{k+1}$, completing the induction step.
\end{proof}

\subsection{Winning but not Cantor winning}
%\comstephen{Some notation made consistent with the rest of the paper through this section, in particular, things like defining with ":=" not a stackrel}

Finally, we will show that there exists a set in $\RR$ which is winning, but is not potential winning (or equivalently, by Theorem \ref{potentialequivR}, Cantor winning). %\comdavid{many changes to this section; in particular, changed ``quasi-arithmetic progression'' to ``Bohr set'' since that is what the proof really needs}

\begin{rem}
The above statement remains true if ``winning'' is replaced by ``weak winning'', but the resulting statement is weaker and so we prove the original (stronger) statement instead. We will however need the weak $(\alpha,\beta)$ game in the proof. It appears that this is a nontrivial application of the weak game.
\end{rem}
%\comerez{CAN WE MAKE PRECISE WHAT WE MEAN BY A WEAKER STATEMENT IN THE REMARK?}
%\comerez{PLEASE REVISE THE SRUCTURE OF THE NEXT PROOF}
By Corollary \ref{theoremPWWintersection}, it suffices to show that there exists a winning set $E$ whose complement is weak winning. Note that this also shows that the weak game does not have the intersection property (W2), since $E$ and its complement are both weak winning, but their intersection is empty, and the empty set is not weak winning. We will use the result below as a substitute for the intersection property.

    \begin{defn}
        A set $E$ is \emph{finitely weak winning} if Alice can play the weak game so as to guarantee that after finitely many moves her ball $A_n$ is a subset of $E$.
    \end{defn}
    \begin{lemma}\label{lem1}
The intersection of countably many finitely weak $\alpha$ winning sets is weak $\alpha$ winning.
    \end{lemma}
    \begin{proof}
     Let $(E_n)_{n\in\NN}$ be a countable collection of finitely weak $\alpha$ winning sets. Alice can use the following strategy to ensure that $\bigcap_{m=0}^\infty B_m$ lies inside $\bigcap_{n=0}^\infty E_n$. She starts by following the strategy to ensure that $A_{m_1} \subset E_1$. After that she starts playing the weak winning game for the set $E_2$, assuming that Bob's initial move is $B_{m_1}$. Therefore she can ensure that $A_{m_1+m_2} \subset E_1\cap E_2$. Then she switches her strategy to $E_3$ and so on.

Finally, we have that  for any $N$
$$
\bigcap_{m=1}^\infty A_m \subset A_{\sum_{k=1}^N m_k}\subset \bigcap_{n=1}^N E_n.
$$
By letting $N$ tend to infinity we prove the lemma.
    \end{proof}
In what follows, the intersection of countably many finitely weak $\alpha$-winning sets will be called $\sigma$ finitely weak $\alpha$-winning. %\comdavid{reverted a change here, since I don't think definitions should appear in the statements of lemmas}
    \begin{proof}[Proof of Theorem \ref{theoremWnotPW}]
    Throughout this proof we let $Z = \{0,1\}$, and we let $\pi$ denote the coding map for the binary expansion, so that $\pi:Z^\NN\to [0,1]$.

    \begin{defn}
        A \emph{Bohr set} is a set of the form
        \[
        A(\gamma,\delta):= \{n \in \NN : n\gamma \in (-\delta,\delta) \text{ mod 1}\},
        \]
        where $\gamma,\delta > 0$.

        For each $\omega = (\omega_n)_{n\in\NN}\in Z^\NN$, we let $E(\omega) = \{n\in\NN : \omega_n = 1\}$.
    \end{defn}

    We now define the set $E \sub \RR$:
    \begin{align*}
        E_0: &= \ZZ + \{\pi(\omega) : \omega\in Z^\NN,\; E(\omega)\text{ contains a Bohr set}\}\\
        E: &= \bigcup_{r\in \QQ} r E_0.
    \end{align*}
    We claim that $E$ is $1/3$ winning, but that its complement is weak $1/3$ winning.

    Fix $1/4<\alpha < 1/3$ and $\beta,\rho_0 > 0$. We show that Alice has a strategy to win the $(\alpha,\beta)$ Schmidt game assuming that Bob has chosen a ball of radius $\rho_0$ with his first move. Choose $0 < r\in \QQ$ so that $3 < 2\rho_0/r < \alpha^{-1}$, and consider the set
    \[
    A:= \{n\in\ZZ : \exists m = m_n\in\NN \text{ such that }(\alpha\beta)^m \rho_0/r \in (3\cdot 2^{-(n + 1)},\alpha^{-1}\cdot 2^{-(n + 1)})\}.
    \]
The bounds on $\alpha$ guarantee that the values $m_n$ for all $n\in A$ are distinct. By taking logarithms and rearranging the expression for the set $A$, we get
\begin{align*}
n\in A \quad&\Leftrightarrow\quad \exists m \;\; \log 3 <(n+1)\log 2 + m\log(\alpha\beta) + \log\left(\rho_0/r\right) < \log \left(\alpha^{-1}\right)\\
&\Leftrightarrow\quad n \frac{\log 2}{-\log(\alpha\beta)} \in \left( \frac{\log \left(\frac{3}{2\rho_0/r}\right)}{-\log(\alpha\beta)} , \frac{\log \left(\frac{\alpha^{-1}}{2\rho_0/r}\right)}{-\log(\alpha\beta)} \right) \text{ mod 1}.
\end{align*}
Since the interval on the right-hand side contains 0, this implies that $A$ contains a Bohr set. %\comdavid{I rewrote the above formula because I'm not sure what it means for an inequality to be true mod 1. Being in an interval mod 1 makes sense because it just means that some representative is in the interval.}

Now fix $n\in A$, and let $m = m_n \in\NN$ be the corresponding value. On the $m$th turn, Bob's interval $B_m$ is of length $2(\alpha\beta)^m \rho_0 \geq 3\cdot 2^{-n} r$, so if we subdivide $\RR$ into intervals of length $2^{-n} r$, then $B_m$ must contain at least two of them. Since Alice's next interval is of length $2\alpha (\alpha\beta)^m \rho_0 \leq 2^{-n} r$, this means that Alice can choose a subinterval of either one of these intervals on her turn. So Alice can control the $n$th binary digit of $x/r$, where $x$ is the outcome of the game, and in particular she can set the $n$th digit to $1$. Since every value $m$ corresponds to only one value $n\in A$, Alice can set all of the digits of $x/r$ whose indices are in $A$ to $1$. This means that $x/r\in E_0$ and thus $x\in E$, so Alice can force the outcome of the game to lie in $E$. So by Proposition~\ref{lemmacontinuity}, $E$ is $1/3$ winning.

    Now if we show that $\RR\butnot E_0$ is $\sigma$ finitely weak $1/3$ winning, then by symmetry $\RR\butnot rE_0$ is $\sigma$ finitely weak $1/3$ winning for all $r > 0$, and since the intersection of countably many $\sigma$ finitely weak $1/3$ winning sets is $\sigma$ finitely weak $1/3$ winning, it follows that $\RR\butnot E$ is $\sigma$ finitely weak $1/3$ winning. By symmetry it suffices to consider the case $r=1$. To prove that $\RR\butnot E_0$ is indeed $\sigma$ finitely weak $1/3$ winning,  we will use the following lemma: %\comdavid{reworded this paragraph}
    \begin{lemma}\label{lem2}
        Fix $0 < \alpha < 1/3$ and $\beta > 0$. Then there exists $N\in\NN$ such that for any $\rho_0 > 0$ there exists $a_1(\rho_0)$ so that for any sequence of integers $(a_k)_{1\le k\le M}$ with
        \begin{equation}
            \label{aksep}
            a_{k + 1} - a_k \geq N \;\; \forall k,\quad \alpha_1\ge \alpha_1(\rho_0),
        \end{equation}
        Alice has a strategy in the finite weak $(\alpha,\beta)$ game to control the digits indexed by the integers $(a_k)_{1\le k\le M}$.
    \end{lemma}
    %\comstephen{This statement is a bit of a mouthful. Also, is there potential confusion with unqualified $\rho_0$ and $\rho$ in the statement? $\rho_0$ Bob's first radius, $\rho$ Bob's radius on an arbitrary turn in the game?} \comdavid{Not sure why $\rho$ and $\rho_0$ were different, I have fixed it}
    \begin{proof}[Proof of Lemma \ref{lem2}]
        Suppose that Bob has just played a move of radius $\rho$. Then Alice can play any radius between $\alpha\rho$ and $\rho$, forcing Bob to play a predetermined radius between $\alpha\beta\rho$ and $\beta\rho$. After iterating $n$ times, Alice can force Bob to play a ball of any given radius between $\alpha^n \beta^n \rho$ and $\beta^n\rho$. Since $\alpha^n < \beta$ for sufficiently large $n$, there exists $\gamma > 0$ such that Alice can force Bob to play a ball of any prescribed radius $\leq \gamma \rho$ (after a sufficient number of turns depending on the radius). Now Alice uses the following strategy: for each $k = 1,\ldots,M$, first force Bob to play a ball whose radius is $3 \cdot 2^{-a_k}$, and then respond to Bob's choice in a way such that the $a_k$th digit of the outcome is guaranteed (this is possible by the argument from the previous proof). Forcing Bob to play these radii is possible as long as $3\cdot 2^{-a_1} \leq \gamma \rho_0$ and $3\cdot 2^{-a_{k+1}} \leq \gamma \alpha 3 \cdot 2^{-a_k}$ for all $k$. By \eqref{aksep}, the latter statement is true as long as $N$ and $a_1(\rho_0)$ are sufficiently large.
\end{proof}

    To continue the proof of Theorem \ref{theoremWnotPW}, fix $0 < \alpha < 1/3$ and $\beta > 0$, and let $N\in\NN$ be as in Lemma~\ref{lem2}. Then:
    \begin{itemize}
        \item For all $k\in\NN$, the set
        \[
        A_k:= \ZZ + \{\pi(\omega) : T(\omega) \text{ contains an arithmetic progression of length $k$ and gap size $N$}\}
        \]
        is finitely weak winning, where $T(\omega) := \NN\butnot E(\omega) = \{n\in\NN: \omega_n=0\}$.
        \item For all $q,i\in\NN$, the set
        \[
        A_{q,i} := \ZZ + \{\pi(\omega) : T(\omega) \cap (q\NN + i) \neq \emptyset\}
        \]
        is finitely weak winning.
    \end{itemize}
    So the intersection
    \[
    A = \bigcap_k A_k \cap \bigcap_{q,i} A_{q,i}
    \]
    is $\sigma$ finitely weak winning. To finish the proof, we need to show that $\RR\butnot E_0 \subp A$. Indeed, fix $n + \pi(\omega) \in A$. Then $T(\omega)$ contains arbitrarily long arithmetic progressions of gap size $N$, and intersects every infinite arithmetic progression nontrivially.

    Now let $A(\gamma,\delta)$ be a Bohr set, and we will show that $A(\gamma,\delta)$ intersects $T(\omega)$ nontrivially. First suppose that $\gamma$ is irrational. Then so is $N\gamma$, so by the minimality of irrational rotations, there exists $M$ such that the finite sequence $0,N\gamma,\ldots,MN\gamma$ is $\delta$-dense in $\RR/\ZZ$. Now let $n,n+N,\ldots,n+MN$ be an arithmetic progression of length $M$ and gap size $N$ contained in $T(\omega)$. Then $n\gamma,n\gamma + N\gamma,\ldots,n\gamma + MN\gamma$ is also $\delta$-dense in $\RR/\ZZ$, and in particular must contain an element of $(-\delta,\delta)$. So $(n+iN)\gamma \in (-\delta,\delta)$ mod 1 for some $i = 0,\ldots,M$. But then $n+iN \in A(\gamma,\delta) \cap T(\omega)$.

    On the other hand, suppose that $\gamma$ is rational, say $\gamma = p/q$. Then $A(\gamma,\delta)$ contains the infinite arithmetic progression $q\NN$, which by assumption intersects $T(\omega)$ nontrivially.

    Thus every Bohr set intersects $T(\omega)$ nontrivially, so $E(\omega)$ does not contain any Bohr set, i.e. $n + \pi(\omega) \notin E_0$. So $A \sub \RR\butnot E_0$, completing the proof.
\end{proof}

\newpage

\section*{Acknowledgements}

EN and DS were supported by EPSRC Programme Grant: EP/J018260/1. EN wishes to thank Dmitry Dolgopyat for suggesting to use games to prove \cite[Theorem 1]{Dolgopyat} during the 2016 ``Dimension and Dynamics'' semester at ICERM.
\\

\bibliographystyle{amsplain}

\providecommand{\bysame}{\leavevmode\hbox to3em{\hrulefill}\thinspace}
\providecommand{\MR}{\relax\ifhmode\unskip\space\fi MR }
\providecommand{\MRhref}[2]{%
  \href{http://www.ams.org/mathscinet-getitem?mr=#1}{#2}
}
\providecommand{\href}[2]{#2}

\bibliography{bibliography}

\vspace{8mm}

\begin{minipage}{0.4\textwidth}
\textbf{Dzmitry Badiziahin}\\
School of Mathematics and Statistics F07\\
University of Sydney NSW 2006\\
Australia\\
\url{dzmitry.badziahin@sydney.edu.au}\\
\end{minipage}
\hspace{8ex}
\begin{minipage}{0.4\textwidth}
\textbf{Stephen Harrap}\\
Department of Mathematical Sciences\\
Durham University\\
Lower Mountjoy\\
Stockton Road\\
Durham DH1 3LE\\
United Kingdom\\
\url{s.g.harrap@durham.ac.uk}\\
\end{minipage}

\vspace{3mm}

\begin{minipage}{0.4\textwidth}
\textbf{David Simmons}\\
Department of Mathematics\\
University of York\\
Heslington\\
York YO10 5DD\\
United Kingdom\\
\url{david9550@gmail.com}\\
\end{minipage}
\hspace{8ex}
\begin{minipage}{0.4\textwidth}
\textbf{Erez Nesharim}\\
Department of Mathematics\\
University of York\\
Heslington\\
York YO10 5DD\\
United Kingdom\\
\url{erez.nesharim@york.ac.uk}\\
\end{minipage}

\end{document}